\documentclass[11pt,reqno,centertags,noamsfonts]{amsart} 
\usepackage[margin=30.4mm]{geometry}

\usepackage[latin2]{inputenc}
\usepackage{amsmath}
\usepackage{amssymb}
\usepackage{amsthm}
\usepackage{graphicx}
\usepackage{tikz}
\usepackage{esint}
\usepackage{xcolor}
\usepackage{floatrow}
\usepackage{color}
\usepackage{epsfig}

\usepackage[UKenglish]{babel}
\usepackage[colorlinks=true,linkcolor=blue,anchorcolor=black,citecolor=blue!50!black,filecolor=black,menucolor=black,runcolor=black,urlcolor=blue]{hyperref}

\usepackage{mathtools}
\mathtoolsset{showonlyrefs}
\usepackage{enumitem}
\usepackage{csquotes}
\usepackage{mathrsfs}
\newcommand{\scrp}{\mathscr{P}}
\newcommand{\sprmA}{\delta} 

\usepackage{physics} 

\newtheorem{theorem}{Theorem}[section]

\newtheorem{proposition}[theorem]{Proposition}
\newtheorem{lemma}[theorem]{Lemma}
\newtheorem{corollary}[theorem]{Corollary}
\newtheorem{definition}[theorem]{Definition}
\newtheorem{remark}[theorem]{Remark}

\newtheorem*{theorem*}{Theorem}

\allowdisplaybreaks[2]

\newcommand{\supp}{\operatorname{supp}} 
\newcommand{\divv}{\operatorname{div}}

\newcommand{\diag}{\operatorname{diag}}

\newcommand{\dist}{\operatorname{dist}^*} 
\newcommand{\Dist}{\operatorname{Dist}^*} 
 
\newcommand{\oDist}{\operatorname{Dist}}

\newcommand{\rel}{\mathrm{rel}}
\newcommand{\loc}{\mathrm{loc}}
\newcommand{\Om}{\Omega}
\newcommand{\ve}{\varepsilon}
\newcommand{\En}{G}
\newcommand{\en}{g}
\newcommand{\mOns}{\varpi}

\newcommand{\zin}{z^\mathrm{in}} 
\newcommand{\uin}{u^\mathrm{in}}
\newcommand{\cin}{c^\mathrm{in}}

\newcommand{\Mzero}{\hyperlink{Mzero}{\bf(M0)}}
\newcommand{\Mone}{\hyperlink{Mone}{\bf(M1)}}

\newcommand{\secentropytools}{Section~2}
\newcommand{\fluxbound}{Lemma~2.3}

\newcommand{\EPestimate}{Lemma~2.1}
\newcommand{\thmWeak}{Theorem~1.2}
\newcommand{\thmRen}{Theorem~1.8}
\newcommand{\ustrong}{Lemma~6.1}
\newcommand{\lowerBdS}{(2.9)}
\newcommand{\iuu}{p_{ij}}
\newcommand{\iuv}{p_{0j}}
\newcommand{\iru}{q_{ij}}

\newcommand{\ol}{\overline}

\numberwithin{equation}{section}

\makeatletter
\@namedef{subjclassname@2020}{\textup{2020} Mathematics Subject Classification}
\makeatother

\begin{document}

\title[Weak-strong uniqueness for ERDS]{Weak-strong uniqueness for\\ energy-reaction-diffusion systems} 

\begin{abstract}
	We establish weak-strong uniqueness and stability properties of 
	renormalised solutions to a class of energy-reaction-diffusion systems. 
	The systems considered are motivated by thermodynamically consistent models, and their formal entropy structure allows us to use as a key tool a suitably adjusted relative entropy method. 
	The weak-strong uniqueness principle holds for \textit{dissipative} renormalised solutions, which in addition to the renormalised formulation obey suitable dissipation inequalities consistent with previous existence results.
We treat general entropy-dissipating reactions without growth restrictions, and certain models with a non-integrable diffusive flux.
	
	 The results also apply
		to a class of (isoenergetic) reaction-cross-diffusion systems.
\end{abstract}

\author[K. Hopf]{Katharina Hopf}
\address{Katharina Hopf, Weierstrass Institute for Applied Analysis and Stochastics (WIAS),
	Mohrenstrasse 39, 10117 Berlin, Germany, 
	E-Mail: hopf@wias-berlin.de}



\date{} 

\maketitle

\section{Introduction}\label{sec:intro}

It is well acknowledged that the evolution of a system of diffusing and reacting chemicals is influenced by the thermal state of the system.
\emph{Energy-reaction-diffusion systems} (ERDS) take into account this thermal dependency
by consistently coupling the evolution of the chemical concentrations $c=(c_1,\dots,c_n)$ to a heat-type equation for the \emph{internal energy density} $u$. Choosing the internal energy density as the thermal variable (as opposed to temperature for instance) has the advantage that the underlying physical entropy is jointly concave in the state variables $z=(u,c)$~\cite{Mielke_generic_2011,Mielke_gradstruc_2011}.

Recently, global existence of weak and renormalised solutions has been obtained for a class of thermodynamically consistent ERDS~\cite{FHKM_2020} taking the form (with~$z:=(u,c)$)
\begin{equation}\label{eq:erds.cu}\left\{
	\begin{array}{rlll}
		\partial_t u\!\!&=\nabla\cdot\big(A_{0j}(z)\nabla z_j\big), &
		t>0, x\in\Om,&\vspace{.1cm}
		\\\partial_t c_i\!\!&=\nabla\cdot\big(A_{ij}(z)\nabla z_j\big)+R_i(z),\qquad&
		t>0, x\in\Om,  \quad& i\in\{1,\dots,n\},\vspace{.1cm}
		\\0\!\!&=A_{ij}(z)\nabla z_j\cdot\nu, &t>0,  x\in\partial\Om, & i\in\{0,\dots,n\},
	\end{array}\right.
\end{equation}
see also the more explicit system~\eqref{eq:ex.erds}.
Eq.~\eqref{eq:erds.cu} is supplemented with an initial condition $(u,c)_{|t=0}=(\uin,\cin)$ for $x\in\Om$, where $\Om\subset\mathbb{R}^d$ is a bounded Lipschitz domain with outer unit normal $\nu$.
Note that we use the  summation convention omitting the summation symbol in repeatedly occurring indices (here $\sum_{j=0}^n$). The diffusion matrix $A(z)=(A_{ij}(z))_{i,j=0,\dots,n}$ and the reactions $(R_i(z))_{i=1,\dots,n}$ are obtained from an underlying formal gradient structure based on entropy functionals $H(z)=\int_\Om h(z)\,\dd x$ with convex densities $h(z)=h(u,c)$ taking the form
\begin{align}
	h(u,c_1,\dots,c_n) = -\sigma(u)+\sum_{i=1}^nb(c_i,w_i(u)).
\end{align}
Here, $\sigma$ denotes the thermal part when the concentrations $c_i$ are in their thermodynamic equilibrium $w_i=w_i(u)$, and $b(s,e):=e\lambda(s/e)$ with $\lambda$ denoting the  Boltzmann function, see Section~\ref{ssec:modelling} for details.
The absence of reactions in the $u$-component of~\eqref{eq:erds.cu} reflects the property of conservation of the total (internal) energy $\int_\Om u$. 

ERDS genuinely feature cross-diffusion effects, such as concentration flux driven by gradients of the internal energy density 
and energy flux due to concentration gradients, 
which are one of the main sources of difficulties in their analysis.
These phenomena are closely linked to the thermodynamic origin of ERDS, and are related to the \emph{Soret effect} and the \emph{Dufour effect}, well-known in physics, which describe concentration flux due to temperature gradients resp.\ heat flux driven by concentration gradients.

In the present manuscript, we aim to derive stability properties including a weak-strong uniqueness result for ERDS~\eqref{eq:erds.cu}  based on their thermodynamic structure.
This can be seen as an attempt to justify the weak solution concept in~\cite{FHKM_2020}.
The main contribution of~\cite{FHKM_2020} was to identify non-trivial classes of thermodynamically consistent models that allow for an existence theory of generalised solutions.
Interestingly, even in a cross-diffusion dominant regime and without physically restrictive growth conditions on the reactions, existence has been obtained in~\cite{FHKM_2020} based on the notion of renormalised solutions~\cite{Fischer_2015}. 
Our weak-strong uniqueness principle covers models involving various cross-diffusion phenomena (Soret effect, Dufour effect, cross diffusion between species),
and
applies in particular to a class of isoenergetic reaction-cross-diffusion systems, thus generalising~\cite{CJ_2019}. 
Weak-strong uniqueness is obtained from a weak stability estimate for a generalised distance involving as in~\cite{Fischer_2017} an adjusted relative entropy.
By suitably exploiting the thermodynamic structure of the system some of the technical issues arising in the proof of~\cite{CJ_2019} will be avoided. 
We also obtain an asymptotic stability result.

\subsection{Thermodynamic modelling}\label{ssec:modelling}
Let us now briefly specify the thermodynamic structure considered in the present manuscript. For more background on the modelling, we refer to~\cite{FHKM_2020,MM_2018,HHMM_2018,Mielke_2013}.
Models compatible with thermodynamics can be derived using the Onsager formalism in~\cite{Mielke_2013}. 
Here, the main ingredient is a triple $(Z,H,\mathbb{K})$ consisting of a state space $Z$, a driving  functional $H$, and a so-called Onsager operator $\mathbb{K}$.
Typically, $Z\subset X$ is a convex subset of a Banach space $X$, $H:Z\to\mathbb{R}\cup\{\infty\}$ a differentiable and convex functional on $Z$ (below usually referred to as entropy due to its correspondence to the negative of the physical entropy),  
while $\mathbb{K}$ can be seen as a generalised inverse Riemannian metric tensor on $Z$.  
More specifically, for every $z\in Z$, $\mathbb{K}(z)$ defines a symmetric and positive semi-definite  (unbounded) linear operator from $T_z^*Z$ to $T_zZ$. 
If $\mathbb{G}=\mathbb{K}^{-1}$ exists, the triple $(Z,H,\mathbb{G})$ forms a gradient system.
Then, motivated by the classical gradient flow equation $\mathbb{G}(z)\dot z=-DH(z)$, with $\dot z$ denoting the time derivative of $z=z(t)$, one considers the evolution law
$$\dot z = -\mathbb{K}(z)DH(z),$$ where here $DH$ denotes the Fr\'echet derivative of the functional $H$. An advantage of this Onsager form is that it facilitates the consistent coupling of different physical phenomena,
which can be realised by an additive decomposition of $\mathbb{K}$~\cite{Mielke_2013}.
Observe that, formally, the above structure encodes the following core entropy dissipation property
\begin{align}\label{eq:ed.a}
	\tfrac{\dd}{\dd t}H(z)=-\langle DH(z),\mathbb{K}(z)DH(z)\rangle \le0
\end{align}
along any solution curve $z=z(t)$ of the above law.
Conservation of the total energy $E(z),$ with $E$ denoting the energy functional on $Z$,  
can be guaranteed by imposing the condition $\mathbb{K}DE=\mathbb{K}^*DE\overset{!}{=}0$, which implies that
$\tfrac{\dd}{\dd t}E(z)=-\langle DE(z),\mathbb{K}(z)DH(z)\rangle = 0.$ 

In the context of ERDS, we consider, as introduced above, $z=(u,c)$ with $u$ the internal energy density and $c=(c_1,\dots,c_n)$ the vector of concentrations. We focus on entropies of the form $$H(z)=\int_\Om h(z)\,\dd x$$ with densities $$h(z)=h(u,c)=-\sigma(u)+\sum_{i=1}^nb(c_i,w_i(u)),$$
composed of a thermal part $\sigma(u)$ and a relative Boltzmann entropy $b(s,e)=e\lambda(s/e)$, where	
\begin{align}\label{eq:deflb}
	\lambda(r):= r\log(r)-r+1,
\end{align}
and with $w_i=w_i(u)$, $i=1,\dots, n$, denoting 
the thermodynamic equilibria of the concentrations $c_i$. The dependence of $w_i$ on the internal energy density $u$ results in a strong coupling of the system and is one of the main sources of difficulties in the analysis. It will be convenient to introduce the function $\hat\sigma(u)=\sigma(u)-\sum_{i=1}^nw_i(u)+n$ and rewrite $h(u,c)$ in the following more explicit form
\begin{align}\label{eq:S0}\tag{h1}
	h(u,c) 	=-\hat \sigma(u) + \sum_{i = 1}^n \Big( \lambda(c_i) - c_i \log w_i(u)\Big).
\end{align}
To simplify the exposition, we will impose the following concrete conditions on the coefficient functions (cf.~\cite{FHKM_2020}):
\begin{equation}
	\label{eq:h2}\tag{h2}\left.
	\begin{tabular}{@{}l@{}}
		$\hat\sigma \in C^2((0,\infty))$ strictly concave \& non-decreasing.
		\\	 $w_i\in C([0,\infty)) \cap C^2((0,\infty))$ concave \& non-decr.\ with $w_i(0)>0$ for all $i$. 
		\\ $\lim_{u\downarrow0}\hat\sigma'(u)=+\infty$, $\lim_{u\uparrow\infty}\hat\sigma'(u)=0$;\quad  $\sup_{u\in(0,1]}\hat\sigma''(u)<0$.
		\\	$\exists\beta\in(0,1)$ such that  $w_i(u)\lesssim 1+u^\beta$ for all $i\in\{1,\dots,n\}$.
	\end{tabular}
	\right\} 
\end{equation}
Typical choices are $\hat\sigma(u)=a\log(u)$ or $\hat\sigma(u)=au^\nu$ for some $\nu\in(0,1)$, $a>0$, and $w_i(u)=(b_{i,1}u+b_{i,0})^{\beta_i}$ or $w_i(u)=b_{i,1}u^{\beta_i}+b_{i,0}$ for $\beta_i\in(0,1),$ $b_{i,0}>0, b_{i,1}\ge0$.

\smallskip

\noindent Our weak-strong uniqueness principle does not rely, in an essential way, on this specific form of the entropy density.\footnote{See Sec.~\ref{sss:rds} for an example of a different entropy density that our technique can be adapted to.}  In fact, in the proof of our main theorem (Thm~\ref{thm:wkuniq}),  identity~\eqref{eq:S0} is only used to guarantee the coercivity properties in Proposition~\ref{prop:coerc.hrel}. The crucial point in the assumptions~\eqref{eq:h2} on the coefficient functions is that they ensure good convexity properties, and more specifically the locally uniform positive definiteness of the Hessian $D^2h$, which is essential for estimate~\eqref{eq:coercA} in Prop.~\ref{prop:coerc.hrel}. The monotonicity assumptions on $\hat\sigma$ and $w_i$ are relevant from the modelling point of view, since they ensure that $u\mapsto h(u,c)$ is non-increasing, so that temperature, which is given by $-\tfrac{1}{\partial_uh}$, is non-negative. 

As in~\cite{MM_2018,FHKM_2020} we are primarily interested in 	
Onsager operators $\mathbb{K}$ of the form
$$\mathbb{K}(z)\zeta=\mathbb{K}_{\mathrm{diff}}(z)\zeta+\mathbb{K}_{\mathrm{react}}(z)\zeta=-\nabla\cdot(\mathbb{M}(z)\nabla\zeta)+\mathbb{L}(z)\zeta,$$
where $\mathbb{M}(z),\mathbb{L}(z)\in\mathbb{R}^{(1+n)\times(1+n)}$ are  positive semi-definite symmetric matrices and where $\nabla=\nabla_x$ is the gradient with respect to $x\in\Om$. 
We will complement $\mathbb{K}$ with the no-flux boundary conditions $\mathbb{M}\nabla\zeta\cdot\nu=0$ on $\partial\Om$, where $\nu$ denotes the outer unit normal vector to $\partial\Om$.
Observing that $E(u,c)=\int_\Om u$ describes the total (internal) energy, the condition $\mathbb{K}DE\equiv0$, ensuring energy conservation, means that 
$\ker \mathbb{L}(z)\supseteq\mathrm{span}\{(1,0)^T\}$.
Thus, by the symmetry of $\mathbb{L}$, the zeroth component $R_0$ of  $R(z):=-\mathbb{L}(z)Dh(z)$ vanishes. Moreover, positive semi-definiteness of $\mathbb{L}$ implies the inequality 
\begin{align}\label{eq:R}
	D_ih(z)R_i(z)\le 0.
\end{align}
In this paper, the specific form of $\mathbb{L}(z)$ will not be relevant. Instead, we directly work with reactions $R(z)$ of the form
$$R(z):=(0,R_1(z),\dots, R_n(z))$$ 
satisfying~\eqref{eq:R}.

With $\mathbb{K}$ as above, the equation $\dot z =-\mathbb{K}(z)DH(z)$ can be written in the form~\eqref{eq:erds.cu} by choosing 
$$A(z):=\mathbb{M}(z)D^2h(z).$$  
In short, 
\begin{equation}\label{eq:erds}\tag{\textsc{erds}}
	\begin{array}{rlcl}
		\partial_t z&\!\!=\, \nabla\cdot\big(A(z)\nabla z\big)+R(z), &&\quad t>0,x\in\Om,\vspace{1mm}
		\\0&\!\!= \, A(z)\nabla z\cdot\nu,&&\quad t>0,x\in\partial\Om,
	\end{array}
\end{equation}
subject to an initial condition $z_{|t=0}=\zin$.
In the above setting, the entropy dissipation property~\eqref{eq:ed.a} takes the form
\begin{align}
	\frac{\dd}{\dd t}H(z)+\int_\Om\scrp(z)\,\dd x=\int_\Om D_ih(z)R_i(z)\,\dd x\le0,
\end{align}
where $\scrp(z):=\nabla D_ih(z)\cdot\mathbb{M}_{il}(z)\nabla D_lh(z)\ge0$, by the positive semi-definiteness of the mobility matrix $\mathbb{M}$, which will be assumed throughout.
Supposing, for instance, that $\scrp(z)\gtrsim\sum_{i=1}^n|\nabla \sqrt{c_i}|^2$ (as it can be proved for many of the models considered in~\cite{FHKM_2020}, see Section~\ref{ssec:examples} and Lemma~\ref{l:modelBounds}), 
and using conservation of $\int_\Om u$ together with suitable bounds on $H(z)$ (cf.\ Lemma~\ref{l:S}),
the entropy dissipation property provides a priori control of $\sum_{i=1}^n\|\nabla \sqrt{c_i}\|^2_{L^2_{t,x}}$.
Let us further note that  the fact that $R_0(z)$ vanishes gives to some extent a scalar-like structure to the $u$-component of~\eqref{eq:erds},
and if, for instance, $\mathbb{M}$ is chosen such that $A_{0j}=a\,\delta_{0j}$ for some function $a=a(z)\ge0$, any $L^p$-type energy $\tfrac{1}{p}\|u(t)\|_{L^p}^p$, $p\in(1,\infty)$, is formally non-increasing in time.

\subsection{Motivation and strategy}

Being able, for a given PDE, to identify concepts of solutions for which both existence and uniqueness can be established is a fundamental concern in modelling and analysis.
For \emph{scalar} equations, there are various tools to identify
frameworks allowing for the existence of a \textit{unique} solution, even in regimes of low regularity and with strong nonlinearities. 
One approach is based on the \enquote{doubling variables} technique first employed by Kru\v{z}kov~\cite{Kruzkov_1970} to entropy solutions of first-order equations, and extended by Carrillo~\cite{Carrillo_1999} to hyperbolic-parabolic-elliptic equations. The concept was adapted to situations where $L^\infty$ bounds are not available to 
give uniqueness in a class of renormalised solutions~\cite{BCW_2000,CW_1999}.
See also~\cite{AI_2012,Perthame_2002,AAO_2020} for more recent developments.
Let us also mention the Young measure approach to conservation laws going back to
Tartar~\cite{Tartar_1979} and DiPerna~\cite{DiPerna_1985}
who obtained, in the scalar case, uniqueness of solutions obeying an entropy inequality.  
For second order parabolic and elliptic equations the viscosity solution technique and associated comparison principles~\cite{Jensen_1988,CIL_1992} are powerful tools and the key to a variety of wellposedness results in geometric, highly nonlinear or degenerate settings, see e.g.~\cite{ES_1991,CGG_1991,CC_1995,CHR_2020}. 
Some extensions of the viscosity solution approach to systems are available for weakly coupled problems with a monotonicity condition~\cite{IK_1991}. Further uniqueness results applying to specific systems and typically in more regular situations include~\cite{Gajewski_1994,Juengel_2000,CJ_2018,PT_2019,GV_2019,BBEP_2020}. 

In general,  the case of \emph{strongly coupled} (parabolic) \emph{systems} tends to be much more difficult.
While under a parabolicity condition the existence of suitable generalised solutions (here referred to as~\enquote{weak} solutions) can often be established, positive uniqueness results in such general settings are rare.
It is therefore common, to relax the quest for uniqueness to the problem of whether weak solutions are uniquely determined  in situations where a sufficiently regular solution (a \enquote{strong} solution) happens to exist. In other words, one is interested in the question of whether such strong solutions are unique in a potentially much larger class of weak solutions. 
The question of weak-strong uniqueness is classical in fluid dynamics problems and goes back to Leray's fundamental work~\cite{Leray_1934}, where it was established for the incompressible Navier--Stokes equations.
We refer to the survey by Wiedemann~\cite{Wiedemann_2018} for more details and further references.
For recent advances on conditional uniqueness results for
dissipative measure-valued solutions to conservation laws, see~\cite{GKS_2020} and references therein.
Rather significant in the thermodynamics context is moreover the relative entropy technique employed  by DiPerna~\cite{DiPernaRonald_1979} and Dafermos~\cite{Dafermos_1979a} for hyperbolic conservation laws. 

Relative entropy methods are nowadays a standard tool to study weak stability properties of nonlinear systems endowed with a (convex) entropy structure.
Generally speaking (using the notation introduced in Sec.~\ref{ssec:modelling}), 
a relative entropy of the form
\begin{align}\label{eq:def.Hrel}
	H_\rel(z,\tilde z) = H(z) - \int_\Om D_ih(\tilde z)(z_i-\tilde z_i)\,\dd x - H(\tilde z)
\end{align}
is used to measure the distance between a weak solution $z$ and a strong solution $\tilde z$.
Observe that for convex entropies $H(z)$, the map $z\mapsto H_\rel(z,\tilde z)$ is a non-negative, convex functional vanishing in $z=\tilde z$.
The thermodynamic structure ensures that regular solutions automatically satisfy an entropy dissipation balance (cf.\ eq.~\eqref{eq:ed.a}). 
Physically relevant processes may, however, in general possess less regularity and here an entropy inequality is often added as an admissibility criterion for weak solutions.
The goal is then to obtain an upper bound on the time evolution of the relative entropy that implies stability of a regular flow on a finite time horizon among generalised solutions. 
This means that given a regular flow, any weak solution that is initially close (in relative entropy) will remain close for some time.

In the present paper, we pursue such a strategy in the context of ERDS.
Weak-strong uniqueness has recently been obtained for entropy-dissipating reaction-diffusion systems with a uniformly elliptic and bounded diagonal diffusion matrix~\cite{Fischer_2017}, where the main difficulty consists in a lack of control of the reaction rates. Extensions to a cross-diffusion system from population dynamics with weak cross diffusion can be found in~\cite{CJ_2019}. Both references are based on the relative entropy method,  but their arguments rely on the specific structure of the diffusion matrix of their systems. Here, we would like to present a more general strategy to deduce stability from an underlying thermodynamic structure. 

Given the strong coupling and lack of a priori bounds in $L^\infty$, there are several difficulties in our ERDS that require an adaptation of the classical relative entropy approach to weak-strong stability.
First, due to the lack of growth restrictions on the reactions and in some cases even the flux term (see the models in~\cite{FHKM_2020}), the evolution of the classical relative entropy 
used to measure the distance between a renormalised solution $z$ and a strong solution $\tilde z$, cannot be properly controlled. This is due to the term 
\begin{align}\label{eq:Hrel.class}
	-\int_\Om D_ih(\tilde z)\tfrac{\dd}{\dd t}z_i \,\dd x = 
	\int_\Om \nabla D_ih(\tilde z)\cdot (A_{ij}(z)\nabla z_j)\,\dd x 
	-\int_\Om D_ih(\tilde z)R_i(z) \,\dd x\quad
\end{align}
arising in the formal computation of the time derivative of $H_\rel(z,\tilde z)$. In fact,  the available a priori estimates do not ensure that $A_{ij}(z)\nabla z_j\in L^1(\Om)$ and $R_i(z)\in L^1(\Om)$  for a.e.\ time. At the same time, the corresponding integrands in~\eqref{eq:Hrel.class} do not have a sign, and there is no hope for the uncontrolled parts to cancel with some of the remaining terms appearing in $\tfrac{\dd}{\dd t}H_\rel(z,\tilde z)$.
It is therefore necessary to adjust the relative entropy $H_\rel(z,\tilde z)$.
This issue has already been encountered in~\cite{Fischer_2017};
it can be resolved by introducing a suitable smooth and compactly supported truncation function $\xi^*=\xi^*(z)$ with $\xi^*(z)=1$ if $\sum_{i=0}^n z_i\le E$ for some $E\gg \sum_{i=0}^n \tilde z_i$ (see Section~\ref{ssec:def.xi*} for details) in the formula for the relative entropy via
\begin{align}
	H_\rel^*(z,\tilde z) := H(z) - \int_\Om D_ih(\tilde z)(\xi^*(z)z_i-\tilde z_i)\,\dd x - H(\tilde z).
\end{align}
The relative entropy density adjusted in this fashion allows to remove the issue 
pointed out above. (Strictly speaking, in the term $D_0h(\tilde z)(z_0-\tilde z_0)$ the truncation function is not needed in the models considered in this paper, and for other applications it may be helpful to use a different choice such as $D_ih(\tilde z)(\xi^*_i(z)z_i-\tilde z_i)$ with $\xi^*_i(z)\equiv1$ for $i=0$, or versions thereof.)

A second difficulty arising in the case of ERDS is the inherent coupling between concentrations and energy density, which manifests itself in the circumstance that the entropy density cannot be additively decomposed into terms depending only on an individual component $z_i$. 
This in turn leads to a non-diagonal diffusion matrix $A(z)$ and renders estimating the evolution of  $H^*_\rel(z,\tilde z)$ substantially more delicate than in the diagonal case. 
One of the main contributions of this manuscript is to show that such estimates can be achieved, with relatively little technical effort, by carefully exploiting the entropy structure.

The energy component $u$ plays a distinguished role in ERDS that has to be taken advantage of when interested in a general analysis. At a technical level, the physical constraint of the convex function $h(u,c)$ being non-increasing in $u$ (to ensure a non-negative temperature) restricts the range of relevant functions $\hat\sigma$ to sublinearities such as $\hat\sigma(u)=u^\nu$ for some $\nu\in[0,1)$ (with $\nu=0$ corresponding to $\log$).  
Unless $\hat\sigma (u)$ has close to linear growth for large values of $u$,
even the possibility of an existence theory solely based on the entropy estimate is questionable in general dimensions.
We are interested in covering more degenerate choices of $\hat\sigma$, and therefore cannot purely rely on the (adjusted) relative entropy  to measure the distance of a weak to the strong solution.
Instead, we exploit the absence of source terms in the $u$-component of the evolution system, which allows to give an extra, scalar-like  structure to the evolution law for $u$. Here, we content ourselves with the arguably simplest choice of an $L^2$-structure, meaning that we consider weighted generalised distances of the form
\begin{align}
	\Dist_\alpha(z,\tilde z) =  H_\rel^*(z,\tilde z)+\tfrac{\alpha}{2}\|u-\tilde u\|_{L^2(\Om)}^2,
	\quad \alpha\in(0,\infty).
\end{align}
This is consistent with the approach in~\cite{FHKM_2020} and allows us in particular to show the weak-strong uniqueness property for the potentially pathological solutions constructed in ~\cite[\thmRen]{FHKM_2020}, where  a cross-diffusion dominant regime was considered 
with  gradients of the internal energy density inducing a (possibly) non-integrable
concentration flux.
Furthermore, by exploiting the existence of such an additional quantity that up to some error term is dissipated along the flow, we can relax the conditions on the entropy functional in~\cite{MM_2018} required for proving exponential convergence to equilibrium.  

\subsection{Technique}\label{ssec:technique}
Here, we briefly outline, at a formal level, the main points of our argument showing a weak-strong stability estimate of the form 
\begin{align}\label{eq:dist.ineq}
	\tfrac{\dd}{\dd t}\Dist_\alpha(z,\tilde z)
	\lesssim_{T,\alpha,\xi^*} \Dist_\alpha(z,\tilde z),
\end{align}
on any finite time horizon $(0,T)$, $T<T^*$, where $z$ is assumed to be a \enquote{weak} (renormalised) solution and $\tilde z$ a \enquote{strong} solution of~\eqref{eq:erds} in $(0,T^*)\times\Om$ for some $T^*\in(0,\infty]$.

First, letting $\dist_\alpha(z,\tilde z)=h_\rel^*(z,\tilde z)+\tfrac\alpha2|u-\tilde u|^2,$
where 
\begin{align}\label{eq:222}
	h_\rel^*(z,\tilde z) = h(z)-D_ih(\tilde z)(\xi^*(z)z_i-\tilde z_i)-h(\tilde z),
\end{align}
we can write $\Dist_\alpha(z,\tilde z)=\int_\Om \dist_\alpha(z,\tilde z)\,\dd x$.
We further recall that the function $\xi^*=\xi^*(z)$ will be chosen such that $\xi^*(z)=1$ if $\sum_{i=0}^n z_i\le E$ for an auxiliary parameter $E\gg \sum_{i=0}^n \tilde z_i$. Then, if $E=E(\tilde z,\min\{\alpha,1\})$ is chosen large enough, $\dist_\alpha(z,\tilde z)\ge0$ for all $z\in[0,\infty)^{1+n}$ with equality if and only if $z=\tilde z$.

To sketch the argument leading to~\eqref{eq:dist.ineq}, let us for simplicity only consider the case where $A_{0j}(z)=a(z)\delta_{0j}$ with $a\gtrsim1$. In this case, it will suffice to take $\alpha\ge1$. We now assume that $z$ and $\tilde z$ are sufficiently regular solutions of~\eqref{eq:erds} (with $A=\mathbb{M}D^2h$, $\mathbb{M}\ge0$, $D_ihR_i\le0$), where the strong solution $\tilde z$ be such that $\|\tilde z\|_{C^{0,1}([0,T]\times\bar\Om)}<\infty$  and
$\inf_{(0,T)\times\Om}\tilde z_i>0$ for all $i\in\{0,\dots,n\}$ and all $T<T^*$. 
To estimate the time evolution of $\Dist_\alpha(z,\tilde z)$, one formally computes
\begin{align}
	\tfrac{\dd}{\dd t}\Dist_\alpha(z,\tilde z)= \int_\Om \;\rho^{(h)}\;\dd x+\alpha \int_\Om \;\rho^{(\en)}\;\dd x, 
\end{align}	
where (see Lemma~\ref{l:evol.hrel})
\begin{align}\label{eq:rho.h.v1}
	\rho^{(h)} &:=-\nabla D_{i}h(z)\cdot\mathbb{M}_{il}(z)\nabla D_{l}h(z)
	\\&\qquad + \nabla (D_i(\xi^*(z)z_j)D_jh(\tilde z) )\cdot \mathbb{M}_{il}(z)\nabla D_lh(z)
	\\&\qquad + \nabla\Big(D_{ij}h(\tilde z)(\xi^*(z)z_j-\tilde z_j)\Big)\cdot 
	\mathbb{M}_{il}(\tilde z)\nabla D_lh(\tilde z)
	\\&\qquad- D_{ij}h(\tilde z)(\xi^*(z)z_j-{\tilde z}_j) R_i(\tilde z)
	\\&\qquad+ (D_ih(z)-D_jh(\tilde z)D_i(\xi^*(z)z_j)) R_i(z),
\end{align}
and 
\begin{align}
	\rho^{(\en)}&:= -a(z)|\nabla u|^2-a(\tilde z)|\nabla \tilde u|^2+a(z)\nabla u\cdot\nabla \tilde u+a(\tilde z)\nabla u\cdot\nabla \tilde u
	\\&=-a(z)|\nabla u-\nabla \tilde u|^2-(a(z)-a(\tilde z))(\nabla u-\nabla \tilde u)\cdot\nabla \tilde u.\label{eq:rho.u.A+}
\end{align}
Thus, to show~\eqref{eq:dist.ineq} it suffices to obtain a pointwise upper bound of the form  
\begin{align}\label{eq:ptw.rho}
	\rho_\alpha:=\rho^{(h)}+\alpha\rho^{(\en)}\lesssim\dist_\alpha(z,\tilde z).
\end{align}
This pointwise estimate will be proved by distinguishing  four cases determined by the value of the weak solution $z=z(t,x)\in[0,\infty)^{1+n}$ at any given point $(t,x)$. This case distinction is motivated by the following observations:

First, if $z\in[0,\infty)^{1+n}$ with $\sum_{i=0}^n z_i\le E$ for $E=E(\tilde z)$ large enough, we want $h_\rel^*(z,\tilde z)$ to coincide with the classical relative entropy density $h_\rel(z,\tilde z)=h(z)-D_ih(\tilde z)(z_i-\tilde z_i)-h(\tilde z)$ to be able to use its distance-like properties. This will be ensured by choosing  $\xi^*(z)=1$ with $D^k\xi^*(z)=0$ for all $k\in\mathbb{N}^+$ whenever $\sum_{i=0}^n z_i\le E$  (cf.\ the definition of $\xi^*$ in Sec.~\ref{ssec:def.xi*}).
Thus, if  $|z-\tilde z|$ is close to zero, the strict convexity, non-negativity and vanishing in $z=\tilde z$ of $\dist_\alpha(\cdot,\tilde z)$ imply that $\dist_\alpha(z,\tilde z)\sim_{\|\tilde z\|_{L^\infty},E} |z-\tilde z|^2$ for $|z|\le E$.
In this case, to show that $\rho_\alpha$ is quadratically small in $|z-\tilde z|$, we 
write (using $\sum_iz_i\le E$)
\begin{align}
	\rho^{(h)}&=-\nabla (D_{i}h(z)-D_ih(\tilde z))\cdot\mathbb{M}_{il}(z)\nabla (D_{l}h(z)-D_lh(\tilde z))
	\\&\qquad -\nabla (D_{i}h(z)-D_ih(\tilde z))\cdot(\mathbb{M}_{il}(z)-\mathbb{M}_{il}(\tilde z))\nabla D_lh(\tilde z)
	\\&\qquad - \nabla\big(D_{i}h(z)-D_ih(\tilde z)-D_{ij}h(\tilde z)(z_j-\tilde z_j)\big)\cdot 
	\mathbb{M}_{il}(\tilde z)\nabla D_lh(\tilde z)
	\\&\qquad+\big(D_ih(z)-D_ih(\tilde z)-D_{ij}h(\tilde z)(z_j-{\tilde z}_j)\big) R_i(\tilde z)
	\\&\qquad+ (D_ih(z)-D_ih(\tilde z))(R_i(z)-R_i(\tilde z)),
\end{align}
see case $\mathcal{A}_+$ in the proof of Theorem~\ref{thm:wkuniq} for details.
In order to deal with the terms involving a gradient of $z$ that appear in the second and the third term on the RHS, one would like to exploit the non-positive first term on the RHS. 
Typically (such as in the ERDS models considered in~\cite{FHKM_2020}),  the submatrix $(\mathbb{M}_{il}(z))_{i,l=1,\dots,n}$ will, however, degenerate as soon as $c_i\searrow0$ for some $i\in\{1,\dots,n\}$. Yet if
$\min\{z_1,\dots, z_n\}\ge\iota$ for some $\iota>0$, then it is possible to assume that $(\mathbb{M}_{il}(z))_{i,l=1,\dots,n}\gtrsim_\iota \mathbb{I}_n$.
This, combined with the second line in~\eqref{eq:rho.u.A+} and suitable smoothness assumptions on $\mathbb{M}$ and $h$, will allow us to infer that 
\begin{align}
	\rho_\alpha\lesssim_{E,\iota,\tilde z} |z-\tilde z|^2 
\end{align}
whenever $z(t,x)\in \mathcal{A}_+:=\{z'\in[0,\infty)^{1+n}:\min\{z_0',\dots, z_n'\}\ge\iota,\;\sum_{i=0}^nz_i'\le E\}$ for some $\iota>0$ and sufficiently large $E\ge 1$.

To deal with the case $z(t,x)\in \mathcal{A}_0:=\{z':\min\{z_0',\dots, z_n'\}<\iota,\;\sum_{i=0}^nz_i'\le E\}$, we fix $\iota=\iota(\tilde z)>0$ small enough such that $\inf\tilde z_i\ge2\iota$ for all $i=0,\dots, n$. This implies that $|z-\tilde z|\ge\iota$ whenever $z\in \mathcal{A}_0$. Thus, since $|z-\tilde z|\gtrsim1$ is bounded away from zero,  so is $\dist_\alpha(z,\tilde z)$ (see Prop.~\ref{prop:coerc.hrel}). 
It then suffices to have suitable coercivity estimates on $\scrp(z)$ that allow to absorb those terms on the RHS of~\eqref{eq:rho.h.v1} that involve gradients of $z$ and do not have a sign by the first term on the RHS, which equals~$-\scrp(z)$. (Such coercivity estimates are typically already needed in the construction of solutions.)

It remains to consider the case $\sum_iz_i>E$, where $E$ will be chosen large enough, in particular such that $E\ge\|\sum_i\tilde z_i\|_{L^\infty((0,T)\times\Om)}+1$ and $E\ge E_0$ with $E_0$ being such that $\dist_\alpha(z,\tilde z)\gtrsim \sum_i{c_i}\log_+(c_i)+u^2+1$ for all $\tilde E\ge E_0$. 
If $z\not\in\supp\xi^*$, $\rho^{(h)}$ takes a simple form. The entropy dissipating property of diffusion and reactions, and the Lipschitz regularity of the strong solution $\tilde z$ are sufficient to deduce~\eqref{eq:ptw.rho} in this case (referred to as $z\in\mathcal{C}$). 

The intermediate case (below referred to as case $z\in\mathcal{B}$), where $0<\xi^*<1$, is  more delicate as can be seen in formula~\eqref{eq:rho.h.v1}, where the second term on the RHS involves summands that are quadratic in the gradient of the renormalised solution. Here, we take advantage of an idea by Fischer~\cite{Fischer_2017}. 
In order to be able to absorb this bad term by the first term on the RHS of~\eqref{eq:rho.h.v1}, 
another scale $E'\gg E$ is introduced (for convenience we choose $E'=E^N$, as in~\cite{Fischer_2017}), and $\xi^*$ will be taken such that $\xi^*(z)=0$ if and only if $\sum_{i=0}^nz_i\ge E'$, $\xi^*(z)=1$ if and only if $\sum_{i=0}^nz_i\le E$. 
On these scales, $\xi^*$ can be chosen in such a way that derivatives of $\xi^*$ have an additional decay property enabling
the desired absorption if $\tfrac{E'}{E}$ is large enough.
Finding $\xi^*$ such that absorption is possible is non-trivial and relies on a logarithmic gain.

\subsection{Outline}
The rest of the article is structured as follows. In Section~\ref{sec:results} we introduce relevant definitions and hypotheses, and formulate our main results:
a weak-strong uniqueness principle for dissipative renormalised solutions to~\eqref{eq:erds} (see Thm~\ref{thm:wkuniq}), 
a strong entropy dissipation inequality as used in the proof of Theorem~\ref{thm:wkuniq} (see Prop.~\ref{prop:edin}), and a result on the exponential convergence to equilibrium (see Prop.~\ref{prop:expconv}).
In Section~\ref{ssec:examples} we present selected examples that our main results apply to, including the class of ERDS considered in~\cite{FHKM_2020} as well as a class of models with cross diffusion between species.

The weak-strong uniqueness principle is proved in Section~\ref{sec:proofs}, starting with several auxiliary results with the actual proof of Theorem~\ref{thm:wkuniq} being given in Section~\ref{ssec:stab.ineq}.
In Sections~\ref{sec:edin} and~\ref{sec:exp.conv} respectively, we establish
the entropy dissipation inequality~\eqref{eq:edin} and  the exponential convergence to equilibrium for a specific ERDS, below referred to as Model~\Mzero{}.
Some auxiliary results are gathered in Appendix~\ref{sec:app}.
In Appendix~\ref{ssec:ED.ren} we explain how to derive inequality~\eqref{eq:edin} for 
the renormalised solutions constructed in~\cite{FHKM_2020} for a model with non-integrable diffusive flux.

\subsection{Notations}\label{ssec:notations}\small
\begin{itemize}
	\item \emph{Summation convention}: any unspecified summations of the form $\sum_i$ are to be understood as $\sum_{i=0}^n$.  For brevity, we use a summation convention for summing over the system's components $i=0,\dots,n$ in case of repeatedly occurring indices while omitting the summation symbol. In ambiguous situations the summation symbol will be used. Summations restricted to $i=1,\dots,n$ (excluding the $u$-component) will always be made explicit.
	In our convention, summation over repeated indices has priority over other mathematical operations such as	integration or taking the absolute value. For instance, by default we let $|A_{ik}(z)\nabla z_k|=|\sum_{k=0}^nA_{ik}(z)\nabla z_k|$.
	\item For technical concerns regarding the notation $\mathbb{M}_{il}(z)\nabla D_lh(z)$, we refer to Remark~\ref{rem:con.M}.
	\item We denote by $R=(0,R_1,\dots, R_n)^T$ the vectors of reaction rates.
	\item Given $T^*\in(0,\infty]$, we  let $I=[0,T^*)$ denote the time horizon of interest. For $T>0$, we abbreviate $\Om_T:=(0,T)\times\Om$.
	\item For functions $f=f(z_0,\dots,z_n)$ we let $D_if=\tfrac{\partial f}{\partial z_i}$ and
	$D_{ij}f=\tfrac{\partial^2f}{\partial z_i\partial z_j}$ for $i,j\in\{0,\dots,n\}$. 
	\item In estimates, $C<\infty$ typically denotes a finite (sufficiently large) constant that may change from line to line, while we often use $\epsilon>0$ to denote a (sufficiently small) positive constant.
	\item For quantities $A,B\ge0$ we write $A\lesssim B$ if there exists a fixed constant $C<\infty$ such that $A\le CB$. The notation $A\gtrsim B$ means $B\lesssim A$, while $A\sim B$ is to be understood as both $A\lesssim B$ and $A\gtrsim B$ being satisfied. In order to indicate dependencies of the constant $C=C(p_1,\dots,p_k)$ on certain parameters $p_1,\dots,p_k$, we write $A\lesssim_{p_1,\dots,p_k} B$, and analogously for $\gtrsim$ and $\sim$.
	\item Any dependence of constants and estimates on the regular solution $\tilde z\in C^{0,1}$ will usually not be explicitly indicated.
	\item We let $\min(z):=\min\{z_0,z_1,\dots,z_n\}$ for $z=(z_0,\dots, z_n)\in[0,\infty)^{1+n}$.
	\item By default, $|\cdot|$ denotes the Euclidean norm, e.g.\ $|z|=(\sum_i|z_i|^2)^\frac{1}{2}$
	\item $|z|_1=\sum_{i=0}^nz_i$ and $|c|_1=\sum_{i=1}^nc_i$.
	\item For time-dependent integral functionals $\int_\Om f((t,x),z(t,x))\,\dd x$, where $z=z(t,x)$ denotes a \enquote{weak} solution of~\eqref{eq:erds} taking in a suitable sense the data $\zin$, 	
	we use the convention
	$$\eval{\int_\Om f((t,x),z(t,x))\,\dd x}_{t=0}^{t=T}:=\int_\Om f((T,x),z(T,x))\,\dd x-\int_\Om f((0,x),\zin(x))\,\dd x$$
	provided the terms on the RHS are well-defined.
	\item For an open set $U\subset\mathbb{R}^N$, $C^k(U)$ denotes the space of continuous functions on $U$ that are continuously differentiable up to order $k\in\mathbb{N}$. By $C^{k,\nu}(U)=C^{k,\nu}_\loc(U)$, we denote the space of functions in $C^k(U)$, whose $k$-th derivative is $\nu$-H\"older continuous for some $\nu\in(0,1]$ on compact subsets $K\subseteq U$. (We use the symbol $C^{k,\nu}_\loc(U)$ for clarity's sake.)
	\item For $\Om$ bounded, we let $L\log L(\Om):=\{f\in L^1(\Om):f\ge0\text{ a.e.\ and }\int_\Om f\log f\,\dd x<\infty\}$. 
	\item The abbreviation `hp.' stands for hypothesis.
\end{itemize}
\normalsize

\section{Main results}\label{sec:results}

\subsection{Assumptions}\label{ssec:ass}
Throughout these notes, we let $d\ge1$ and $\Omega\subset\mathbb{R}^d$ be a  bounded Lipschitz domain with $|\Om|=1$. We further let $T^*\in(0,\infty]$ and $I=[0,T^*)$.

To prepare for stating our main result, the weak-strong uniqueness principle (Theorem~\ref{thm:wkuniq}), we gather the following conditions.
\begin{enumerate}[label=\textup{(A$\arabic*$)}]
	\item\label{it:hC3} Entropy: $h\in C^4((0,\infty)^{1+n})$ is of the form~\eqref{eq:S0}, where $\hat\sigma$ and $w_i$ are supposed to satisfy~\eqref{eq:h2}.
\end{enumerate}
\begin{enumerate}[resume,label=\textup{(A$\arabic*$)},ref=\textup{A$\arabic*$}]
	\item\label{it:react} Reactions: 
	$R=(0,R_1,\dots, R_n)\in C([0,\infty)^{1+n})^{1+n}$
	\begin{enumerate}[label=\textup{(\roman*)},ref=\textup{(\ref*{it:react}.\roman*)}]
		\item\label{eq:hp.Rdiss} satisfy $\sum_{i=1}^nD_ih(z)R_i(z)\le 0$ in $(0,\infty)^{1+n}$
		\item\label{it:R.locLip} are locally Lipschitz continuous in $(0,\infty)^{1+n}$.
	\end{enumerate}
\end{enumerate}
\begin{enumerate}[resume,label=\textup{(A$\arabic*$)}]
	\item\label{eq:hp.Mnondeg}  Mobility matrix: 
	$\mathbb{M}\in \big(C^{0,1}_\loc((0,\infty)^{1+n})\cap  C([0,\infty)^{1+n})\big)^{(1+n)\times(1+n)}$ and there exist non-negative functions  $m,a\in C^{0,1}_\loc((0,\infty)^{1+n})$ and $\mOns\in\{0,1\}$ with $0\le m\lesssim\mOns$ and $a\gtrsim 1$ such that 
	\begin{align}\label{eq:hp.ueq1}\tag{A3.a}
		\mathbb{M}_{0l}=\mathbb{\widetilde M}_{0l}+\delta_{0l}m
		\quad \text{ for }l=0,\dots,n,
	\end{align}
	for suitable $\mathbb{\widetilde M}_{0l}$ satisfying 
	$\sum_{l=0}^n\mathbb{\widetilde M}_{0l}D_{lj}h = \delta_{0j}a.$
	
	Moreover, for all $z\in[0,\infty)^{1+n}$ with $\min_iz_i\ge\iota$ for some $\iota>0$
	there exists $\epsilon(\iota)>0$ such that
	\begin{align}\label{eq:Mnd}\tag{A3.b}
		\mathbb{M}(z)\ge\diag(m(z),0,\dots,0)+\epsilon(\iota)\diag(0,1,\dots,1).
	\end{align}
	By continuity, when $\iota=0$,~\eqref{eq:Mnd} holds true with $\epsilon(\iota)=0$.
\end{enumerate}
Using our standard notation $A(z)=\mathbb{M}(z)D^2h(z)$, hp.~\eqref{eq:hp.ueq1} implies that, formally,
\begin{align}
	\label{eq:hp.ueq}\tag{A3.c}
	\sum_{j=0}^nA_{0j}(z)\nabla z_j=\sum_{l=0}^n\mathbb{M}_{0l}(z)\nabla D_lh(z)=a(z)\nabla u+m(z)\nabla D_0h(z).
\end{align}

We further need certain bounds on the flux and the concentration gradients in terms of the entropy dissipation. For this purpose we define for non-negative functions $z_j\in L^1_\loc(I,L^1(\Om))$ such that $\nabla (z_j^s)\in L^2_\loc(I;L^2(\Om))$ for some $s\in\{\tfrac{1}{2},1\}$ for each $j\in\{0,\dots, n\}$, the quantity  
\begin{align}
	\scrp(z):&=\nabla z: (D^2h(z)A(z)\nabla z)
	\\	&=\nabla D_ih(z)\cdot (\mathbb{M}_{il}(z)\nabla D_lh(z)),
\end{align}
where the second equality is to be understood in a formal sense, see Remark~\ref{rem:con.M}.
By the positive semi-definiteness of $\mathbb{M}$ imposed by hp.~\eqref{eq:Mnd}, we have $\scrp(z)\ge0$ for any such $z$, and more specifically, $\scrp(z)\ge m(z)|\nabla D_0h(z)|^2$.

\begin{enumerate}[resume,label=\textup{(A$\arabic*$)}]
	\item\label{hp:trunc} For all $K\ge1$ 
	\begin{align}\tag{A4.a}
		\chi_{\{|z|\le K\}}|\nabla c|&\lesssim_K\sqrt{\scrp(z)},\label{eq:hp.dztrunc.P}
		\\\chi_{\{|z|\le K\}}|\sum_{j=0}^nA_{ij}(z)\nabla z_j|&\lesssim_K\sqrt{\scrp(z)}\label{eq:hp.fluxtrunc.P}\tag{A4.b}
	\end{align}
	for all $i=0,\dots,n$.
	\item\label{eq:hp.A00} For $a(z)$ as in~\ref{eq:hp.Mnondeg}, 
	\begin{align}
		|a(z)\nabla u|\lesssim (1+u)\sqrt{\scrp(z)}.
	\end{align}
\end{enumerate}
Additionally, we often impose the following bound:
\begin{enumerate}[resume,label=\textup{(A$\arabic*$)}]
	\item \label{hp:grad.flux.control} For all $0\le i\le n$ 
	$$\chi_{\{|z|_1\ge1\}}|\nabla z||\sum_{j=0}^nA_{ij}(z)\nabla z_j|\lesssim |z|\scrp(z)+|z||\nabla u|^2.$$
\end{enumerate}
If~\ref{hp:grad.flux.control} is not satisfied, we have to assume that $\mOns=0$ in~\ref{eq:hp.Mnondeg} together with the condition:
\begin{enumerate}[resume,label=\textup{(A$\arabic*$')}]\addtocounter{enumi}{-1}
	\item \label{hp:trunc.grad.flux} For all $0\le i\le n$ and all $\underline u\in(0,1]$
	$$\chi_{\{u\ge\underline u\}}\chi_{\{|z|_1\ge1\}}|\nabla z||\sum_{j=0}^nA_{ij}(z)\nabla z_j|\lesssim |z|\scrp(z)+C(|z|, \underline{u})|\nabla u|^2.$$
\end{enumerate}
Let us observe that \ref{hp:grad.flux.control}$\implies$\ref{hp:trunc.grad.flux}.

A selection of relevant examples fulfilling the above hypotheses is provided in Section~\ref{ssec:examples}.

\subsection{Definitions and Results}
Throughout this text, we write $A(z):=\mathbb{M}(z)D^2h(z)$, where $h$ takes the form~\eqref{eq:S0},~\eqref{eq:h2}. We further recall our summation convention (see Notations~\ref{ssec:notations}).

\begin{definition}[Renormalised solution]\label{def:renorm}Let $I=[0,T^*)$ and suppose that the vector-valued function $z=(u,c_1,\dots,c_n)$ has non-negative components $z_i\ge0$ satisfying
	$\sqrt{z_i}\in L^2_\loc(I;H^1(\Om))$ or $z_i\in L^2_\loc(I;H^1(\Om))$ for all $i=0,\dots, n$.
	Further suppose that  for all $E\ge1$
	\begin{equation}\label{eq:zreg}
		\chi_{\{|z|\le E\}}A_{ik}(z)\nabla z_k\in L^2_\loc(I;L^2(\Omega)),
	\end{equation}
	for every $i\in\{0,\dots,n\}$.
	
	We call such $z$ a \emph{renormalised solution} of the energy-reaction-diffusion system~\eqref{eq:erds} in  $\Om_{T^*}:=(0,T^*)\times\Om$ with initial data $\zin$
	if for all $\xi\in C^\infty(\mathbb{R}_{\ge0}^{1+n})$ with compactly supported derivative $D\xi$, all $\psi\in C^\infty(I\times\bar\Omega)$ and almost all $T\in(0,T^*)$
	\begin{equation}\label{eq:118}
		\begin{split}
			&\int_\Om\xi(z(T,\cdot))\psi(T,\cdot)\,\dd x-\int_\Om\xi(\zin)\psi(0,\cdot)\,\dd x-\int_0^T\!\!\int_\Om\xi(z)\partial_t\psi\,\dd x\dd t
			\\&\qquad=
			-\int_0^T\!\!\int_\Om D_{ij}\xi(z)A_{ik}(z)\nabla z_k\cdot\nabla z_j\psi \,\dd x\dd t 
			\\&\qquad\quad -\int_0^T\!\!\int_\Om D_i\xi(z)A_{ik}(z)\nabla z_k\cdot\nabla\psi\,\dd x\dd t
			+\int_0^T\!\!\int_\Om D_i\xi(z)R_i(z)\psi\,\dd x\dd t.
		\end{split}
	\end{equation}
\end{definition}

The renormalised formulation~\eqref{eq:118} and the required functional setting alone are in general too weak to deduce the weak-strong uniqueness principle.
In Definition~\ref{def:diss.renorm} below, we introduce a more satisfactory generalised solution concept that strengthens Definition~\ref{def:renorm}, but is still general enough to be consistent with the existence results in~\cite{FHKM_2020}.

\begin{remark}\label{rem:testfk}
	By approximation, given a renormalised solution $z$, the equality~\eqref{eq:118} can be seen to hold true for a larger set of test functions $\psi\in C(I\times\bar\Omega)$ with $\partial_t\psi\in L^1_\loc(I;L^1(\Omega))$, $\nabla\psi\in L^2_\loc(I;L^2(\Omega))$, and for truncation functions $\xi\in C^2(\mathbb{R}_{\ge0}^{1+n})$ with $\supp D\xi$ compact.
\end{remark}

\begin{remark}[Notation]\label{rem:con.M} Let $z$ denote a renormalised solution of~\eqref{eq:erds} in the sense of Def.~\ref{def:renorm}.
	To keep notation simple and better emphasise the entropy structure of the diffusive part, we will often use a \enquote{symbolic} notation writing $\mathbb{M}_{il}(z)\nabla D_lh(z)$ instead of $A_{ik}(z)\nabla z_k$, where as before the summation convention is used. Likewise, we write  $\nabla D_ih(z)$ instead of the more precise notation $ D_{ij}h(z)\nabla z_j$. The point here is that while, by hypothesis, the weak derivatives $\nabla z_j, j=0,\dots,n,$ are well-defined (in the standard Sobolev/distributional sense), the function $D_ih(z)$ may not be weakly differentiable.
\end{remark}

\begin{remark}
	The following equation, equivalent to~\eqref{eq:118}, can be obtained  by \enquote{reversing the product rule} for $\nabla$
	\begin{equation}\label{eq:118..}
		\begin{split}
			&\int_\Om\xi(z(T,\cdot))\psi(T,\cdot)\,\dd x-\int_\Om\xi(\zin)\psi(0,\cdot)\,\dd x
			-\int_0^T\!\!\int_\Om\xi(z)\partial_t\psi\,\dd x\dd t
			\\&\quad=-\int_0^T\!\!\int_\Om \nabla\big(D_i\xi(z)\psi\big)\cdot \mathbb{M}_{il}(z)\nabla D_l(z)\,\dd x\dd t 
			+\int_0^T\!\!\int_\Om D_i\xi(z)R_i(z)\psi\,\dd x\dd t,
		\end{split}
	\end{equation}
	where we recall our convention $\mathbb{M}_{il}(z)\nabla D_lh(z):=A_{ik}(z)\nabla z_k$, see Remark~\ref{rem:con.M}.
\end{remark}

Notice that thanks to the hypothesis of $D\xi$ being compactly supported,  no growth restrictions have to be imposed on the reaction term $R(z)$ in~\eqref{eq:118}, and none of the flux terms $A_{ik}(z)\nabla z_k$ is necessarily required to be integrable in order for the integrals in~\eqref{eq:118} to converge. At the same time, this restrictive condition on the set of admissible truncation functions $\xi$ means that recovering a separate (e.g.~weak) formulation of a single component $i_0\in\{0,\dots, n\}$ of the system (assuming integrability of $A_{i_0k}(z)\nabla z_k$ and $R_{i_0}(z)$) is not immediate unless all components of the flux and the reactions are known to be integrable. 
Thus, consistent with the existence result for~\eqref{eq:erds} in~\cite{FHKM_2020}, the present analysis
additionally assumes a weak formulation for the energy component.

\begin{definition}[Weak solution of energy equation]\label{def:energyeq.weak}
	Let $z=(u,c)$ be as in Def.~\ref{def:renorm}. We say that $u$ is a \emph{weak solution} of the energy component $\partial_tu=\divv(A_{0j}(z)\nabla z_j)$ in $\Om_{T^*}$ with no-flux boundary conditions and initial condition $\uin$
	if $A_{0j}(z)\nabla z_j\in L^1_\loc(I;L^1(\Om))$ and if for all $\varphi\in C^1(I\times\bar\Om)$ and almost all $T<T^*$
	\begin{equation}\label{eq:u.weak}
		\begin{split}
			\int_\Om  u(T,\cdot)\varphi(T,\cdot)\,\dd x-\int_\Om  \uin\varphi(0,\cdot)\,\dd x 
			&- \int_0^T\!\!\int_\Om u\partial_t\varphi\,\dd x\dd t
			\\=& -\int_0^T\!\!\int_\Om A_{0j}(z)\nabla z_j\cdot\nabla\varphi\,\dd x\dd t.
		\end{split}
	\end{equation}
\end{definition}

By carefully using lower semicontinuity-type properties of the entropy and entropy dissipation, 
the existence proof of global-in-time weak and renormalised  solutions to ERDS as provided 
in~\cite{FHKM_2020} allows to show that 
for almost all $T\in(0,\infty)$ the constructed solutions satisfy the entropy dissipation inequality
\begin{equation}\label{eq:edin}\tag{\textsc{ed}}
	H(z(T))-H(\zin)	\le -\int_0^T\!\!\int_\Om \scrp(z)\,\dd x \dd t+ \int_0^T\!\!\int_{\Om} R_i(z)D_ih(z)\,\dd x\dd t,
\end{equation}
cf.~Section~\ref{ssec:ED.ren},
where we recall the notation $\scrp(z):=\nabla D_{i}h(z)\cdot\mathbb{M}_{il}(z)\nabla D_{l}h(z)\ge0$.
Observe that thanks to the non-negativity of $\scrp(z)$ and $-R_i(z)D_ih(z)$ (to be assumed throughout), any function $z=(u,c)$ with $u\in L^\infty_tL^1_x$ and  well-defined, measurable $\scrp(z)$ that satisfies~\eqref{eq:edin} with $h(\zin)\in L^1(\Om)$ necessarily has the regularity 
$\scrp(z)\in L^1(\Om_T)$ and $R_i(z)D_ih(z)\in L^1(\Om_T)$. (For the estimate on $H(z(T))$ required in this argument, see~\eqref{eq:112ptw} in the appendix.)

We further note that the energy equation $\partial_tu=\divv(A_{0j}(z)\nabla z_j)$ is satisfied in the weak sense by the solutions constructed in ref.~\cite{FHKM_2020}, that $\nabla u \in L^2(\Om_T)$, and that the quantity $G(u)=\tfrac{1}{2}\|u\|_{L^2(\Om)}^2$ satisfies (with an equality) 
\begin{equation}\label{eq:ene}\tag{\textsc{ene}}
	\En(u(T))-\En(\uin) \le -\int_0^T\!\!\int_\Om a(z)|\nabla u|^2\,\dd x\dd t
	-\int_0^T\!\!\int_\Om m(z)\nabla D_0h(z)\cdot\nabla u\,\dd x\dd t
\end{equation}
for almost all $T\in(0,\infty)$.
(In~\cite{FHKM_2020}, the case $m\equiv0$ was considered.)

The weak-strong uniqueness principle we will establish in our main theorem  (Theorem~\ref{thm:wkuniq} below) is valid for renormalised solutions that are \emph{dissipative} in the sense that they conform to~\eqref{eq:edin} and~\eqref{eq:ene}
(and satisfy some extra hypotheses related to the energy component).

\begin{definition}[Dissipative renormalised solution]\label{def:diss.renorm} 
	We call $z=(u,c)$ a \emph{dissipative renormalised solution} of  
	system~\eqref{eq:erds} in $\Om_{T^*}$  with initial data $\zin$ 
	if it is a renormalised solution with data $\zin$  in the sense of Definition~\ref{def:renorm} 
	that fulfils the energy equation in the weak sense of Definition~\ref{def:energyeq.weak} 
	with $\int_0^T\!\!\int_\Om a(z)|\nabla u|^2\,\dd x\dd t<\infty$ for all $T\in(0,T^*)$
	and further obeys the inequalities~\eqref{eq:edin} and~\eqref{eq:ene} for almost all $T\in(0,T^*)$.
\end{definition}

Let us stress that the renormalised solutions in~\cite[Theorem~1.8]{FHKM_2020} 
are in fact dissipative renormalised solutions in the sense of the above definition.\footnote{In Section~\ref{ssec:ED.ren} we explain how to deduce~\eqref{eq:edin} along the construction in~\cite[Theorem~1.8]{FHKM_2020}.}
Thus,  dissipative renormalised solutions provide a natural framework for studying~\eqref{eq:erds}. It is interesting to note that there are parallels of the present setting  
to classical literature such as~\cite{BBGGV_1995,BM_1997} on elliptic/parabolic equations, where a satisfactory solution concept is obtained by additionally imposing certain entropy or regularity conditions that are indispensable for the uniqueness proof.
See also~\cite{DalMMOP_1999,CW_1999} and references therein as well as the classical literature on conservation laws mentioned in the introduction.

Before stating our main theorem, we need to specify what we  understand by a \enquote{strong} solution.
\begin{definition}[Weak solution]\label{def:weaksol}Let $I=[0,T^*)$.
	We call a function $z\in L^1_\loc(I;W^{1,1}(\Om))$, $z=(u,c_1,\dots, c_n)$ 
	with $z_i\ge0$ for all $i$, 
	a \emph{weak solution} of~\eqref{eq:erds} in $\Omega_{T^*}$ with initial data $\zin$ if $A_{ij}(z)\nabla z_j,R_i(z)\in L^1_\loc(I;L^1(\Om))$ for all $i$,
	and if for all $\psi\in C^\infty(I\times\bar\Omega)^{1+n}$ and almost all $T\in(0,T^*)$
	\begin{equation}\label{eq:weakform}
		\begin{split}
			\int_\Om  z_{i}(T,\cdot)\psi_i(T,\cdot)\,\dd x&-\int_\Om  \zin_i\psi_i(0,\cdot)\,\dd x 
			- \int_0^T\!\!\int_\Om z_i\partial_t\psi_i\,\dd x\dd t
			\\&= -\int_0^T\!\!\int_\Om A_{ij}(z)\nabla z_j\cdot\nabla\psi_i\,\dd x\dd t+\int_0^T\!\!\int_\Om R_i(z)\psi_i\,\dd x\dd t.
		\end{split}
	\end{equation}
\end{definition}

\begin{definition}[\enquote{Strong} solution]\label{def:strongsol} Let $I=[0,T^*)$.
	We call $z=z(t,x)$ a \emph{strong solution} of system~\eqref{eq:erds} in $\Om_{T^*}$ 
	with data $\zin$ 
	if it is a weak solution in the sense of Definition~\ref{def:weaksol}, has the regularity $z\in C^{0,1}_\loc(I\times\bar\Om)$ and satisfies $\inf_{\Om_{T}} z_i>0$ for $i=0,\dots,n$ and every $T\in(0,T^*)$.
\end{definition}

For strictly positive, Lipschitz continuous initial data, Amann's work~\cite{Amann_II} guarantees the local-in-time existence of a strong solution for the applications considered in the present manuscript (see Section~\ref{ssec:examples}).

\begin{theorem}[Weak-strong uniqueness]\label{thm:wkuniq}
	Let $T^*\in(0,\infty]$.	Assume hp.~\ref{it:hC3}--\ref{eq:hp.A00}.
	Further suppose that~\ref{hp:grad.flux.control} holds true,  or alternatively that $\mOns=0$ and that~\ref{hp:trunc.grad.flux} is fulfilled.
	Let $\zin=(\uin,\cin)\in L^1(\Om)^{1+n}$, $\zin_i\ge0$ for all $i$, and $h(\zin)\in L^1(\Om), \;\uin\in L^2(\Om)$.
	Let $z=(u,c)$ be a dissipative renormalised solution of system~\eqref{eq:erds} in $\Om_{T^*}$ in the sense of Definition~\ref{def:diss.renorm} taking the initial data $\zin$.
	If $\tilde z$ is a strong solution in $\Om_{T^*}$ in the sense of Definition~\ref{def:strongsol} taking the same initial data $\zin$,
	then $z=\tilde z$ a.e.~in $\Om_{T^*}$.
\end{theorem}
The proof of Theorem~\ref{thm:wkuniq} will be completed in Section~\ref{ssec:stab.ineq}.
In Section~\ref{ssec:examples} we provide a list of examples covered by this theorem, including a class of (isoenergetic) cross-diffusion systems. 

As mentioned in the introduction, the proof of Theorem~\ref{thm:wkuniq} is based on a weak-strong stability type estimate with a generalised distance involving a modified relative entropy and an $L^2$-part for the energy component  (cf.~Sec.~\ref{ssec:technique}).
The entropy inequality~\eqref{eq:edin} that dissipative renormalised solutions conform to 
is a fundamental ingredient in the proof.  As we will see in Section~\ref{ssec:ED.ren}, under suitable hypotheses on the data, the solutions constructed in ref.~\cite{FHKM_2020} enjoy this estimate. 
Note that similar situations are encountered in the context of weak-strong uniqueness in fluid dynamics problems, see e.g.~\cite{Wiedemann_2018,RRS_2016} and references therein.
We would, however, like to point out that for many of the models we are interested in, 
it is possible to derive the entropy dissipation inequality~\eqref{eq:edin} from the renormalised formulation if suitable regularity conditions are met (the integrability of $\scrp(z)$ and $a(z)|\nabla u|^2$ are sufficient). 
In order to illustrate the ideas, we provide a proof of inequality~\eqref{eq:edin} 
for one of the models in the ref.~\cite{FHKM_2020} considering ERDS of the form
\begin{subequations}
	\label{eq:ex.erds}
	\begin{align}
		\label{ex.u}
		\dot u &=
		\divv \!\big( \,  a(u,c)  \nabla u+m(z)\nabla D_0h(z)\big),
		\\
		\label{ex.c}
		\dot c_i &= \divv \Big( m_i(u,c) \nabla \log(\tfrac{c_i}{w_i(u)})
		+   a(u,c)  c_i \tfrac{w'_i(u)}{w_i(u)} \nabla u \bigg) + R_i(u,c),
	\end{align}
\end{subequations}
where, as before, $h(u,c)$ satisfies~\eqref{eq:S0},~\eqref{eq:h2}, and where 
$a(z):=\pi_1(z)\gamma(z)$, with 
\begin{align}\label{eq:def.gamma}
	\gamma(u,c):=\big(\partial^2_u h- \sum_{i=1}^n c_i \big(\tfrac{w'_i}{w_i}\big)^2 \big) 
	= -\hat\sigma''(u)-\sum_{i=1}^nc_i\tfrac{w_i''(u)}{w_i(u)}>0.
\end{align}
System~\eqref{eq:ex.erds} is obtained as a special case of~\eqref{eq:erds} by choosing
\begin{equation}
	\label{eq:mobility1}
	\mathbb M(z) := \diag \big(m, m_1,\dots,m_n\big) + \pi_1\mu\otimes\mu
\end{equation}
with non-negative functions $m,m_i,\pi_1\in C([0,\infty)^{1+n})$ 
to be specified below, and
$\mu=(1,\mu_1,\dots, \mu_n)$ determined by
\begin{align}
	\label{eq:defmui}
	\mu_i(u,c) := \frac{w_i'(u)}{w_i(u)}c_i\qquad\qquad \text{ for } i\in\{1,\dots, n\}.
\end{align}
Throughout, 
the bound
\begin{align}\label{eq:hp.pi1}
	0\le \sqrt{\pi_1(u,c)}\lesssim 1+u
\end{align} 
and the mild regularity condition $\sqrt{\pi_1(u,c)}w_i'(u)\in C([0,\infty)^{1+n}), i=1,\dots,n,$ will be imposed,
the latter ensuring the continuity of $\mathbb{M}$ in $[0,\infty)^{1+n}$.

The functions $m_i(u,c)$, $i\in\{1,\dots, n\}$, are  assumed to take the following form for certain $a_i\in C([0,\infty)^{1+n})$ and $\kappa_{0,i},\kappa_{1,i}\ge0:$
\begin{equation}\label{eq:ai}
	\begin{split}
		&m_i(u,c) = c_ia_i(u,c), \qquad\quad\text{where}
		\quad\; a_i(u,c)\sim\kappa_{0,i}+\kappa_{1,i}c_i.
	\end{split}
\end{equation}

\smallskip

{\noindent \bf Hypotheses of Model~\hypertarget{Mzero}{(M0)}.}	 Model~\Mzero{} consists of the following conditions:
\begin{itemize}
	\item Entropy density $h$ is given by~\eqref{eq:S0},~\eqref{eq:h2} with coefficient functions satisfying
	\begin{align}\label{eq:wi.M0}
		(w_i')^2\lesssim -w_i''w_i
	\end{align}
	for all $i\in\{1,\dots,n\}$.
	\item Reactions $R_i\in C([0,\infty)^{1+n})$, $i=1,\dots, n$, satisfy~\ref{eq:hp.Rdiss}, where $R_0\equiv0$.
	\item $\mathbb{M}$ is given by~\eqref{eq:mobility1}--\eqref{eq:ai} with 
	\begin{itemize}[label=$\circ$]
		\item rank-one part:
		$\pi_1(z)\sim\tfrac{1}{\gamma(z)}$, where $\gamma$ is given by~\eqref{eq:def.gamma}.
		\item diagonal part: $0\le m(z)\lesssim\mOns$ for some $\mOns\in \{0,1\}$, and 
		$m_i(u,c)$ given by~\eqref{eq:ai} with $\kappa_{0,i}=1,\kappa_{1,i}=0$ for all $i\in\{1,\dots,n\}$.
	\end{itemize}
\end{itemize}
We recall that the evolution law~\eqref{eq:erds} associated with Model~\Mzero{} takes the form~\eqref{ex.u}, \eqref{ex.c}. (Cf.\ Lemma~\ref{l:modelBounds} and~\cite{FHKM_2020} for details.)
It is further easy to see that condition~\eqref{eq:hp.pi1} is compatible with the choice  $\pi_1(z)\sim\tfrac{1}{\gamma(z)}$ for any power law $\hat\sigma(u)= u^\nu$, $\nu\in(0,1)$ and $\hat\sigma(u)=\log(u)$.

Model~\Mzero{} generalises the special case $\mathbb{M}(u,c)=\text{const}\cdot (D^2h(u,c))^{-1}$ considered in~\cite{HHMM_2018}. It allows for species-dependent diffusivities, and genuinely contains cross terms in this case. More precisely, for models with species-dependent diffusion coefficients, thermodynamical consistency always leads to cross-diffusion effects, since for a diagonal diffusion matrix $A=\diag(\dots)\in \mathbb{R}^{(1+n)\times(1+n)}$ that is not a multiple of the identity matrix the product $\mathbb{M}=A(D^2h)^{-1}$ is not symmetric. 

In the derivation of the entropy dissipation inequality~\eqref{eq:edin}, we need to assume that 
\begin{align}\label{eq:705}
	m(u,c)\sum_{l=1}^nc_l\lesssim (1+u)^2,
\end{align} 
and have to impose the following conditions mainly serving to avoid issues for small values of $u$ close to zero:
\begin{align}\label{eq:704}\begin{cases}
		m(u,c)|\hat\sigma''(u)|\lesssim 1,
		\\-\hat\sigma''(v)\lesssim -\hat\sigma''(u)+1\quad\text{ for all }v\ge u.
	\end{cases}
\end{align} 

\begin{proposition}[Strong entropy dissipation inequality]\label{prop:edin}
	Let $T^*\in(0,\infty]$.
	Let the hypotheses of Model \Mzero{} hold, and assume locally $\epsilon_0$-H\"older continuous reactions $R\in C^{\epsilon_0}_\loc([0,\infty)^{1+n})$ for some $\epsilon_0\in(0,1)$.
	In addition, assume the bounds~\eqref{eq:705},~\eqref{eq:704}, and suppose that $w_i\in C^2([0,\infty))$ for $i=1,\dots,n$.
	Let $\zin=(\uin,\cin)$, $\zin_i\ge0$,  and $\uin,\hat\sigma_{-}(\uin)\in L^1(\Om)$ and  $\cin_i\in L\log L(\Om)$ for all $i$.
	Let $z=(u,c)$ have non-negative components and suppose that 
	\begin{align}
		&u\in L^\infty_\loc([0,T^*);L^1(\Om)),
		\\&\scrp(z)\in L^1(\Omega_T),\quad a(z)|\nabla u|^2\in L^1(\Omega_T)\quad\text{ for all }T<T^*.
	\end{align}
	If $z$ is a renormalised solution (in the sense of Def.~\ref{def:renorm}) of system~\eqref{eq:erds} in $\Om_{T^*}$ with initial data $\zin$, then the strong entropy dissipation inequality is satisfied, i.e.
	\begin{align}\label{eq:edin.s}\tag{\textsc{ed}.s}
		H(z(t))-H(z(s)) 
		\le	-\int_s^t\!\int_\Om\scrp(z)\,\dd x\dd\tau+ \int_s^t\!\int_{\Om} R_i(z)D_ih(z)\,\dd x\dd \tau
	\end{align}
	for a.e.~$0\le s<t<T^*$,
	and for $s=0$\footnote{With the understanding that $H(z(0))=H(\zin)$.} and a.e.~$t\in(0,T^*)$. In particular, ineq.~\eqref{eq:edin} holds true  for a.e.~$T\in(0,T^*)$. 
\end{proposition}
For the proof of Proposition~\ref{prop:edin}, see Section~\ref{sec:edin}. 
Some comments on generalisations of Proposition~\ref{prop:edin} to other models are provided in Section~\ref{ssec:examples}.

In our final main result, we illustrate that a version of the generalised distance can further be used to prove exponential convergence to equilibrium. Exponential convergence in relative entropy has been studied at a formal level in~\cite{MM_2018} by means of log-Sobolev type inequalities leading to entropy-entropy dissipation estimates (see also~\cite{HHMM_2018}). In contrast to the present approach, \cite{MM_2018} solely relies on the relative entropy, which restricts the results to thermal parts $\sigma$ close to linear; for instance, the choice $\sigma(u)\sim\log u$ related to gas dynamics is only admitted in dimensions $d\le 2$. 
For simplicity, in the following result we disregard the reaction term and refer to~\cite{MM_2018}
for applications on mass action-type reactions.
We will further assume the strong energy inequality  i.e.
\begin{align}
	\En(u(t))-\En(u(s))
	\le -\int_s^t\!\int_\Om a(z)|\nabla u|^2\,\dd x\dd \tau
	-\int_s^t\!\int_\Om m(z)\nabla D_0h(z)\cdot\nabla u\,\dd x\dd \tau\;\quad\qquad
	\label{eq:ene.s}\tag{\textsc{ene}.s}
\end{align}
for a.e.~$0\le s<t<T^*$, and for $s=0$ and a.e.~$t\in(0,T^*)$.

\begin{proposition}[Exponential convergence to equilibrium]\label{prop:expconv}
	Let $\Om$ be smooth, recall that $|\Om|=1$, and let the hypotheses of Model~\Mzero{} hold.
	Let $\zin=(\uin,\cin)$ have non-negative components with $\hat\sigma_{-}(\uin)\in L^1(\Om)$, $\uin\in L^2(\Om)$, $\cin_i\in L\log L$.
	Let $z=(u,c)$ with $u\in L^\infty_\loc([0,\infty),L^1(\Om))$ 
	and $\nabla u\in L^2_\loc([0,\infty),L^2(\Om))$
	be a global-in-time renormalised solution of~\eqref{eq:erds} with $R\equiv0$,  and suppose that \eqref{eq:edin.s} and~\eqref{eq:ene.s} are satisfied (with $T^*=\infty$). 
	
	Then $\int_\Om z_i(t,x)\,\dd x = \int_\Om \zin_i(x)\,\dd x=:\bar z_i$ for all $i\in\{0,\dots,n\}$ and a.a.~$t>0$, and there exist constants $\alpha\in(0,1]$ and $\lambda=\lambda(\bar z,\alpha,\Om)>0$ such that for a.e. $t>0$
	\begin{align}\label{eq:expconvEntropy}
		\oDist_\alpha(z(t),\bar z)\le \mathrm{e}^{-\lambda t}\oDist_\alpha(\zin,\bar z),
	\end{align}
	where $\oDist_\alpha(z,\bar z)=H_\rel(z,\bar z)+\alpha G_\rel(u,\bar u)$ (cf.\ eq.~\eqref{eq:def.Hrel}).
\end{proposition}
See Section~\ref{sec:exp.conv} for the proof of Prop.~\ref{prop:expconv}. 
For non-trivial reactions obeying mass action kinetics, 	
the steady state $\bar z$ associated with~\eqref{eq:erds} is determined by solving a constrained minimisation problem 
for the convex entropy functional $H(u,c)$ imposing energy conservation $E(u,c)=\int_\Om u\overset{!}{=}E_0$ and further linear constraints taking into account possible conservation laws for the concentrations (see~\cite{HHMM_2018,MM_2018}). 
Prop.~\ref{prop:expconv} considers the simplest case, where all species have a conserved mass. Extension to more general reactions is usually achieved by means of suitable coercivity estimates for the dissipation term $-D_ih(z)R_i(z)\ge0$ associated with the reactions.
See~\cite{HHMM_2018,MM_2018} for applications in a non-isothermal setting; for previous works
in the isothermal case, we refer to~\cite{MHM_2015,DFT_2017,Mielke_2017_expDecay} and references therein.
Let us observe that since~\Mzero{} allows for $m\not=0$, leading to energy flux induced by temperature gradients, 
a maximum principle for the internal energy density (as is crucially used in~\cite{HHMM_2018}) is not available here.
We further note that, with a standard Csisz\'ar--Kullback--Pinsker inequality~\cite{Csiszar_1963,UAMT_2000}, the bound~\eqref{eq:expconvEntropy} can be used to deduce exponential convergence to equilibrium in $L^2(\Om)\times (L^1(\Om))^n$. 

\begin{remark}Observe that the condition~\eqref{eq:edin.s} in Prop.~\ref{prop:expconv} is satisfied under the additional hypotheses on the coefficient functions imposed in Prop.~\ref{prop:edin}.
	Under suitable regularity hypotheses, the strong energy inequality~\eqref{eq:ene.s} (with an equality) can be proved similarly as in~\cite[\ustrong]{FHKM_2020}.
\end{remark}

\subsection{Examples}\label{ssec:examples}

Below, we provide relevant applications of the weak-strong uniqueness result. 

\subsubsection{Energy-reaction-diffusion systems}
The hypotheses of Theorem~\ref{thm:wkuniq} are compatible with the class of ERDS introduced in~\cite{FHKM_2020}. 
In that work, the existence of generalised solutions (weak or renormalised) has been established for two classes of models, both taking the form~\eqref{eq:ex.erds} with $m\equiv0$.

One of the models in~\cite{FHKM_2020} is a special case of~\Mzero{} (see page~\pageref{eq:wi.M0}) with $m\equiv0$. A brief verification of the model hypotheses of Theorem~\ref{thm:wkuniq} for~\Mzero{} is provided in the appendix, see Lemma~\ref{l:applic.M0}.
The existence analysis for~\Mzero{} focuses on reactions obeying suitable growth conditions, in which case there are global-in-time weak solutions.
However, renormalised solutions can be constructed 
by adapting the proof of~\cite[\thmRen]{FHKM_2020}, and in this case no growth restrictions on $|R(z)|$ are required. 
Conceptually, the construction of renormalised solutions for~\Mzero{} is simpler than in~\cite[\thmRen]{FHKM_2020}, since the diffusive flux is integrable and
the strong convergence of $\nabla u$ is not required in case of~\Mzero{}.

The second class of models considered in~\cite{FHKM_2020} will here be referred to as~\hypertarget{Mone}{{\bf (M1)}}. 
It again takes the form~\eqref{eq:ex.erds} with $m\equiv0$, and assumes the following conditions:
\begin{itemize}
	\item Entropy density $h$ is given by~\eqref{eq:S0},~\eqref{eq:h2}.
	\item Reactions $R_i\in C([0,\infty)^{1+n})$, $i=1,\dots, n$, satisfy~\ref{eq:hp.Rdiss}, where $R_0\equiv0$.
	\item $\mathbb{M}$ is given by~\eqref{eq:mobility1}--\eqref{eq:ai} with $m(z)\equiv0$, and 
	$m_i(u,c)$ given by~\eqref{eq:ai}, where $\kappa_{0,i},\kappa_{1,i}$ satisfy one of the following: 
	\begin{enumerate}[label=(\roman*)]
		\item\label{it:M1.ren} $\kappa_{0,i}=1$ and $\kappa_{1,i}=0$ for all $i$
		\item\label{it:M1.weak} $\kappa_{0,i}\ge0$ and $\kappa_{1,i}=1$ for all $i$
	\end{enumerate}
	\item Moreover,
	\begin{itemize}[label=$\circ$]
		\item $\pi_1\gamma\gtrsim 1$ \qquad \qquad ($\gamma$ as in~\eqref{eq:def.gamma})
		\item $w_i'(u)\lesssim-w_i''(u)\sqrt{\pi_1}$
		\item $\sqrt{\pi_1}\tfrac{w_i'}{w_i}\lesssim 1$.
	\end{itemize}
\end{itemize}
Global weak solutions in case~\ref{it:M1.weak} have been constructed  in~\cite[\thmWeak]{FHKM_2020} for reactions obeying suitable growth hypotheses.
More interesting is case~\ref{it:M1.ren}, in which the existence of global renormalised solutions has been established in~\cite[\thmRen]{FHKM_2020} for general continuous reactions satisfying~\ref{eq:hp.Rdiss}. A pecularity of this model lies in the circumstance that the renormalised formulation is needed not only to give a meaning to the reaction rates, but also to the diffusion flux $A(z)\nabla z$, which may be non-integrable.
Model~\text{(M1.i)} satisfies conditions~\eqref{eq:hp.ueq1} (with~$\mOns=0$), \eqref{eq:Mnd}, \ref{hp:trunc}, \ref{eq:hp.A00} and~\ref{hp:trunc.grad.flux} of Theorem~\ref{thm:wkuniq}.
We refer to~\cite[\secentropytools.2]{FHKM_2020} (notably the proof of~\fluxbound{} therein), where the necessary estimates can be found.
For verifying~\ref{hp:trunc.grad.flux}, one should also use the fact that the coefficient function $a(z)=\pi_1(z)\gamma(z)$ satisfies the bound
$\chi_{\{u\ge\underline{u}\}}|a(u,c)|\lesssim C(|z|,\underline{u})$ for any $\underline{u}>0$.
Thus, under the extra smoothness assumptions $\hat\sigma, w_i\in C^4((0,\infty))$ and $R_i,m,m_i,\pi_1\in C^{0,1}_\loc((0,\infty)^{1+n})$, Theorem~\ref{thm:wkuniq} is applicable. 
We caution that verifying~\eqref{eq:edin} as it was done in Prop.~\ref{prop:edin} for \Mzero{} is much more delicate for Model (M1.i) due to the possibility of strong cross diffusion caused by a non-integrable diffusion flux. 
Whether or not (M1.i) admits an analogue of Prop.~\ref{prop:edin} is an open question.
Recall, however, that for the solutions constructed in~\cite[Theorem~1.8]{FHKM_2020} inequality~\eqref{eq:edin} can be deduced from the construction using lower semicontinuity arguments (cf.~Section~\ref{ssec:ED.ren}).

\subsubsection{Reaction-cross-diffusion systems}\label{sss:rds}
Our results further apply to (isoenergetic) population models generalising the two-species system for pattern formation by Shigesada, Kawasaki, and Teramoto (SKT). Reduction to the isoenergetic case is achieved by choosing $\mOns=0$ (see~\ref{eq:hp.Mnondeg}) and $u\equiv \uin$ to be spatially constant, which is consistent with the evolution law for the energy density $u$ if $\mOns=0$. Then, the given energy density $u=\uin\in\mathbb{R}^+$ can be regarded as a fixed system's parameter (in particular $\nabla u=0$) and one can write $h=h(c), A=A(c),R=R(c),$ and $\scrp=\scrp(c)$ etc.

The generalised SKT system as considered in~\cite{CDJ_2018} states 
\begin{equation}\label{eq:skt}\tag{\textsc{skt}}
	\begin{array}{rlcl}
		\partial_tc_i&\!\!=\, \nabla\cdot(A_{ij}(c)\nabla c_j)+R_i(c), &&\quad t>0,x\in\Om,\vspace{1mm}
		\\0&\!\!= \, A_{ij}(c)\nabla c_j\cdot\nu,&&\quad t>0,x\in\partial\Om,
	\end{array}
\end{equation}

\noindent where $A_{ij}(c) = \delta_{ij}p_i(c)+c_i\tfrac{\partial p_i}{\partial c_j}(c)$ for $i,j=1,\dots,n$
with $p_i(c)=a_{i0}+\sum_{k=1}^na_{ik}c_k^s$ for suitable $a_{il}\ge0$ and some $s>0$.
Under certain hypotheses, this system has a generalised gradient structure with
entropy density given by $h(c)=\sum_{i=1}^n\pi_i\lambda_s(c_i)$ for constants $\pi_i>0$ and $\lambda_1$ given by~\eqref{eq:deflb}, $\lambda_s(r)=\tfrac{r^s-sr}{s-1}+1$ for $s\not=1$.
Under a weak  cross-diffusion (wc) condition  (see eq.~(12) in~\cite{CDJ_2018}) or the detailed balance (db) condition  $\pi_ia_{ij}=\pi_ja_{ji}$ for all $i,j\in\{1,\dots, n\}$ together with $a_{i0}>0, a_{ii}>0$ for all $i$,
the matrix $\mathbb{M}(z)=A(z)(D^2h(z))^{-1}$ satisfies the non-degeneracy condition~\eqref{eq:Mnd}, i.e.\ $\mathbb{M}(c)\gtrsim_\iota I_n$ whenever $\min\{c_1,\dots, c_n\}\ge \iota$, see the explicit estimate in~\cite[Lemma~2.1]{CJ_2019} for $s=1$, and~\cite[Section~2]{CDJ_2018} for the general case under certain extra hypotheses. 
Observe that the detailed balance condition is equivalent to the symmetry of the mobility matrix $\mathbb{M}=A(D^2h)^{-1}$, and thus ensures the symmetry of the diffusive part $\mathbb{K}_{\mathrm{diff}}$ of the Onsager operator.
\begin{enumerate}[label=\alph*)]
	\item  Linear transition rates $s=1$. In this case, assuming (wc) or (db) with $a_{i0}>0, a_{ii}>0$,  one has 	
	$\scrp(c)\gtrsim\sum_{i=1}^n\big(|\nabla c_i|^2+|\nabla \sqrt{c_i}|^2\big)$ and $|A(c)\nabla c|\lesssim|c||\nabla c|$. 
	Thus, assumptions~\ref{it:hC3},~\ref{eq:hp.Mnondeg}--\ref{hp:grad.flux.control} on the entropy density and the mobility matrix are satisfied. (The more general weights $\pi_i>0$ in $h(c)$ as opposed to the unit weights in~\ref{it:hC3} do not impact the analysis.)
	We leave it to reader to verify that, by adapting the proof of Prop.~\ref{prop:edin} (see also~\cite[Prop.~5]{Fischer_2017}),
	the entropy dissipation inequality~\eqref{eq:edin} can be proved for renormalised solutions to this system, when assuming the regularity 
	\begin{align}\label{eq:reg.rds}
		c_i,\sqrt{c_i}\in L^2_\loc(I;H^1(\Om))\quad\text{for all }i.
	\end{align}
	Observe that this regularity is essentially equivalent to the assumption in Prop.~\ref{prop:edin} that $\scrp(z)\in L^1(\Om_T)$ for all $T<T^*$.
	Thus,   for~\eqref{eq:skt} with $s=1$ and reactions satisfying~\eqref{it:react},  Theorem~\ref{thm:wkuniq} yields the weak-strong uniqueness of renormalised solutions of the regularity~\eqref{eq:reg.rds}.
	A similar result has been obtained previously in~\cite[Theorem~1]{CJ_2019}.
	Here, we should caution that the regularity assumption  $c_i\in L^2_\loc(I;H^1(\Om))$ \emph{and}
	$\sqrt{c_i}\in L^2_\loc(I;H^1(\Om))$ is also needed in the proof of~\cite[Theorem~1]{CJ_2019}, is, however, incompletely stated in this theorem.
	Moreover, our result shows that hypothesis \enquote{(H2.iii)} in~\cite{CJ_2019} on the reactions can be dropped. 
	\item\label{it:gen.s} Nonlinear transition rates. An adjustment of our hypotheses further allows to treat system~\eqref{eq:skt} with superlinear transition rates $s\in(1,2]$.
	In this case, condition~\ref{hp:grad.flux.control} is no longer fulfilled, but with a suitable adjustment of the truncation function $\xi^*(c)$ in the definition of the modified relative entropy density $h^*_\rel(z,\tilde z)$ (given by~\eqref{eq:222}) our technique can still be applied. See Remark~\ref{rem:xi.superlin} for technical details.\smallskip
	\item[] For sublinear transition rates, $s<1$, the problem becomes more delicate since a direct analogue of condition~\ref{hp:grad.flux.control} is not available. When relying exclusively on entropy estimates for a priori bounds, even the construction of renormalised solutions, which to the author's knowledge is currently only available in the case of linear transition rates~\cite{CJ_2019_existence} (but likely to be extendable to the case $s\in[1,2]$), is open for small $s\in(0,1)$.
\end{enumerate}

\section{Weak-strong uniqueness principle}\label{sec:proofs}

In this section we will establish a stability estimate implying Theorem~\ref{thm:wkuniq}.
Throughout this section, we therefore assume the hypotheses of Theorem~\ref{thm:wkuniq}. Before turning to the actual proof in Subsection~\ref{ssec:stab.ineq}, we gather some technical auxiliary results.

\subsection{Truncation function \texorpdfstring{$\xi^*(z)$}{}}\label{ssec:def.xi*}
Let $\iota,B$ be fixed constants, $0<\iota<1<B<\infty$,   that are such that for all $i\in\{0,\dots,n\}$
\begin{equation}\label{eq:strbd}
	2\iota\le \tilde z_i\le B.
\end{equation}

\paragraph{\bf Different regimes/case distinction}

We henceforth abbreviate $|z|_1=\sum_{i=0}^nz_i$.
For sufficiently large but fixed parameters $E,N\ge2$  (to be chosen later) we decompose $[0,\infty)^{1+n}$ into the following three sets: 
\begin{equation}\label{eq:ABC}
	\begin{split}
		\mathcal{A}&=\{z\in[0,\infty)^{1+n}: |z|_1\le E\},
		\\\mathcal{B}&=\{z\in[0,\infty)^{1+n}: E<|z|_1< E^N\},
		\\\mathcal{C}&=\{z\in[0,\infty)^{1+n}: |z|_1\ge E^N\}.
	\end{split}
\end{equation}
A motivation for this decomposition is given in the introduction (see Section~\ref{ssec:technique}). Let us also mention that in the proof of Theorem~\ref{thm:wkuniq} the set $\mathcal{A}$ will be further decomposed into 
$\mathcal{A}_+$ and $\mathcal{A}_0$ (defined in~\eqref{eq:A+A0}).
The parameter $E$ will always be supposed to satisfy $E\ge 2B$. A finite number of further lower bounds on $E$ will be imposed later on. 

\paragraph{\bf Adjusted relative entropy}
We now define
\begin{equation}\label{eq:hrel}
	h_\rel^*(z,\tilde z) = h(z)-D_ih(\tilde z)(\xi^*(z)z_i-\tilde z_i)-h(\tilde z),
\end{equation}
where $\xi^*=\xi^{(E,N)}\in C^\infty([0,\infty)^{1+n})$, $0\le \xi^*\le 1$, is a truncation function subordinate to the above case distinction enjoying the following properties:
\begin{enumerate}[label=\textup{(t$\arabic*$)}]
	\item\qquad\label{it:xi*=1} $\xi^*(z)=1$ and $D^k\xi^*(z)=0$ for all $k\in\mathbb{N}_+$  \qquad if $z\in \mathcal{A}$,
	\item\qquad\label{it:xi*=0} $\xi^*(z)=0$  and $D^k\xi^*(z)=0$ for all $k\in\mathbb{N}_+$ \qquad if $z\in \mathcal{C}$,
\end{enumerate}
and 
\begin{enumerate}[label=\textup{(t$\arabic*$)},resume]
	\item\qquad \label{eq:Dxi*} $|D\xi^*(z)|\lesssim \tfrac{1}{N|z|_1},\quad |D^2\xi^*(z)|\lesssim \tfrac{1}{N(|z|_1)^2}$ 
	\qquad \;for all $z\in[0,\infty)^{1+n}.$
\end{enumerate}  
A function $\xi^*$ that has these properties can be obtained as follows:
let $\vartheta\in C^\infty_b(\mathbb{R})$ be non-increasing with $\vartheta(r)=1$ for $r\le0$ and $\vartheta(r)=0$ for $r\ge 1$ and define (cf.~\cite{Fischer_2017})
\begin{equation}\label{eq:phi}
	\xi^*(z)=\xi^{(E,N)}(z)=\vartheta\left(\frac{\log(|z|_1)-\log(E)}{\log(E^N)-\log(E)}\right).
\end{equation}
It is elementary to check that this choice satisfies the above properties. For instance, note that
$$|D\xi^{(E,N)}(z)|\lesssim\tfrac{1}{N|z|_1}\cdot\tfrac{1}{\log(E)}$$
if $N\ge 2$.

\begin{remark}[Superlinear transition rates]\label{rem:xi.superlin}
	When dealing with (isoenergetic) reaction-cross-diffusion population systems with \emph{superlinear} transition rates $s\in(1,2]$ for concentrations $c_1,\dots, c_n$ (see Example~\ref{it:gen.s} in Sec.~\ref{sss:rds}),
	the decay property~\ref{eq:Dxi*} is no longer sufficient, and 
	the above choice of $\xi^*$ should be replaced by 
	\begin{equation}\label{eq:phi.s>1}
		\xi^*(c):=\xi^{(E,N,s)}(c):=\vartheta\left(\frac{\log(\rho_s^s(c))-\log(E^s)}{\log(E^{sN})-\log(E^s)}\right) = \vartheta\left(\frac{\log(\rho_s(c))-\log(E)}{\log(E^{N})-\log(E)}\right),
	\end{equation}
	where $\vartheta$ is as before, and 
	\begin{align}
		\rho_s(c):=\Big(\sum_{i=1}^n(c_i+\delta)^s\Big)^\frac{1}{s}
	\end{align}
	with $\delta=\delta(s)=1$ for $s\in (1,2)$ and $\delta(s)=0$ for $s=2$.
	The parameter $\delta\in\{0,1\}$ ensures the smoothness of $\rho_s^s$ and thus of $\xi^*$ .
	
	Introducing the regimes $\mathcal{A}^{(s)}=\{\rho_s(c)\le E\},$ $\mathcal{B}^{(s)}=\{E<\rho_s(c)< E^N\}$, $\mathcal{C}^{(s)}=\{\rho_s(c)\ge E^N\}$,
	properties~\ref{it:xi*=1} and~\ref{it:xi*=0} remain valid with $\mathcal{A},$ $\mathcal{C}$  replaced by $\mathcal{A}^{(s)}$ and $\mathcal{C}^{(s)}$, respectively.
	Derivatives of $\xi^*$  given by~\eqref{eq:phi.s>1} enjoy the following decay properties,
	specifically adapted to the problem at hand,
	\begin{equation}\label{eq:Dxi*.s>1}
		|D_i\xi^*(c)|\lesssim \tfrac{(c_i+\delta)^{s-1}}{N|c|_s^s},\qquad |D_{ik}\xi^*(c)|\lesssim \tfrac{(c_i+\delta)^{s-1}(c_k+\delta)^{s-1}}{N|c|_s^{2s}}+\tfrac{(c_i+\delta)^{s-2}}{N|c|_s^s}\delta_{ik},
	\end{equation}
	where $\delta_{ik}$ denotes the Kronecker delta.
	As will become clear in Section~\ref{ssec:stab.ineq}, if $c\in \mathcal{A}^{(s)}$, one can simply follow the reasoning in the proof of Theorem~\ref{thm:wkuniq} using the fact that~\eqref{eq:skt} satisfies (under reasonable hypotheses) a non-degeneracy condition analogous to~\eqref{eq:Mnd}.
	
	If $c\in \mathcal{B}^{(s)}$, we need an analogue of assumption~\ref{hp:grad.flux.control} to be able to absorb terms without a good sign involving products of gradients of the renormalised solution by the entropy dissipation.
	Typical systems~\eqref{eq:skt} satisfy the coercivity bound (cf.~\cite{CDJ_2018})
	\begin{align}\label{eq:221}
		\scrp(c)\gtrsim\sum_{i=1}^n|\nabla c_i^s|^2+\sum_{i=1}^n|\nabla c_i^{s/2}|^2.
	\end{align}
	Confining to systems~\eqref{eq:skt} enjoying this bound, a suitable generalisation of hp.~\ref{hp:grad.flux.control} (in the isoenergetic case) that is satisfied by such systems is
	\begin{align}\label{eq:sqgrad.skt}
		\chi_{\{|c|_s\ge 1\}}\cdot	\big(c_i^{s-1}+c_i^{s/2	-1}\big)|A_{ij}(c)\nabla c_j|\lesssim |c|^{s}_s\sqrt{\scrp(c)}.
	\end{align}
	Observe that since $s\in(1,2]$, the factor $\big(c_i^{s-1}+c_i^{s/2-1}\big)$ in this condition can be equivalently replaced by $\big((c_i+\delta)^{s-1}+c_i^{s/2-1}\big)$.
	Using the above model bounds and decay properties of derivatives of $\xi^*$, one can verify the estimate
	\begin{equation}\label{eq:223}
		\begin{split}
			\chi_{\{|c|_s\ge 1\}}| \nabla D_i(\xi^*(c)c_l)||A_{ij}(c)\nabla c_j|
			\lesssim \tfrac{1}{N}\scrp(c),
		\end{split}
	\end{equation}
	which allows to deal with the case $c\in \mathcal{B}^{(s)}$ (cf.\ ineq.~\eqref{eq:ineq.B} in the proof of Thm~\ref{thm:wkuniq}).
	When verifying this bound, one uses the fact that 
	$|\nabla c_l|\lesssim|\nabla c_l^s|+|\nabla c_l^{s/2}|$, which holds true since $s\in(1,2]$.
	
	If $c\in \mathcal{C}^{(s)}$, one can follow the reasoning in the proof of Thm~\ref{thm:wkuniq}.
	
	Thus, for models with superlinear transition rates $s\in(1,2]$ that satisfy~\eqref{eq:221}, weak-strong uniqueness is obtained by adapting the proof of Theorem~\ref{thm:wkuniq} as sketched above. Here, we also use the fact that thanks to the locally uniform convexity of the entropy density associated with~\eqref{eq:skt} (see Sec.~\ref{sss:rds}) an analogue of Prop.~\ref{prop:coerc.hrel} is immediate.
\end{remark}

\subsection{Coercivity properties of the generalised distance}

We henceforth let $\en(s)=\tfrac{1}{2}s^2$,
\begin{equation}\label{eq:defgrel}
	\en_\rel(u,\tilde u):=\en(u)-\en'(\tilde u)(u-\tilde u)-\en(\tilde u)=\tfrac{1}{2}|u-\tilde u|^2,
\end{equation}
and define for $\alpha\in(0,\infty)$
\begin{equation}\label{eq:defdist}
	\dist_\alpha(z,\tilde z)=h_\rel^*(z,\tilde z)+\alpha\en_\rel(u,\tilde u).
\end{equation}

For the following assertion we recall that $z\in\mathcal{A}$ if and only if $\sum_{i=0}^nz_i\le E$. 
\begin{proposition}[Coercivity properties]\label{prop:coerc.hrel}
	Recall that $h(z)=h(u,c)$ is given by~\eqref{eq:S0},~\eqref{eq:h2}, and that $\tilde z$ satisfies~\eqref{eq:strbd}.
	For any $E\in[1,\infty)$, we have\footnote{We recall that any dependence of estimates on $\iota$ and $B$ (see~\eqref{eq:strbd}) will usually not be indicated explicitly.}
	\begin{equation}\label{eq:coercA}
		\begin{split}
			|z-\tilde z|^2\lesssim_E h_\rel^*(z,\tilde z)\quad \text{if }z\in\mathcal{A}.
		\end{split}
	\end{equation}
	There exists $\underline E=\underline E(\tilde z)<\infty$
	such that for any $E\ge\underline{E}$ and any $\alpha\ge1$
	\begin{equation}\label{eq:coercAc}
		\begin{split}
			h_\rel^*(z,\tilde z)+\alpha\en_\rel(u,\tilde u)\ge  \epsilon\bigg(\sum_{i=1}^n c_i\log_+ c_i+ u^2\bigg)+1
			\quad \text{if }z\in\mathcal{A}^c,
		\end{split}
	\end{equation}
	where $\epsilon>0$ is a positive constant only depending on model parameters.
	
	For any $\alpha\in(0,1]$ there exists $\underline E=\underline E(\alpha,\tilde z)<\infty$
	such that for any $E\ge\underline{E}$
	\begin{equation}\label{eq:coercAc01}
		\begin{split}
			h_\rel^*(z,\tilde z)+\alpha\en_\rel(u,\tilde u)\ge  \epsilon\bigg(\sum_{i=1}^n c_i\log_+ c_i+\alpha\, u^2\bigg)+1
			\quad \text{if }z\in\mathcal{A}^c,
		\end{split}
	\end{equation}
	where $\epsilon>0$ is a positive constant only depending on model parameters.
\end{proposition}

\begin{remark}
	Note that as long as  $\alpha\in[1,\infty)$, the lower bound $\underline E$ can be chosen independently of $\alpha$. This property is essential for obtaining the stability estimate~\eqref{eq:dist.ineq} in the case where~\ref{hp:grad.flux.control} is not satisfied (including model~\Mone{} in case~\ref{it:M1.ren}, where $\kappa_{1,i}=0$; cf.\ Section~\ref{ssec:examples}).
\end{remark}

\begin{proof}[Proof of Prop.~\ref{prop:coerc.hrel}]
	For the first assertion, we note that, as can be seen from the proof of~\cite[Prop.~2.1]{MM_2018}, the entropy density $h$ is locally uniformly convex
	on $[0,\infty)^{1+n}$ with 
	$D_{ij}h(z)\ge \epsilon_0(E)\delta_{ij}$ if $|z|_1\le E$.
	Since, by construction,  $\xi^*(z)=1$ whenever $z\in\mathcal{A}$, we thus infer $$h_\rel^*(z,\tilde z)=h(z)-D_ih(\tilde z)(z-\tilde z)_i-h(\tilde z)\gtrsim_E|z-\tilde z|^2.$$
	
	Let us now turn to assertions~\eqref{eq:coercAc} and~\eqref{eq:coercAc01}. Since $\tilde u\le B$, we have the bound $|u-\tilde u|^2\ge \tfrac{1}{4}u^2\chi_{\{u\ge 2B\}}$. By Lemma~\ref{l:S},  
	we further have for some $\nu\in[0,1)$ and some positive constant $\epsilon>0$
	\begin{equation*}
		h(z)\ge \epsilon\sum_{i=1}^n c_i\log_+ c_i-C(\hat\sigma_+(u)+u^\nu)- C.
	\end{equation*}
	Hence, 
	\begin{align}\label{eq:103}
		h_\rel^*(z,\tilde z)+\alpha g_\rel(u,\tilde u)
		&\ge \epsilon\sum_{i=1}^n c_i\log_+ c_i+\alpha u^2\chi_{\{u\ge 2B\}}-C\sum_{i=0}^n z_i-C.
	\end{align}
	Inequality~\eqref{eq:coercAc} is now immediate since $\sum_{i=1}^nc_i\log(c_i)+u^2$ dominates $\sum_{0=1}^nz_i$ for $|z|_1\ge \underline{E}$ whenever $\underline{E}$ is large enough. Observe that the lower bound on $\underline E$ can be chosen independently of $\alpha\in[1,\infty)$.
	
	Let now $\alpha\in(0,1]$ be given. Inequality~\eqref{eq:103} shows that if $\underline{E}=\underline{E}(\alpha)$ is large enough, we obtain~\eqref{eq:coercAc01} if $E\ge\underline{E}$.
\end{proof}

\subsection{Evolution inequality for the  generalised distance}
We will abbreviate 
\begin{align}
	H_\rel^*(z,\tilde z):=\int_\Om \,h_\rel^*(z,\tilde z)\,\dd x,\qquad
	\En_\rel(u,\tilde u):=\int_\Om\en_\rel(u,\tilde u)\,\dd x.
\end{align}
In this subsection, we exploit  the evolution laws satisfied by the
dissipative  renormalised solution $z$ and the strong solution $\tilde z$ of Theorem~\ref{thm:wkuniq} to derive an evolution  inequality for the generalised distance 
\begin{equation}\label{eq:600}
	\Dist_\alpha(z,\tilde z) := \int_\Om\dist_\alpha(z,\tilde z)
	\,\dd x= H_\rel^*(z,\tilde z)+\alpha\,\En_\rel(u,\tilde u),
\end{equation}
where $\alpha\in(0,\infty)$ is a suitably chosen weight (to be specified in Section~\ref{ssec:stab.ineq}).

We recall that $I:=[0,T^*)$. Furthermore,
we note that since $\partial_t\tilde z\in L^\infty(\Om_T)$ for any $T<T^*$, we can integrate by parts with respect to time in the weak formulation~\eqref{eq:weakform} satisfied by the strong solution~$\tilde z$ to find \begin{equation}\label{eq:weakformStrong}
	\begin{split}
		\int_0^T\!\!\int_\Om \partial_t\tilde z_i\psi_i\,\dd x\dd t
		= -\int_0^T\!\!\int_\Om A_{ij}(\tilde z)\nabla \tilde z_j\cdot\nabla\psi_i\,\dd x\dd t+\int_0^T\!\!\int_\Om R_i(\tilde z)\psi_i\,\dd x\dd t.
	\end{split}
\end{equation}
By a density argument, one can see that eq.~\eqref{eq:weakformStrong} holds true for all $\psi\in L^1_\loc(I;W^{1,1}(\Om))$.

We consider separately the two quantities $H_\rel^*(z,\tilde z)$ and $\En_\rel(u,\tilde u)$ appearing in~\eqref{eq:600}, beginning with the former.
\begin{lemma}[Evolution of the entropic part]\label{l:evol.hrel}
	For a.e.\ $T<T^*$ one has 
	\begin{equation}\label{eq:117}
		\begin{split}
			\eval{H_\rel^*(z,\tilde z)}_{t=0}^{t=T} &\le  \int_0^T\!\!\int_\Om \;\rho^{(h)}\;\dd x\dd t,
		\end{split}
	\end{equation}
	where 
	\begin{equation}\label{eq:126'b}
		\begin{split}
			\rho^{(h)} &:=-\nabla D_{i}h(z)\cdot\mathbb{M}_{il}(z)\nabla D_{l}h(z)
			\\&\qquad + \nabla (D_i(\xi^*(z)z_j)D_jh(\tilde z) )\cdot \mathbb{M}_{il}(z)\nabla D_lh(z)
			\\&\qquad + \nabla\Big(D_{ij}h(\tilde z)(\xi^*(z)z_j-\tilde z_j)\Big)\cdot 
			\mathbb{M}_{il}(\tilde z)\nabla D_lh(\tilde z)
			\\&\qquad- D_{ij}h(\tilde z)(\xi^*(z)z_j-{\tilde z}_j) R_i(\tilde z)
			\\&\qquad+ (D_ih(z)-D_jh(\tilde z)D_i(\xi^*(z)z_j)) R_i(z).
		\end{split}
	\end{equation}
	
\end{lemma}

\begin{proof}[Proof of Lemma~\ref{l:evol.hrel}]
	The subsequent observations apply to a.e.\ $T<T^*$.
	
	We write 
	\begin{equation*}
		H_\rel^*(z,\tilde z) = H(z) - \int_\Om D_ih(\tilde z)\xi^*(z)z_i\,\dd x 
		+\int_\Om (D_ih(\tilde z)\tilde z_i -h(\tilde z))\,\dd x.
	\end{equation*}
	For the first term on the RHS we use the fact that, by hypothesis, $z$ satisfies~\eqref{eq:edin}, i.e.\ 
	\begin{equation*}\begin{split}
			\eval{H(z)}_{t=0}^{t=T} &\le  -\int_0^T\!\!\int_\Om \nabla D_{i}h(z)\cdot\mathbb{M}_{il}(z)\nabla D_{l}h(z)\,\dd x\dd t
			+\int_0^T\!\!\int_\Om D_ih(z) R_i(z)\,\dd x\dd t.
		\end{split}
	\end{equation*}
	For the second term, we want to use the fact that $z$ satisfies the renormalised formulations~\eqref{eq:118} and~\eqref{eq:118..} with the truncation function $\xi(z)=\xi^*(z)z_j$ and the test function $\psi=D_jh(\tilde z)\in W^{1,\infty}(\Om_T)$. (For the admissibility of this choice, see Remark~\ref{rem:testfk}.)
	Inserting these choices in eq.~\eqref{eq:118..}, we obtain
	\begin{equation}\label{eq:119}
		\begin{split}
			-\eval{\int_\Om D_jh(\tilde z)\xi^*(z)z_j\,\dd x}_{t=0}^{t=T}&+\int_0^T\!\!\int_\Om \xi^*(z)z_j\tfrac{\dd}{\dd t}D_jh(\tilde z)\,\dd x\dd t
			\\& =
			\int_0^T\!\!\int_\Om \nabla (D_i(\xi^*(z)z_j)D_jh(\tilde z) )\cdot\mathbb{M}_{il}(z)\nabla D_lh(z)\,\dd x\dd t 
			\\&\quad-\int_0^T\!\!\int_\Om D_i(\xi^*(z)z_j)R_i(z)D_jh(\tilde z)\,\dd x\dd t.
		\end{split}
	\end{equation}
	We next rewrite the second term on the LHS choosing in the weak form~\eqref{eq:weakformStrong} for $\tilde z$ the test function $\psi:=D_{ij}h(\tilde z)\xi^*(z)z_j$.
	This yields
	\begin{equation}\label{eq:119t}
		\begin{split}
			\int_0^T\!\!\int_\Om D_{ij}h(\tilde z)\xi^*(z)z_j\partial_t{\tilde z}_i\,\dd x\dd t
			&=-\int_0^T\!\!\int_\Om \nabla\big(D_{ij}h(\tilde z)\xi^*(z)z_j)\cdot 
			\mathbb{M}_{il}(\tilde z)\nabla D_lh(\tilde z)\,\dd x\dd t
			\\&\qquad +\int_0^T\!\!\int_\Om D_{ij}h(\tilde z)\xi^*(z)z_j R_i(\tilde z)\,\dd x\dd t.
		\end{split}
	\end{equation}
	Observe that  
	since $\chi_{\{|z|\le E\}}\nabla z\in L^2_\loc(I;L^2(\Om))$ for any $E<\infty$, the function $\psi:=D_{ij}h(\tilde z)\xi^*(z)z_j\in L^2_\loc(I;H^1(\Om))$ is indeed admissible in the weak equation~\eqref{eq:weakformStrong} for~$\tilde z$.
	
	We finally need to determine the evolution of the term
	\begin{equation}\label{eq:120}
		\begin{split}
			\int_\Om (D_ih(\tilde z)\tilde z_i -h(\tilde z))\,\dd x.
		\end{split}
	\end{equation}
	To this end, note that thanks to the regularity of $\tilde z$, 
	\begin{equation}\label{eq:121}
		\begin{split}
			\eval{\int_\Om h(\tilde z)\,\dd x}_{t=0}^{t=T} &= \int_0^T\!\!\int_\Om D_ih(\tilde z)\partial_t{\tilde z}_i\,\dd x\dd t,
			\\\eval{\int_\Om D_ih(\tilde z)\tilde z_i\,\dd x}_{t=0}^{t=T} &= \int_0^T\!\!\int_\Om D_ih(\tilde z)\partial_t{\tilde z}_i\,\dd x\dd t
			+\int_0^T\!\!\int_\Om \big(\tfrac{\dd}{\dd t}D_ih(\tilde z)\big){\tilde z}_i\,\dd x\dd t.
		\end{split}
	\end{equation}
	Subtracting the first from the second equality then yields
	\begin{align}\label{eq:122}
		\eval{\int_\Om \big(D_ih(\tilde z)\tilde z_i-h(\tilde z)\big)\,\dd x}_{t=0}^{t=T} &
		=\int_0^T\!\!\int_\Om D_{ij}h(\tilde z){\tilde z}_i\partial_t {\tilde z}_j\,\dd x\dd t
		\\&=\int_0^T\!\!\int_\Om D_{ij}h(\tilde z){\tilde z}_j\partial_t {\tilde z}_i\,\dd x\dd t
		\\&=-\int_0^T\!\!\int_\Om \nabla\big(D_{ij}h(\tilde z){\tilde z}_j\big)\cdot\mathbb{M}_{il}(\tilde z)\nabla D_lh(\tilde z)\,\dd x\dd t
		\\&\qquad+\int_0^T\!\!\int_\Om D_{ij}h(\tilde z){\tilde z}_j R_i(\tilde z)\,\dd x\dd t,
	\end{align}
	where in the second step we have used the symmetry of the Hessian of $h$.
	
	In combination, the above estimates yield the bound
	\begin{align}\label{eq:123}
		\eval{H_\rel^*(z,\tilde z)}_{t=0}^{t=T} &\le  -\int_0^T\!\!\int_\Om \nabla D_{i}h(z)\cdot\mathbb{M}_{il}(z)\nabla D_{l}h(z)\,\dd x \dd t
		\\&\qquad+\int_0^T\!\!\int_\Om D_ih(z) R_i(z)\,\dd x\dd t
		\\&\qquad +\int_0^T\!\!\int_\Om \nabla (D_i(\xi^*(z)z_j)D_jh(\tilde z) )\cdot M_{il}(z)\nabla D_lh(z)\,\dd x\dd t 
		\\&\qquad -\int_0^T\!\!\int_\Om D_i(\xi^*(z)z_j)R_i(z)D_jh(\tilde z)\,\dd x\dd t
		\\&\qquad +\int_0^T\!\!\int_\Om \nabla\big(D_{ij}h(\tilde z)\xi^*(z)z_j)\cdot 
		\mathbb{M}_{il}(\tilde z)\nabla D_lh(\tilde z)\,\dd x\dd t
		\\&\qquad -\int_0^T\!\!\int_\Om D_{ij}h(\tilde z)\xi^*(z)z_j R_i(\tilde z)\,\dd x\dd t
		\\&\qquad-\int_0^T\!\!\int_\Om \nabla\big(D_{ij}h(\tilde z){\tilde z}_j\big)\cdot\mathbb{M}_{il}(\tilde z)\nabla D_lh(\tilde z)\,\dd x\dd t
		\\&\qquad+\int_0^T\!\!\int_\Om D_{ij}h(\tilde z){\tilde z}_j R_i(\tilde z)\,\dd x\dd t.
	\end{align}
	The asserted inequality is now obtained upon rearranging the integrals on the RHS.
\end{proof}

We next turn to the energetic part. We first note that equation~\eqref{eq:weakformStrong} and the fact that $R_0\equiv0$ imply that 
\begin{equation}\label{eq:u.strong}
	\begin{split}
		\int_0^T\!\!\int_\Om \partial_t\tilde u\varphi\,\dd x\dd t
		= -\int_0^T\!\!\int_\Om A_{0j}(\tilde z)\nabla \tilde z_j\cdot\nabla\varphi\,\dd x\dd t
	\end{split}
\end{equation}
for all $\varphi\in L^1_\loc(I;W^{1,1}(\Om))$.

\begin{lemma}[Evolution of the energetic part]\label{l:evol.grel}
	Recall the definition of $g_\rel$ in~\eqref{eq:defgrel} and the notation $G_\rel(u,\tilde u):=\int_\Om g_\rel(u,\tilde u)\,\dd x$.
	For almost every $T>0$, we have
	\begin{equation}\label{eq:500}
		\begin{split}
			\eval{\En_\rel(u,\tilde u)}_{t=0}^{t=T} = \int_0^T\!\!\int_\Om \;\rho^{\en}\;\dd x\dd t,
		\end{split}
	\end{equation}
	where
	\begin{equation}\label{eq:126''b}
		\begin{split}
			\rho^{(\en)}:= -a(z)&|\nabla u-\nabla \tilde u|^2
			\\&-(a(z)-a(\tilde z))(\nabla u-\nabla \tilde u)\cdot\nabla \tilde u
			\\&-m(z)(\nabla u-\nabla \tilde u)\cdot(\nabla D_0h(z)-\nabla  D_0h(\tilde z))
			\\&-(m(z)-m(\tilde z))(\nabla u-\nabla \tilde u)\cdot\nabla  D_0h(\tilde z).
		\end{split}
	\end{equation}
\end{lemma}
\begin{proof}
	We expand $\en_\rel(u,\tilde u)=\tfrac{1}{2}u^2-u\tilde u+\tfrac{1}{2}\tilde u^2$.
	
	To deal with the first term on the RHS, we use the energy inequality~\eqref{eq:ene}, i.e.\ the property that for a.e.\ $T<T^*$ 
	\begin{equation*}
		\eval{\En(u)}_{t=0}^{t=T} \le -\int_0^T\!\!\int_\Om a(z)|\nabla u|^2\,\dd x\dd t
		-\int_0^T\!\!\int_\Om m(z)\nabla D_0h(z)\cdot\nabla u\,\dd x\dd t.
	\end{equation*}
	
	To determine the time evolution of the term $\int_\Om u \tilde u\,\dd x$, we assert that  the Lipschitz function $\tilde u$ is admissible in the weak formulation~\eqref{eq:u.weak} of the equation for $u$, thus yielding
	\begin{equation*}
		\begin{split}
			-\eval{\int_\Om  u \tilde u\,\dd x}_{t=0}^{t=T}&+\int_0^T\!\!\int_\Om  u\partial_t\tilde u\,\dd x\dd t
			\\&=\int_0^T\!\!\int_\Om \big(a(z)\nabla u\cdot\nabla\tilde u+m(z)\nabla D_0h(z)\cdot\nabla \tilde u\big)\,\dd x\dd t,
		\end{split}
	\end{equation*}	
	where we used~\eqref{eq:hp.ueq}.
	The admissibility of $\tilde u$ can be shown as follows: first exploit the regularity properties of  $\nabla u\in L^2_\loc(I;L^2(\Om))$, $u\in L^\infty_\loc(I;L^2(\Om))$, which hold true by hypothesis resp.\ follow from~\eqref{eq:ene} and the fact that $\scrp(z)\in L^1_\loc(I;L^1(\Om))$.
	Assumption~\ref{eq:hp.A00} and the Gagliardo--Nirenberg interpolation applied to $u$ as well as the estimate $$m|\nabla D_0h(z)|\le \sqrt{m}\sqrt{\scrp(z)}\lesssim \sqrt{\scrp(z)}$$ then imply improved integrability of the flux term, namely  for some $s=s(d)>1$ $$A_{0j}(z)\nabla z_j=a(z)\nabla u+m(z)\nabla D_0h(z)\in  L^s_\loc(I;L^s(\Om)).$$ With these bounds one can now use an approximation argument to show that, under the current hypotheses, eq.~\eqref{eq:u.weak} can be extended in particular to Lipschitz functions $\varphi\in C^{0,1}(I\times\bar\Om)$.
	
	Finally, using the test function $\varphi=u-\tilde u \in L^1_\loc(I;W^{1,1}(\Om))$ in the weak equation~\eqref{eq:u.strong} for $\tilde u$ gives
	\begin{equation*}
		-\int_0^T\!\!\int_\Om \partial_t\tilde u(u-\tilde u)\,\dd x\dd t = 
		\int_0^T\!\!\int_\Om a(\tilde z)\nabla\tilde u\cdot\nabla(u-\tilde u)\,\dd x\dd t.
	\end{equation*}
	
	The asserted identity~\eqref{eq:500} is now obtained by adding up the above equations and rearranging  appropriately the terms on the RHS.
\end{proof}       
The evolution inequality for our generalised distance is an immediate consequence of the previous two propositions. 
\begin{corollary}\label{cor:distalpha}
	Let $\alpha\in(0,\infty)$. We have 
	\begin{equation}\label{eq:evol.dist}
		\begin{split}
			\eval{\Dist_\alpha(z,\tilde z)}_{t=0}^{t=T} \le 
			\int_0^T\!\!\int_\Om \;\rho_\alpha\;\dd x \dd t,
		\end{split}
	\end{equation}
	where $\rho_\alpha:=\rho^{(h)}+\alpha\rho^{(\en)}$ with $\rho^{(h)}$, $\rho^{(\en)}$
	given by~\eqref{eq:126'b} resp.~\eqref{eq:126''b}.
\end{corollary}

\subsection{Stability estimate}\label{ssec:stab.ineq}

\begin{proof}[Proof of Theorem~\ref{thm:wkuniq}]	
	Since \ref{hp:grad.flux.control} implies \ref{hp:trunc.grad.flux}, it suffices to prove the assertion for the case \enquote{$\mOns=1$ (and thus, by hp.,  \ref{hp:grad.flux.control})} and the case \enquote{$\mOns=0$ and \ref{hp:trunc.grad.flux}}, henceforth referred to as Case~$\mOns=1$ resp.~Case~$\mOns=0$.
	
	We will show the following.
	\begin{itemize}
		\item Case~$\mOns=1:$\; 
		if $\alpha\in(0,1]$ is sufficiently small, and if $E=E(\alpha)$ and $N=N(E)$ are large enough, then for almost all $T\in(0,T^*)$
		\begin{equation}\label{eq:Gronwall}
			\eval{\Dist_\alpha(z(t,\cdot),\tilde z(t,\cdot))}_{t=0}^{t=T}\lesssim_{E,N,\alpha} 
			\int_0^T\Dist_\alpha(z,\tilde z)\,\dd t.
		\end{equation}
		\item Case~$\mOns=0:$\; 
		if $E$, $N=N(E)$ and $\alpha\in[1,\infty)$ are chosen large enough ($\alpha$ possibly depending on $E,N$), then for a.a. $T\in(0,T^*)$ ineq.~\eqref{eq:Gronwall} holds true.
	\end{itemize}
	Once ineq.~\eqref{eq:Gronwall} has been established, we can invoke Gronwall's inequality to infer that for a.e.\ $T\in(0,T^*)$
	\begin{align}\label{eq:Gronwall.1}
		\Dist_\alpha(z(T,\cdot),\tilde z(T,\cdot))\le \Dist_\alpha(\zin,\widetilde{\zin})\exp(kT),
	\end{align}
	where $k=k(E,N,\alpha)>0.$ The estimates will also depend on the fixed constant $\iota>0$, i.e.\ on the pointwise lower bound for $\min(\tilde z)$. This dependency will only be indicated occasionally and for the sake of clarity.
	
	In view of inequality~\eqref{eq:evol.dist} it suffices to show the pointwise bound
	\begin{align}\label{eq:601}
		\rho_\alpha\;\lesssim_{E,N,\alpha}\;\dist_\alpha(z,\tilde z).
	\end{align}
	An elementary ingredient in the proof of this bound will be the coercivity properties of $\dist_\alpha$ (see Prop.~\ref{prop:coerc.hrel}). We anticipate that referring to Prop.~\ref{prop:coerc.hrel} will be the only instance, where the present proof makes use of the more specific form of the entropy density $h(u,c)$ assumed in~\ref{it:hC3}. Loosely speaking, besides the locally strict convexity ensuring~\eqref{eq:coercA}, we will rely on a lower bound on the generalised distance of the form $\dist_\alpha(z,\tilde z)\gtrsim 1+u^2$ for $|z|\gg1$.
	
	We will distinguish the cases $\mathcal{A}, \mathcal{B}$ and $\mathcal{C}$ introduced on page~\pageref{eq:ABC}, where, owing to the degeneracies of $\mathbb{M}(z)$ occurring when one of the concentrations vanishes, we further decompose the set~$\mathcal{A}$ into 
	\begin{equation}\label{eq:A+A0}
		\mathcal{A}_+:=\{z':\min(z')\ge \iota\}\cap\mathcal{A} \qquad\text{and}\qquad 
		\mathcal{A}_0:=\{z':\min(z')< \iota\}\cap\mathcal{A},
	\end{equation}
	where $\min(z'):=\min\{z_0',\dots,z_n'\}$ for $z'=(z_0',\dots,z_n')\in[0,\infty)^{1+n}$. This decomposition further serves to avoid regularity issues of $h$ as $z_i\searrow0$ for some $i\in\{0,\dots,n\}$.
	
	If $z\in (\mathcal{A}_+)^c$, we will make use of the following equivalent formula for $\rho^{(\en)}$
	\begin{equation}\label{eq:rho.g.far}
		\begin{split}
			\rho^{(\en)}= -a(z)|\nabla u|^2&-a(\tilde z)|\nabla \tilde u|^2+a(z)\nabla u\cdot\nabla \tilde u+a(\tilde z)\nabla u\cdot\nabla \tilde u
			\\&-m(z)(\nabla u-\nabla \tilde u)\cdot\nabla D_0h(z)
			\\&+m(\tilde z)(\nabla u-\nabla \tilde u)\cdot\nabla  D_0h(\tilde z).
		\end{split}
	\end{equation}
	Since, by hypothesis, $0\le m(z)\lesssim \mOns$ and $a(z)\gtrsim1$, this implies that
	\begin{align}\label{eq:rho.g5}
		\rho^{(\en)}&\le -\tfrac{a(z)}{2}|\nabla u|^2+C|a(z)\nabla u|+Cm(z)\mOns|\nabla D_0h(z)|^2+C.
	\end{align}
	Using this form in the case when $z\in (\mathcal{A}_+)^c$,  we can avoid for instance issues due to $a(z)$ becoming singular as $u\to0$ by using the bound~\ref{eq:hp.A00} on the energy flux.
	
	\smallskip
	
	Finally, note that $z\in\mathcal{A}$ implies $\xi^*(z)=1$ and $D^k\xi^*(z)=0$ $\forall k\in\mathbb{N}_+$, so that, if $z\in\mathcal{A}$,  one has  by formula~\eqref{eq:126'b}
	\begin{equation}\label{eq:126A'}
		\begin{split}
			\rho^{(h)} &=-\nabla D_{i}h(z)\cdot\mathbb{M}_{il}(z)\nabla D_{l}h(z)
			\\&\qquad + \nabla D_ih(\tilde z)\cdot M_{il}(z)\nabla D_lh(z)
			\\&\qquad + \nabla\Big(D_{ij}h(\tilde z)(z_j-\tilde z_j)\Big)\cdot 
			\mathbb{M}_{il}(\tilde z)\nabla D_lh(\tilde z)\,\dd x\dd t
			\\&\qquad- D_{ij}h(\tilde z)(z_j-{\tilde z}_j) R_i(\tilde z)
			\\&\qquad+ (D_ih(z)-D_ih(\tilde z)) R_i(z).
		\end{split}
	\end{equation}
	We are now ready to tackle the four cases.
	
	\smallskip
	
	\underline{Case $z\in \mathcal{A}_+:$}
	in this case we have the control $\iota\le z_i\le E$ for all $i\in\{0,\dots,n\}$, and we need to show that
	$\rho_\alpha\lesssim|z-\tilde z|^2$.   
	We therefore rewrite formula~\eqref{eq:126A'} as 
	\begin{equation}\label{eq:126A.'}
		\begin{split}
			\rho^{(h)} &=-\nabla (D_{i}h(z)-D_ih(\tilde z))\cdot\mathbb{M}_{il}(z)\nabla D_{l}h(z)
			\\&\qquad + \nabla\big(D_{ij}h(\tilde z)(z_j-\tilde z_j)\big)\cdot 
			\mathbb{M}_{il}(\tilde z)\nabla D_lh(\tilde z)
			\\&\qquad+\big(D_ih(z)-D_ih(\tilde z)-D_{ij}h(\tilde z)(z_j-{\tilde z}_j)\big) R_i(\tilde z)
			\\&\qquad+ (D_ih(z)-D_ih(\tilde z))(R_i(z)-R_i(\tilde z))
			\\&=-\nabla (D_{i}h(z)-D_ih(\tilde z))\cdot\mathbb{M}_{il}(z)\nabla (D_{l}h(z)-D_lh(\tilde z))
			\\&\qquad -\nabla (D_{i}h(z)-D_ih(\tilde z))\cdot(\mathbb{M}_{il}(z)-\mathbb{M}_{il}(\tilde z))\nabla D_lh(\tilde z)
			\\&\qquad - \nabla\big(D_{i}h(z)-D_ih(\tilde z)-D_{ij}h(\tilde z)(z_j-\tilde z_j)\big)\cdot 
			\mathbb{M}_{il}(\tilde z)\nabla D_lh(\tilde z)
			\\&\qquad+\big(D_ih(z)-D_ih(\tilde z)-D_{ij}h(\tilde z)(z_j-{\tilde z}_j)\big) R_i(\tilde z)
			\\&\qquad+ (D_ih(z)-D_ih(\tilde z))(R_i(z)-R_i(\tilde z)).
		\end{split}
	\end{equation}
	Since, by hp.~\eqref{eq:Mnd}, $\mathbb{M}(z)\ge\diag(m(z)\mOns,0,\dots,0)+\epsilon(\iota)\diag(0,1,\dots,1)$ for a suitable constant $\epsilon(\iota)>0$, we have
	\begin{multline}\label{eq:127'}
		\nabla (D_{i}h(z)-D_ih(\tilde z))\cdot\mathbb{M}_{il}(z)\nabla (D_{l}h(z)-D_lh(\tilde z))
		\\\ge m(z)\mOns|\nabla D_0h(z)-\nabla D_0h(\tilde z)|^2+ \epsilon(\iota)\sprmA|\nabla D_ch(z)-\nabla D_ch(\tilde z)|^2
	\end{multline}
	for any $\sprmA\in(0,1]$. The auxiliary parameter $\sprmA=\sprmA(\alpha)$ will eventually be chosen small enough to be specified below.
	For $i\in\{1,\dots, n\}$ we observe that since $\iota\le z_j\le E$ for all $j\in\{0,\dots,n\}$,  the triangle inequality yields
	\begin{equation*}\begin{split}
			|\nabla D_ih(z)-\nabla D_ih(\tilde z)|&\ge |\nabla\log(c_i)-\nabla \log(\tilde c_i)|-|\nabla \log(w_i(u))-\nabla \log(w_i(\tilde u))|
			\\&\ge \tfrac{1}{c_i}|\nabla c_i-\nabla \tilde c_i|-\big|\tfrac{1}{c_i}-\tfrac{1}{\tilde c_i}\big||\nabla\tilde c_i|-|\nabla u-\nabla \tilde u|\tfrac{w_i'(u)}{w_i(u)}
			\\&\qquad-\big|\tfrac{w_i'(u)}{w_i(u)}-\tfrac{w_i'(\tilde u)}{w_i(\tilde u)}\big||\nabla \tilde u|
			\\&\ge
			\epsilon(E,\iota)|\nabla c-\nabla \tilde c|-C(E,\iota)|\nabla u-\nabla\tilde u|-C(E,\iota)|z-\tilde z|,
		\end{split}
	\end{equation*}
	where $\epsilon(E,\iota)>0$ is some sufficiently small constant.
	
	We next estimate, using the local Lipschitz continuity of $\mathbb{M}$,
	\begin{equation}\label{eq:127.'}
		\begin{split}
			|\nabla (D_{i}h(z)-D_ih(\tilde z))\cdot(\mathbb{M}_{il}(z)-\mathbb{M}_{il}(\tilde z))\nabla D_lh(\tilde z)|&\lesssim_{E,\iota}|\nabla Dh(z)-\nabla Dh(\tilde z)||z-\tilde z|
			\\&\lesssim_{E,\iota}  |z-\tilde z||\nabla z-\nabla \tilde z|+|z-\tilde z|^2.
		\end{split}
	\end{equation}
	Before estimating the remaining terms, we compute  for $f\in C^3((0,\infty)^{1+n})$
	\begin{equation}\label{eq:ft2}
		\begin{split}
			\nabla \big(f(z)-f(\tilde z)-D_jf(\tilde z)(z_j-\tilde z_j)\big)
			&=D_jf(z)\nabla z_j-D_jf(\tilde z)\nabla \tilde z_j-D_jf(\tilde z)\nabla (z_j-\tilde z_j)
			\\&\qquad -D_{jk}f(\tilde z)(z_j-\tilde z_j)\nabla \tilde z_k
			\\&=(D_jf(z)-D_jf(\tilde z))\nabla z_j -D_{jk}f(\tilde z)(z_j-\tilde z_j)\nabla \tilde z_k
			\\&=(D_jf(z)-D_jf(\tilde z))\nabla (z_j-\tilde z_j)
			\\&\qquad+[D_kf(z)-D_kf(\tilde z)-D_{jk}f(\tilde z)(z_j-\tilde z_j)]\nabla \tilde z_k
		\end{split}
	\end{equation}
	and, using Taylor's theorem, for $k=0,\dots,n$
	\begin{equation*}
		|D_kf(z)-D_kf(\tilde z)-D_{jk}f(\tilde z)(z_j-\tilde z_j)|\lesssim_{E,\iota,f}|z-\tilde z|^2.
	\end{equation*}
	Letting $f(z)=D_ih(z)$, we infer since $h\in C^4((0,\infty)^{1+n})$ (see hp.~\ref{it:hC3}) that
	\begin{equation}\label{eq:127..'}
		\begin{split}
			|\nabla\big(D_{i}h(z)-D_ih(\tilde z)-D_{ij}h(\tilde z)(z_j-\tilde z_j)\big)|
			&\lesssim_{\iota, E} |z-\tilde z||\nabla z-\nabla \tilde z|+|z-\tilde z|^2.
		\end{split}
	\end{equation}
	Using the previous bounds to estimate the RHS of~\eqref{eq:126A.'}, recalling also hp.~\ref{it:R.locLip}, and applying Young's inequality and an absorption argument, we thus infer for suitable $\epsilon(\iota,E)>0$
	\begin{equation}\label{eq:"700}
		\begin{split}
			\rho^{(h)}\le -\epsilon(\iota,E)\sprmA|\nabla c-\nabla \tilde c|^2-
			m(z)&\mOns|\nabla D_0h(z)-\nabla D_0h(\tilde z)|^2
			\\&+C_1(E,\iota)\sprmA|\nabla u-\nabla\tilde u|^2+C(\delta,E,\iota)|z-\tilde z|^2.
		\end{split}
	\end{equation}
	On the other hand, using the fact that $z\mapsto a(z)$ is locally Lipschitz continuous in $(0,\infty)^{1+n}$,  we deduce from eq.~\eqref{eq:126''b} for suitable $\epsilon_1>0$ and $C_2<\infty$ (independent of $E,\iota$)
	\begin{equation}\label{eq:"701}
		\rho^{(\en)}\le -\epsilon_1 |\nabla u-\nabla \tilde u|^2+C_2m(z)\mOns|\nabla D_0h(z)-\nabla D_0h(\tilde z)|^2+C(E,\iota)|z-\tilde z|^2,
	\end{equation}  
	where we used the fact that $0\le m(z)\lesssim \mOns$.
	
	If $\mOns=1$, we choose $\alpha\in(0,1]$ small enough such that $\alpha C_2\le 1$ and subsequently $\sprmA=\sprmA(\alpha,E,\iota)$ sufficiently small such that 
	$\sprmA C_1(E,\iota)\le \alpha\epsilon_1$. We may then conclude that 
	\begin{align}\label{eq:"703}
		\rho_\alpha=\rho^{(h)} +\alpha \rho^{(\en)}\le C(\alpha,E,\iota)|z-\tilde z|^2.
	\end{align}
	Let us emphasise that the smallness condition of $\alpha$ is independent of $E$.
	
	If instead $\mOns=0$,  we choose for given\footnote{In the case $\mOns=0$, it suffices to restrict $\alpha$ to the range $1\le\alpha<\infty$.} $\alpha\in[1,\infty)$ the parameter $\sprmA=\sprmA(\alpha, E,\iota)$ small enough such that $\sprmA \,C_1(E,\iota)\le \epsilon_1\alpha$, and obtain as before
	\begin{align}\label{eq:"703.}
		\rho_\alpha\le C(\alpha, E,\iota)|z-\tilde z|^2.
	\end{align}
	\smallskip
	
	\underline{Case $z\in \mathcal{A}_0:$}
	in this case, we have no lower bound on $z_i$ away from zero, but since 
	$\min(z)\le\iota$ and $\min(\tilde z)\ge2\iota$, we know that $|z-\tilde z|\ge\iota$. By Prop.~\ref{prop:coerc.hrel} (cf.~\eqref{eq:coercA}) it thus suffices to prove that $\rho_\alpha\lesssim_{E} 1$.
	
	Recalling~\eqref{eq:126A'} and~\ref{eq:hp.Rdiss}, we estimate
	\begin{equation}\label{eq:126A..'}
		\begin{split}
			\rho^{(h)} &\le-\nabla D_{i}h(z)\cdot\mathbb{M}_{il}(z)\nabla D_{l}h(z)
			\\&\qquad + \nabla D_ih(\tilde z)\cdot \mathbb{M}_{il}(z)\nabla D_lh(z)
			\\&\qquad + \nabla\Big(D_{ij}h(\tilde z)(z_j-\tilde z_j)\Big)\cdot 
			\mathbb{M}_{il}(\tilde z)\nabla D_lh(\tilde z)
			\\&\qquad +C(E)
			\\&\le - \scrp(z) + C(E)\sqrt{\scrp(z)}+C|\nabla u|+ C(E),
		\end{split}
	\end{equation}
	where the last step uses hp.~\eqref{eq:hp.dztrunc.P} and hp.~\eqref{eq:hp.fluxtrunc.P}.
	Hence, 
	\begin{equation*}
		\rho^{(h)}\le -\tfrac{1}{2} \scrp(z) +C(E)|\nabla u|+ C(E).
	\end{equation*}
	Next, by ineq.~\eqref{eq:rho.g5} and hp.~\eqref{eq:hp.fluxtrunc.P}, 
	\begin{align}\label{eq:rhogbd}
		\rho^{(\en)}&\le -\tfrac{a(z)}{2}|\nabla u|^2+C(\delta_1, E)+\tfrac{\delta_1 }{4}\scrp(z)+Cm(z)\mOns|\nabla D_0h(z)|^2
	\end{align}
	for any $\delta_1>0$.
	
	If $\mOns=1$, we let $\delta_1=1$ and use the estimate $m(z)\mOns |\nabla D_0h(z)|^2\le \scrp(z)$, which follows from hp.~\ref{eq:hp.Mnondeg}, to see that after possibly decreasing $\alpha\in(0,1]$ we have $\rho_\alpha\le-\alpha\tfrac{a(z)}{2}|\nabla u|^2+C(E)|\nabla u|+C(E)\le -\alpha\tfrac{a(z)}{4}|\nabla u|^2+C(\alpha,E)$.
	
	If $\mOns=0$, we choose $\delta_1=\tfrac{1}{\alpha}\le 1$. Then 
	\begin{equation*}
		\begin{split}
			\rho_\alpha&\le   -\tfrac{1}{2} \scrp(z) 
			-\alpha\tfrac{a(z)}{2}|\nabla u|^2+\tfrac{a(z)}{4}|\nabla u|^2+C(\alpha, E)
			+\tfrac{1}{4}\scrp(z)
			\\&\lesssim_{\alpha,E} 1.
		\end{split}
	\end{equation*}
	
	\smallskip
	
	\underline{Case $z\in \mathcal{B}:$} in this case derivatives of $\xi^*$ do in general not vanish, but we know that $1\ll E<|z|_1<E^N$.
	We estimate
	\begin{align}
		\rho^{(h)} &=-\nabla D_{i}h(z)\cdot\mathbb{M}_{il}(z)\nabla D_{l}h(z)
		\\&\qquad + \nabla (D_i(\xi^*(z)z_j)D_jh(\tilde z) )\cdot \mathbb{M}_{il}(z)\nabla D_lh(z)
		\\&\qquad + \nabla\Big(D_{ij}h(\tilde z)(\xi^*(z)z_j-\tilde z_j)\Big)\cdot 
		\mathbb{M}_{il}(\tilde z)\nabla D_lh(\tilde z)
		\\&\qquad- D_{ij}h(\tilde z)(\xi^*(z)z_j-{\tilde z}_j) R_i(\tilde z)
		\\&\qquad+ D_ih(z)R_i(z)-D_jh(\tilde z)D_i(\xi^*(z)z_j) R_i(z)
		\\&\le - \scrp(z)			\label{eq:126B'}
		\\&\qquad +C| \nabla D_i(\xi^*(z)z_j)\cdot \mathbb{M}_{il}(z)\nabla D_lh(z)|
		\\&\qquad + C(E,N)\sqrt{\scrp(z)}+C(E,N)|\nabla u|\hspace{2cm}(\text{by hp.}~\eqref{eq:hp.fluxtrunc.P},~\eqref{eq:hp.dztrunc.P})
		\\&\qquad +C(E,N)\hspace{2cm}(\text{using hp.}~\ref{eq:hp.Rdiss})
		\\&\le -\tfrac{1}{2} \scrp(z)
		\\&\qquad +C| \nabla D_i(\xi^*(z)z_j)\cdot \mathbb{M}_{il}(z)\nabla D_lh(z)|
		\\&\qquad +C(E,N)|\nabla u|+C(E,N).
	\end{align}
	In order to estimate the term $|\nabla D_i(\xi^*(z)z_j)\cdot \mathbb{M}_{il}(z)\nabla D_lh(z)|$, we observe that 
	\begin{equation}\label{eq:301}
		\begin{split}
			|\nabla D_i(\xi^*(z)z_j)|&\lesssim|D\xi^*(z)||\nabla z|+|D^2\xi^*(z)z||\nabla z|
			\\&\lesssim\tfrac{1}{N|z|_1}|\nabla z|	\hspace{2.5cm}(\text{by }\ref{eq:Dxi*}).
		\end{split}
	\end{equation}
	We first consider the case $\mOns=1$. Then~\ref{hp:grad.flux.control} is at our disposal, which yields using~\eqref{eq:301}
	\begin{equation*}
		\begin{split}
			| \nabla D_i(\xi^*(z)z_j)||\mathbb{M}_{il}(z)\nabla D_lh(z)|&\lesssim \tfrac{1}{N|z|_1}|\nabla z||\sum_jA_{ij}(z)\nabla z_j|\lesssim \tfrac{1}{N}\scrp(z)
			+\tfrac{1}{N}|\nabla u|^2.
		\end{split}
	\end{equation*}
	Thus, choosing $N=N(\min\{1,\alpha\})$ sufficiently large, we infer
	\begin{align}
		\rho^{(h)} &\le -\tfrac{1}{4} \scrp(z)+\min\{1,\alpha\}\tfrac{a(z)}{8}|\nabla u|^2+ C(E,N,\alpha).
	\end{align}
	Next, simiarly as in~\eqref{eq:rhogbd}, we estimate 
	\begin{align}\label{eq:rhogbd.m}
		\rho^{(\en)}&\le -\tfrac{a(z)}{2}|\nabla u|^2+C(E,N)+\tfrac{1}{8}\scrp(z)+C_4m(z)\mOns|\nabla D_0h(z)|^2.
	\end{align}
	Decreasing $\alpha\in(0,1]$, if necessary, to ensure that $\alpha C_4m(z)\mOns|\nabla D_0h(z)|^2\le \tfrac{1}{8} \scrp(z)$, 
	we obtain $\rho_\alpha\lesssim_{N,E}1.$
	
	\smallskip
	
	It remains to consider the case where $\mOns=0$ and~\ref{hp:trunc.grad.flux} are fulfilled.
	Since $\underline{u}:=\inf \uin= \inf\tilde u(0,\cdot)\ge 2\iota$, Lemma~\ref{l:u.positivity} yields~$\inf u\ge\underline u>0$. 
	By hp.~\ref{hp:trunc.grad.flux},
	we infer for all $0\le i\le n$
	\begin{align}
		|\nabla z||\mathbb{M}_{il}(z)\nabla D_lh(z)|\le C|z|_1\scrp(z)
		+C(E,N,\underline u)|\nabla u|^2.
	\end{align}
	Thus, recalling ineq.~\eqref{eq:301}, we can estimate
	for $N$ large enough
	\begin{align}\label{eq:ineq.B}
		| \nabla D_i(\xi^*(z)z_j)\cdot \mathbb{M}_{il}(z)\nabla D_lh(z)|&\le 
		\tfrac{1}{4}\scrp(z)+C(E,N,\underline{u})|\nabla u|^2
	\end{align}
	to infer
	\begin{equation*}
		\rho^{(h)} \le -\tfrac{1}{4} \scrp(z)+ C_1(E,N,\underline{u})|\nabla u|^2+C(E,N).
	\end{equation*}
	Next, since $\mOns=0$, ineq.~\eqref{eq:rho.g5} yields
	\begin{align}
		\rho^{(\en)}&\le -\tfrac{a(z)}{2}|\nabla u|^2+C(E,N)\sqrt{\scrp(z)}+C
		\\&\le -\tfrac{a(z)}{2}|\nabla u|^2+\tfrac{1}{8\alpha}\scrp(z) +C(\alpha,E,N).
	\end{align}
	Increasing $\alpha=\alpha(E,N,\underline{u})$, if necessary, to ensure that $\alpha\tfrac{a(z)}{2}\ge C_1(E,N,\underline{u})$, we conclude 
	\begin{align}
		\rho_\alpha\le   -\tfrac{1}{8} \scrp(z) +C(\alpha,E,N,\underline{u}).
	\end{align}
	
	\medskip
	
	\underline{Case $z\in \mathcal{C}:$}   
	in this case, $\xi^*(z)=0$ and $D^k\xi^*(z)=0$ for all $k\in\mathbb{N}_+$.
	Thus
	\begin{equation}\label{eq:126C'}
		\begin{split}
			\rho^{(h)} &=-\nabla D_{i}h(z)\cdot\mathbb{M}_{il}(z)\nabla D_{l}h(z)
			\\&\qquad - \nabla(D_{ij}h(\tilde z)\tilde z_j)\cdot 
			\mathbb{M}_{il}(\tilde z)\nabla D_lh(\tilde z)
			\\&\qquad +D_{ij}h(\tilde z){\tilde z}_jR_i(\tilde z)
			\\&\qquad+ D_ih(z)R_i(z)
			\\&\le - \scrp(z)+C,
		\end{split}
	\end{equation}	
	where we used hp.~\ref{eq:hp.Rdiss}.
	
	If $\mOns=1$, we have thanks to~\eqref{eq:rho.g5} and hp.~\ref{eq:hp.Mnondeg}, \ref{eq:hp.A00}
	$$ \rho^{(\en)}\le -\tfrac{a(z)}{2}|\nabla u|^2+C_4\scrp(z) +u^2+C,$$
	where $C_4$ is independent of $E,N$. 
	Hence, after possibly decreasing $\alpha\in(0,1]$ to ensure that $\alpha C_4\le 1$, we find 
	\begin{align}
		\rho_\alpha&\le   C+\alpha u^2
		\lesssim_{\alpha} \dist_\alpha(z,\tilde z),
	\end{align}
	where the second step follows from~\eqref{eq:coercAc01} (after choosing $E=E(\alpha)$ large enough).
	
	If $\mOns=0$, we estimate using again~\eqref{eq:rho.g5} and hp.~\ref{eq:hp.A00}
	\begin{align}
		\rho^{(\en)}&\le -\tfrac{a(z)}{2}|\nabla u|^2+C(1+u)\sqrt{\scrp(z)}+C
		\\&\le -\tfrac{a(z)}{2}|\nabla u|^2+\tfrac{1}{2\alpha}\scrp(z) +C(\alpha)u^2+C(\alpha),
	\end{align}
	and infer
	\begin{align}
		\rho_\alpha &\le -\tfrac{1}{2} \scrp(z)-\alpha \tfrac{a(z)}{2}|\nabla u|^2+C(\alpha)u^2+C(\alpha)
		\\&\lesssim_{\alpha}\dist_\alpha(z,\tilde z).
	\end{align}
	The second step follows from the coercivity property~\eqref{eq:coercAc} and the fact that $\alpha\ge1$.
	
	This proves the bound~\eqref{eq:601} and thus completes the proof of Theorem~\ref{thm:wkuniq}.
\end{proof}

\section{Strong entropy dissipation property}\label{sec:edin}

\begin{proof}[Proof of Proposition~\ref{prop:edin}]
	We first establish~\eqref{eq:edin.s} for $s=0$ and a.e.~$t=T\in(0,T^*)$, that is, we first prove~\eqref{eq:edin}. In a second step (see page~\pageref{eq:111}), we point out how to extend the result to a.e.~$0<s<t<T^*$.
	\smallskip
	
	\underline{Case 1: $s=0, t=T\in(0,T^*)$.}
	
	We consider for a small parameter $\delta>0$ the regularised entropy density
	\begin{equation*}
		h_{\delta}(u,c) = \delta u+ h(u,c).
	\end{equation*}
	The additive term $\delta u$ serves to ensure coercivity, since the original density $h(z)$ may in general allow for cancellations at infinity reflecting the coupling between concentrations and energy component. In particular, for every $L\in \mathbb{N}$ the sublevel set $\{z\in \mathbb{R}_{\ge0}^{1+n}:h_\delta(z)\le L\}$ is bounded.
	This coercivity property easily follows from the lower bound
	\begin{equation}\label{eq:hdel.coer}
		h(z)\ge -\hat\sigma(u)+\epsilon_*\sum_{i=1}^n c_i\log c_i-Cu^{\nu}- C,
	\end{equation}
	valid for suitable $\nu\in[0,1)$, $\epsilon_*>0$ (see Lemma~\ref{l:S}),  together with the sublinearity of the increasing function $\hat\sigma(u)$ as $u\to\infty$ (see~\eqref{eq:h2}).
	
	In order to define an admissible truncation function, we consider  as in~\cite[Proof of Prop.~5]{Fischer_2017} for $L\ge 2$ an auxiliary function
	$\theta_L\in C^\infty(\mathbb{R})$ satisfying $\theta_L(s)=s$ for $|s|\le L$, 
	$0\le\theta_L'\le 1$, 
	\begin{align}\label{eq:decay.theta''}
		|\theta''_L(s)|\lesssim \tfrac{C}{1+|s|\log(|s|+\mathrm{e})} 
	\end{align}for all $s\in\mathbb{R}$ and $\theta'_L(s)=0$ for $|s|\ge L^C$ for some sufficiently large  constant $C\ge 2$, which is kept fixed throughout the proof.
	
	To derive the entropy dissipation inequality, we would like to choose the truncation function $\theta_L(h_\delta(\cdot))$ and the test function $\psi\equiv 1$ in the renormalised formulation satisfied by $z$, and then let $L\to\infty$ and subsequently $\delta\to0$. Since derivatives of $h_\delta$ are in general not bounded as $u\searrow0$ or $c_i\searrow0$, further regularisation is required. We let 
	\begin{equation*}
		h_{\delta,\ve}(u,c) = h_\delta(z)+\hat\sigma(u)-\hat\sigma(u+\ve),
	\end{equation*}
	abbreviate for $z=(u,c_1,\dots,c_n)$
	$$z^{\tilde\ve}:=(u,c_1+\tilde\ve,\dots,c_n+\tilde\ve),\quad\tilde\ve\in(0,1],$$
	and then consider the function 
	$z\mapsto h_{\delta,\ve}(z^{\tilde\ve})\in C^2([0,\infty)^{1+n})$.
	Thanks to~\eqref{eq:hdel.coer}, it is easy to see that 
	for fixed $\delta\in(0,1]$ sublevel sets of $z\mapsto h_{\delta,\ve}(z^{\tilde\ve})$ are bounded; in fact 
	\begin{align}\label{eq:hreg}
		h_{\delta,\ve}(z^{\tilde\ve})\ge \tfrac\delta2u+\epsilon_*\sum_{i=1}^n c_i\log_+ c_i-C_\delta
	\end{align}
	for $\delta,\ve,\tilde\ve\in(0,1]$. 
	Hence, the $C^2$-function $\xi(z):=\theta_L(h_{\delta,\ve}(z^{\tilde\ve}))$ has compactly supported derivative $D\xi$, and is thus an admissible truncation in eq.~\eqref{eq:118} (cf.~Remark~\ref{rem:testfk}). This, combined with the choice $\psi\equiv 1$  in eq.~\eqref{eq:118}, yields for a.e.\ $T<T^*$
	\begin{equation}\label{eq:115*n}
		\begin{split}
			LHS:=&\eval{\int_\Om \theta_L(h_{\delta,\ve}(z^{\tilde\ve}))\,\dd x}_{t=0}^{t=T} 
			\\=&-\int_0^T\!\!\int_\Om \theta_L'(h_{\delta,\ve}(z^{\tilde\ve})) D_{ij}h_{\delta,\ve}(z^{\tilde\ve})\nabla z_j\cdot A_{ik}(z)\nabla z_k\,\dd x \dd t
			\\& -\int_0^T\!\!\int_\Om \theta_L''(h_{\delta,\ve}(z^{\tilde\ve}))D_jh_{\delta,\ve}(z^{\tilde\ve})D_ih_{\delta,\ve}(z^{\tilde\ve})\nabla z_j \cdot A_{ik}(z)\nabla z_k\,\dd x \dd t
			\\& +\int_0^T\!\!\int_\Om\theta_L'(h_{\delta,\ve}(z^{\tilde\ve}))D_ih(z^{\tilde\ve}) R_i(z)\,\dd x\dd t
			\\=&: I+II+III,
		\end{split}
	\end{equation}
	where we recall the summation convention (see Notations~\ref{ssec:notations}).
	In the reaction term we have used the fact that $D_ih_{\delta,\ve}(z)=D_ih(z)$ whenever $i\not=0$ together with $R_0\equiv0$.
	
	We will establish the asserted inequality~\eqref{eq:edin} by taking the $\liminf$ of the LHS and the 
	$\limsup$ of the RHS of the above equation~\eqref{eq:115*n} as $\tilde\ve\to0$, $L\to\infty$ and $\delta,\ve\to0$,  in the stated order. We perform the corresponding  limits separately in $LHS$ and  in each of the three terms $I, II,III$.
	Below we use, without explicit reference, the following basic properties satisfied under the hypotheses of Model~\Mzero{}, see Lemma~\ref{l:modelBounds}:
	\begin{align}\label{eq:700}
		&a(z)|\nabla u|^2\sim |\nabla u|^2,
		\\&	\scrp(z)\gtrsim \sum_{i=1}^n|\nabla \sqrt{c_i}|^2+|\sqrt{\gamma}\,\nabla u|^2
		+|\sqrt{m}\nabla D_0h(z)|^2,
	\end{align}
	where $\gamma(u,c)=-\hat\sigma''(u)-\sum_{l=1}^n\tfrac{w_l''(u)}{w_l(u)}c_l$, 
	\begin{align}\label{eq:702}
		|A_{ik}(z)\nabla z_k|&\lesssim \sqrt{c_i}\sqrt{\scrp(z)}\qquad\text{for }i\ge1,
		\\|A_{0k}(z)\nabla z_k|&\lesssim |\nabla u|+\sqrt{m}\sqrt{\scrp(z)}.\label{eq:703}
	\end{align}
	\smallskip
	
	\emph{LHS:}\\
	The limit $\tilde\ve\to0$ of the $LHS$ is immediate due to the boundedness of $\theta_L$, and yields
	\begin{align}
		\eval{\int_\Om \theta_L(h_{\delta,\ve}(z))\,\dd x}_{t=0}^{t=T} .
	\end{align}
	We next take $$\liminf_{\ve,\delta\to\infty}\liminf_{L\to\infty}$$ of the last expression using a combination of the dominated convergence theorem and Fatou's lemma.
	At initial time, we estimate using the lower and upper bounds on $h$ in Lemma~\ref{l:S}
	$$|\theta_L(h_{\delta,\ve}(\zin))|\le |h_{\delta,\ve}(\zin)|\lesssim \uin+|\hat\sigma_{-}(\uin)|+\sum_{i=0}^n\cin_i\log_+\cin_i+1.$$
	We can hence use dominated convergence to deduce that, as $L\to\infty$ and $\delta,\ve\to0$,
	\begin{align}
		\int_\Om \theta_L(h_{\delta,\ve}(\zin))\,\dd x \to \int_\Om h(\zin)\,\dd x.
	\end{align}
	To deal with the integral at time $t=T$, we first observe that, thanks to the regularity $u\in L^\infty_\loc([0,T^*),L^1(\Om))$ and~\eqref{eq:hdel.coer}, the negative part of $\theta_L(h_{\delta,\ve}(z(T,\cdot)))$ is controlled pointwise in $x$, for a.e.\ $T\in(0,T^*]$, by an integrable function,
	uniformly in $L,{\delta,\ve}$, and its integral can thus be shown to converge in the same way as the term at initial time. For the positive part of $\theta_L(h_{\delta,\ve}(z(T,\cdot)))$ we use Fatou's lemma:
	\begin{align}
		&	\int_\Om \max\{h_{\delta,\ve}(z(T,\cdot)),0\}\,\dd x \le\liminf_{L\to\infty}\int_\Om \max\{\theta_L(h_{\delta,\ve}(z(T,\cdot))),0\}\,\dd x,
		\\&	\int_\Om \max\{h(z(T,\cdot)),0\}\,\dd x\le\liminf_{\delta,\ve\to0}\int_\Om \max\{h_{\delta,\ve}(z(T,\cdot)),0\}\,\dd x.
	\end{align}
	In combination, we infer 
	\begin{align}\label{eq:lhs.liminf}
		\eval{\int_\Om h(z)\,\dd x}_{t=0}^{t=T}\le\liminf_{\delta,\ve\to0}\liminf_{L\to\infty}\lim_{\tilde\ve\to0}	\eval{\int_\Om \theta_L(h_{\delta,\ve}(z^{\tilde\ve}))\,\dd x}_{t=0}^{t=T}.
	\end{align}
	
	\medskip
	
	\emph{Diffusive dissipation term I:}\\
	We assert that 
	\begin{equation}
		\limsup_{\delta,\ve\to0}\limsup_{L\to\infty}\limsup_{\tilde\ve\to0}\Big(	-\int_0^T\!\!\int_\Om \theta_L'(h_{\delta,\ve}(z^{\tilde\ve}))\nabla z_j\cdot D_{ji}h_{\delta,\ve}(z^{\tilde\ve})A_{ik}(z)\nabla z_k\,\dd x \dd t\Big)
	\end{equation}
	can be bounded above by the non-positive term
	\begin{align}
		-\int_0^T\!\!\int_\Om  \nabla z_j\cdot D_{ji}h(z)A_{ik}(z)\nabla z_k\,\dd x \dd t.
	\end{align}
	
	To show this, we will mainly rely on the dominated convergence theorem.
	We therefore start by listing several uniform pointwise estimates on the terms involved.
	
	We first estimate for $i,j$ fixed the term
	\begin{align}
		\iuu:=	\theta_L'(h_{\delta,\ve}(z^{\tilde\ve})) D_{ji}h_{\delta,\ve}(z^{\tilde\ve})\nabla z_j\cdot \sum_kA_{ik}(z)\nabla z_k,
	\end{align}
	where we recall that $0\le \theta'_L\le 1.$
	
	Case $i,j\ge1:$ in this case $D_{ij}h_{\delta,\ve}(z^{\tilde\ve})=\tfrac{1}{c_i+\tilde\ve}\delta_{ij}$ and thus
	\begin{align}
		|\iuu|\le|\tfrac{1}{c_i+\tilde\ve}\nabla c_i\cdot \sum_kA_{ik}(z)\nabla z_k|
		\lesssim \tfrac{1}{c_i}|\nabla c_i|\sqrt{c_i}\sqrt{\scrp(z)}\lesssim \scrp(z).
	\end{align}
	
	Case $i\ge1,j=0:$
	observing that $-D_{i0}h_{\delta,\ve}(z^{\tilde\ve})=\tfrac{w_i'}{w_i}\lesssim \sqrt{\tfrac{-w_i''}{w_i}}$ (cf.~\eqref{eq:wi.M0}),
	we find 
	\begin{align}
		|\iuu|\le|\tfrac{w_i'}{w_i}\nabla u\cdot \sum_kA_{ik}(z)\nabla z_k|
		\lesssim |\sqrt{\tfrac{-w_i''}{w_i}c_i}\nabla u|\sqrt{\scrp(z)}\lesssim \scrp(z).
	\end{align}
	
	Case $i=0,j\ge0:$  we split the sum over $k$ in the definition of $\iuv$ into two parts:
	$$\sum_kA_{0k}(z)\nabla z_k=a(z)\nabla u+m(z)\nabla D_0h(z)$$
	and split $\iuv$ accordingly into $\iuv=\iuv^{(0)}+\iuv^{(1)}.$
	
	For $j\ge 1$ we estimate as above
	\begin{align}
		|\iuv^{(0)}|&\le|\tfrac{w_j'}{w_j}\nabla c_j\cdot a(z)\nabla u|
		\lesssim |\sqrt{\tfrac{-w_j''}{w_j}c_j}\nabla u||\nabla \sqrt{c_j}|
		\lesssim \scrp(z).
	\end{align}
	For $j=0$, we have $p_{00}^{(0)}=\theta_L'(h_{\delta,\ve}(z^{\tilde\ve}))D_{00}h_{\delta,\ve}(z^{\tilde\ve})a(z)|\nabla u|^2\ge0$,  and since in the equation it comes with a minus sign, its integral can easily be handled using Fatou's lemma.
	
	\medskip 
	
	It remains to estimate the part~$\iuv^{(1)}$. 
	To deal with the limit $\tilde\ve\to0$ (for finite $L$), we estimate 
	\begin{align}
		|\nabla D_0h_{\delta,\ve}(z^{\tilde\ve})\cdot m(z)\nabla D_0h(z)|
		&\lesssim\sqrt{m}|\nabla D_0h_{\delta,\ve}(z^{\tilde\ve})|\sqrt{\scrp(z)}
	\end{align}
	and for $ |z^{\tilde\ve}|\le C(L,\delta)$
	\begin{align}
		\sqrt{m}|\nabla D_0h_{\delta,\ve}(z^{\tilde\ve})|
		&\le |D_{00}h_{\delta,\ve}(z^{\tilde\ve})\nabla u|+
		\sum_{j=1}^n|D_{0j}h_{\delta,\ve}(z^{\tilde\ve})\nabla c_j|
		\\&\lesssim_{L,\delta,\ve}  |\nabla u|+\sqrt{\scrp(z)}.
	\end{align}
	Thus the limit $\tilde\ve\to0$ can be handled using dominated convergence.

	We can now let $\tilde\ve=0$ and compute
	\begin{align}
		\nabla &D_0h_{\delta,\ve}(z)\cdot m(z)\nabla D_0h(z)
		\\&=m(z)|\nabla D_0h(z)|^2
		-(\hat\sigma''(u+\ve)-\hat\sigma''(u))\nabla u\cdot m(z)\nabla D_0h(z).
	\end{align}
	Note that the first term on the RHS is bounded above by $\scrp(z)\in L^1(\Om_T)$.
	Concerning the second term, we estimate
	\begin{align}
		|(\hat\sigma''(u+\ve)-\hat\sigma''(u))&\nabla u\cdot m(z)\nabla D_0h(z)|
		\\&\lesssim \sqrt{m}|(\hat\sigma''(u+\ve)-\hat\sigma''(u))\nabla u|\sqrt{\scrp(z)}
		\\&\lesssim |(1+\sqrt{-\hat\sigma''(u)})\nabla u|	\sqrt{\scrp(z)}
		\\&\lesssim \scrp(z)+|\nabla u|^2,
	\end{align}
	where in the penultimate step we have used hypothesis~\eqref{eq:704}.
	
	Combining the above estimates allows to take the successive limits 
	$$\limsup_{\ve,\delta\to0}\limsup_{L\to\infty}\limsup_{\tilde\ve\to0}\dots$$
	as above, thus yielding the asserted inequality for term~$I$.
	
	\medskip
	
	\emph{Remainder gradient term II:}\\
	We will show that \small
	\begin{equation}\label{eq:gradrem}
		\limsup_{L\to\infty}\limsup_{\tilde\ve\to0}\Big( -\int_0^T\!\!\int_\Om \theta_L''(h_{\delta,\ve}(z^{\tilde\ve}))D_jh_{\delta,\ve}(z^{\tilde\ve})\nabla z_j\cdot D_ih_{\delta,\ve}(z^{\tilde\ve}) A_{ik}(z)\nabla z_k\,\dd x \dd t\Big)\le0.
	\end{equation}\normalsize    
	As in the previous paragraph, the main task is to obtain uniform pointwise estimates on the terms involved, where here we can afford a dependence of our estimates on $\delta$ and $\ve$.
	We introduce for $i,j$ fixed the term
	\begin{align}
		\iru=- \theta_L''(h_{\delta,\ve}(z^{\tilde\ve}))D_jh_{\delta,\ve}(z^{\tilde\ve})\nabla z_j\cdot D_ih_{\delta,\ve}(z^{\tilde\ve}) \sum_kA_{ik}(z)\nabla z_k
	\end{align}
	and observe that, by \eqref{eq:decay.theta''} and~\eqref{eq:hreg}, for $L\ge L_0(\delta)$ (henceforth to be assumed)
	\begin{align}\label{eq:theta''}
		|\theta''_L(h_{\delta,\ve}(z^{\tilde\ve}))|\lesssim_\delta \tfrac{1}{1+(u+\sum_ic_i\log_+(c_i))\log(|z|_1+\mathrm{e})}.
	\end{align}
	
	Case $i,j\ge1:$ 
	in this case, 
	\begin{align}
		|\iru|\lesssim 
		|\theta''_L(h_{\delta,\ve}(z^{\tilde\ve}))|
		|\log(\tfrac{c_i+\tilde\ve}{w_i(u)})||\log(\tfrac{c_j+\tilde\ve}{w_j(u)})|
		\sqrt{c_i}\sqrt{c_j}|\nabla\sqrt{c_j}|\sqrt{\scrp(z)}.
	\end{align}
	Estimating 
	\begin{align}
		\sqrt{c_i}|\log(\tfrac{c_i+\tilde\ve}{w_i(u)})|\lesssim \sqrt{c_i}\log_+(c_i)+ \sqrt{w_i(u)}\log_+(w_i(u))+1,
	\end{align}
	we find, using~\eqref{eq:h2} and~\eqref{eq:theta''}, 
	\begin{multline}
		|\theta''_L(h_{\delta,\ve}(z^{\tilde\ve}))|\,\sqrt{c_i}|\log(\tfrac{c_i+\tilde\ve}{w_i(u)})|\,\sqrt{c_j}|\log(\tfrac{c_j+\tilde\ve}{w_j(u)})|
		\\\lesssim 	|\theta''_L(h_{\delta,\ve}(z^{\tilde\ve}))|(\sqrt{c_i}\log_+(c_i)+ \sqrt{u}+1)(\sqrt{c_j}\log_+(c_j)+ \sqrt{u}+1)
		\;	\lesssim_\delta \;1.\quad
	\end{multline}
	Thus, $|\iru|\lesssim_\delta  \scrp(z)$.
	
	Case $i\ge1,j=0:$
	here,
	\begin{align}
		|\iru|&\lesssim \sqrt{c_i}|\log(\tfrac{c_i+\tilde\ve}{w_i(u)})||D_0h_{\delta,\ve}(z^{\tilde\ve})\nabla u|
		\sqrt{\scrp(z)}\,|\theta''_L(h_{\delta,\ve}(z^{\tilde\ve}))|.
	\end{align}
	In view of the factor $\sqrt{c_i},$ this shows that $|\iru|\le C(L,\delta,\ve)|\nabla u|\sqrt{\scrp(z)}$ uniformly in $\tilde\ve$, allowing us to infer by dominated convergence
	\begin{align}
		\limsup_{\tilde\ve\to0}\int_0^T\!\!\int_\Om\iru \,\dd x \dd t\le 
		-\int_0^T\!\!\int_\Om \theta_L''(h_{\delta,\ve}(z))D_0h_{\delta,\ve}(z)\nabla u\cdot \log(\tfrac{c_i}{w_i(u)}) A_{ik}(z)\nabla z_k\,\dd x \dd t.
	\end{align}
	Since $\ve>0$ and $\tfrac{w'_l(u)}{w_l(u)}\lesssim \sqrt{\tfrac{-w''_l(u)}{w_l(u)}}$,
	we have the rough bound 
	\begin{equation}\label{eq:200}\begin{split}		
			|D_0h_{\delta,\ve}(z)\nabla u|&\lesssim_\ve |\nabla u|+\sum_l\sqrt{c_l}\sqrt{\gamma}|\nabla u|
			\\&\lesssim |\nabla u|+\sum_l\sqrt{c_l}\sqrt{\scrp(z)}.
		\end{split}
	\end{equation}
	Moreover, 
	\begin{align}
		|\log(\tfrac{c_i}{w_i(u)}) A_{ik}(z)\nabla z_k|\lesssim (\log_+(c_i)+1+\log_+(w_i(u)))\sqrt{c_i}\sqrt{\scrp(z)},
	\end{align}
	and hence
	\begin{align}
		|\theta_L''(h_{\delta,\ve}(z))&D_0h_{\delta,\ve}(z)\nabla u\cdot D_ih_{\delta,\ve}(z) A_{ik}(z)\nabla z_k|
		\lesssim_\ve \scrp(z)+|\nabla u|^2.
	\end{align}
	
	Case $i=0,j\ge0:$
	using the fact that for $|z|_1\le C(L,\delta)$
	\begin{align}
		|D_jh_{\delta,\ve}(z^{\tilde\ve})\nabla z_j||D_0h_{\delta,\ve}(z^{\tilde\ve})|
		|a(z)\nabla u+m(z)\nabla D_0h(z)|&\lesssim_{L,\delta, \ve}\scrp(z)+|\nabla u|^2,
	\end{align}
	one can take the limit $\tilde\ve\to0$
	\begin{align}
		\limsup_{\tilde\ve\to0}\big(& -\int_0^T\!\!\int_\Om \theta_L''(h_{\delta,\ve}(z^{\tilde\ve}))D_jh_{\delta,\ve}(z^{\tilde\ve})\nabla z_j\cdot D_0h_{\delta,\ve}(z^{\tilde\ve}) \sum_kA_{0k}(z)\nabla z_k\,\dd x \dd t\big)
		\\&\le -\int_0^T\!\!\int_\Om \theta_L''(h_{\delta,\ve}(z))D_jh_{\delta,\ve}(z)\nabla z_j\cdot D_0h_{\delta,\ve}(z) \sum_kA_{0k}(z)\nabla z_k\,\dd x \dd t.
	\end{align}
	To obtain $L$-uniform bounds of the integrand on the RHS, we split 
	$$\sum_kA_{0k}(z)\nabla z_k=a(z)\nabla u + m(z)\nabla D_0h(z).$$
	Using~\eqref{eq:200} we have for  $j\ge 1$
	\begin{multline}
		|\nabla z_jD_jh_{\delta,\ve}(z)||D_0h_{\delta,\ve}(z)a(z)\nabla u |
		\\\lesssim_\ve \sqrt{c_j}(\log_+(c_j)+1+\log_+(w_j(u)))(|\nabla u|\sqrt{\scrp(z)}+\sum_l\sqrt{c_l}\scrp(z)),
	\end{multline}
	and for $j=0$ 
	\begin{align}
		|\nabla uD_0h_{\delta,\ve}(z)||D_0h_{\delta,\ve}(z)a(z)\nabla u |
		\lesssim_{\epsilon}\big(|\nabla u|+\sqrt{|c|_1}\sqrt{\scrp(z)}\big)^2
		\lesssim |\nabla u|^2+|c|_1\scrp(z).
	\end{align}
	It remains to consider the term involving $m(z)$.
	We estimate for $j\ge1$
	\begin{multline}
		|D_jh_{\delta,\ve}(z)\nabla z_j\cdot D_0h_{\delta,\ve}(z) ||m(z)\nabla D_0h(z)|
		\\\lesssim_{\ve}\sqrt{c_j}(\log_+(c_j)+1+\log_+(w_j(u)))(1+\tfrac{\sum_lc_l}{1+u}) \sqrt{m}\scrp(z),
	\end{multline}
	where we used the fact that~$\tfrac{w_l'(u)}{w_l(u)}\lesssim\tfrac{1}{1+u}$ for all $l$
	allowing us to estimate
	\begin{align}
		|D_0h_{\delta,\ve}(z)|\lesssim_\ve 1+\tfrac{\sum_{l=1}^nc_l}{1+u}.
	\end{align}
	Similarly, for $j=0$ we estimate, using also the bound~\eqref{eq:200},
	\begin{multline}
		|D_0h_{\delta,\ve}(z)\nabla u\cdot D_0h_{\delta,\ve}(z) ||m(z)\nabla D_0h(z)|
		\\\lesssim_{\ve}(1+\tfrac{\sum_{l'=1}^nc_{l'}}{1+u})\big(|\nabla u|+\sum_{l=1}^n\sqrt{c_l}\sqrt{\scrp(z)}\big)\sqrt{m}\sqrt{\scrp(z)}.
	\end{multline}
	Thus, recalling the conditions~\eqref{eq:705},~\eqref{eq:704} on $m(z)$, we infer the $L$-uniform bound
	\begin{align}
		|\theta_L''(h_{\delta,\ve}(z))D_jh_{\delta,\ve}(z)\nabla z_j\cdot D_0h_{\delta,\ve}(z) \sum_kA_{0k}(z)\nabla z_k|\lesssim_{\delta,\ve}\scrp(z)+|\nabla u|^2.
	\end{align}
	
	Combining the above estimates, we can take the limits 
	$$\limsup_{L\to\infty}\limsup_{\tilde\ve\to0}\dots$$
	of term~$II$ and obtain~ineq.~\eqref{eq:gradrem} by the pointwise convergence $\theta_L''(s)\to0$ as $L\to\infty$.
	
	\medskip
	
	\emph{Reactions III:}\\
	Concerning the reaction term $III$, we first need to take care of the fact that $D_ih(z)$ is unbounded near $c_i=0$. By~\eqref{eq:hreg}, we have $|z|,|z^{\tilde\ve}|\le  C(L,\delta)$ unless
	$\theta_L'(h_{\delta,\ve}(z^{\tilde\ve}))=0$.
	Using the local $\epsilon_0$-H\"older regularity of $R_i$, we then compute for
	$|z^{\tilde\ve}|\le C(L,\delta)$
	\begin{equation*}
		\begin{split}
			D_ih(z^{\tilde\ve})R_i(z)&=D_ih(z^{\tilde\ve})[R_i(z)-R_i(z^{\tilde\ve})]
			+D_ih(z^{\tilde\ve})R_i(z^{\tilde\ve})
			\\&\le C(L,\delta)(1+ \sum_{i=1}^n\log(c_i+\tilde\ve))\tilde\ve^{\epsilon_0}+D_ih(z^{\tilde\ve})R_i(z^{\tilde\ve}).
		\end{split}
	\end{equation*}
	The first term in the last line converges, as $\tilde\ve\to0$, uniformly to zero on the set $\{ |z^{\tilde\ve}|\le C(L,\delta)\}$. The second term is non-positive by hp.~\ref{eq:hp.Rdiss}. 
	Thus, since $\theta'_L\ge0$ and $\lim_{L\to\infty}\theta'_L(s)=1$ for all $s\in\mathbb{R}$ we can use  Fatou's lemma to infer
	\begin{equation}\label{eq:800}
		\begin{split}
			\limsup_{L\to\infty}\limsup_{\tilde\ve\to0}\int_0^T\!\!\int_\Om\theta_L'(h_{\delta,\ve}(z^{\tilde\ve}))&D_ih(z^{\tilde\ve}) R_i(z)\,\dd x\dd t
			\\&
			\le \limsup_{L\to\infty}\int_0^T\!\!\int_\Om  \theta_L'(h_{\delta,\ve}(z))D_ih(z)R_i(z)\,\dd x\dd t
			\\&
			\le \int_0^T\!\!\int_\Om D_ih(z)R_i(z)\,\dd x\dd t.
		\end{split}
	\end{equation}
	Observe that the last line is independent of $\delta$ and $\ve$.
	
	\smallskip
	
	Put together,  the above inequalities and equation~\eqref{eq:115*n} imply the entropy dissipation inequality~\eqref{eq:edin}.
	\medskip
	
	\underline{Case~2: $0<s<t<T^*$.}
	
	We assert that the fact that $z$ is a renormalised solution in $\Om_{T^*}$ in the sense of Definition~\ref{def:renorm} implies that for a.e.~$0<s<t<T^*$, and all $\tilde\psi\in C^\infty([0,T^*)\times\bar\Om)$ 
	\begin{align}
		&\int_\Om\xi(z(t,\cdot))\tilde\psi(t,\cdot)\,\dd x-\int_\Om\xi(z(s,\cdot))\tilde\psi(s,\cdot)\,\dd x-\int_s^t\!\int_\Om\xi(z)\partial_\tau\tilde\psi\,\dd x\dd\tau
		\\&\qquad=
		-\int_s^t\!\int_\Om D_{ij}\xi(z)A_{ik}(z)\nabla z_k\cdot\nabla z_j\tilde\psi \,\dd x\dd\tau  \label{eq:111}
		\\&\qquad\quad -\int_s^t\!\int_\Om D_i\xi(z)A_{ik}(z)\nabla z_k\cdot\nabla \tilde\psi\,\dd x\dd\tau
		+\int_s^t\!\int_\Om D_i\xi(z)R_i(z)\tilde\psi\,\dd x\dd\tau
	\end{align}
	for all $\xi\in C^\infty(\mathbb{R}_{\ge0}^{1+n})$.
	
	This can be proved as follows: take $\eta\in C^\infty(\mathbb{R})$, $\eta'\ge0$, $\eta=0$ on $(-\infty,-1]$, $\eta=1$ on $[0,\infty)$, and define $\eta_{s,\delta}(\tau):=\eta(\tfrac{\tau-s}{\delta})$ for $0<\delta\ll 1$.
	Observe that $\eta_{s,\delta}(\tau)=0$ for $\tau\le s-\delta$ and $\eta_{s,\delta}(\tau)=1$ for $\tau\ge s$.
	In the renormalised formulation~\eqref{eq:118} with $T=t$, we choose the test function $\psi(\tau,x):=\tilde\psi(\tau,x)\eta_{s,\delta}(\tau)$. 
	Then, the corresponding right-hand side of~\eqref{eq:118} converges, as $\delta\to0$, to the right-hand side of eq.~\eqref{eq:111} by the dominated convergence theorem.
	The corresponding left-hand side takes the form 
	\begin{align}\label{eq:112}
		\int_\Om\xi(z(t,\cdot))\tilde\psi(t,\cdot)\,\dd x
		-\int_{s-\delta}^s\int_\Om\xi(z)\tilde\psi\,\dd x\,\partial_\tau\eta_{s,\delta}\dd\tau+O(\delta)
		-\int_s^t\int_\Om\xi(z)\partial_\tau\tilde\psi\,\dd x\dd\tau.
	\end{align}
	The second term in the last line can be rewritten as $\left(F\ast(\tfrac{1}{\delta}\eta'(-\frac{\cdot}{\delta}))\right)(s)$ for the measurable, bounded function $F(\tau):=-\int_\Om\xi(z(\tau,x))\tilde\psi(\tau,x)\,\dd x$, and
	since $\int_{\mathbb{R}}\eta'(-\tau)\dd\tau=1$, we have the convergence $\left(F\ast(\tfrac{1}{\delta}\eta'(-\tfrac{\cdot}{\delta}))\right)(s)\to F(s)=-\int_\Om\xi(z(s,x))\tilde\psi(s,x)\,\dd x$ for a.e.~$s\in(0,T^*)$ as $\delta\to0$. This establishes the asserted identity~\eqref{eq:111}. 
	
	From~\eqref{eq:111} we infer that, for a.e.~$0<s<T^*$, the function $(\tau,x)\mapsto z(s+\tau,x)$ is a renormalised solution in $(0,T^*-s)\times\Omega$ with initial data $z(s,\cdot)$. Now we can invoke Case~1 and deduce~\eqref{eq:edin.s}.
\end{proof}

\section{Exponential convergence to equilibrium}\label{sec:exp.conv}

\begin{proof}[Proof of Proposition~\ref{prop:expconv}]
	Let us first observe that the regularity hypotheses on $z=(u,c)$ together with 
	the bounds \eqref{eq:edin.s} and~\eqref{eq:ene.s} imply that $\scrp(z)+|\nabla u|^2\in L^1(\Om_T)$ for any $T<\infty$ and that $u\in L^\infty_\loc([0,\infty),L^2(\Om))$, $c_i\log c_i\in L^\infty_\loc([0,\infty),L^1(\Om))$ for all $i\in\{1,\dots,n\}$. Here, we also used the lower and upper bounds on $H(z)$ provided in Lemma~\ref{l:S}.
	
	The energy  and mass conservation properties $\int_\Om z_l(t,x)\,\dd x=\bar z_l$ for a.e.~$t>0$,  where $l\in\{0,1,\dots,n\}$, can be seen as follows. 
	In the renormalised formulation~\eqref{eq:118}, we choose 
	$\psi\equiv1$ and
	$\xi(z):=\varphi_l^E(z)$ for $E\ge 1$,
	where $\varphi_l^E(z)=E\varphi_l(E^{-1}z)$ for some $\varphi_l\in C^\infty(\mathbb{R}_{\ge0}^{1+n})$ with $\supp D\varphi_l$ compact and
	$\varphi_l(z)=z_l$ for $|z|_1<1$ (see~\cite{FHKM_2020} for an example of such $\varphi_l$).
	This gives
	\begin{equation}\label{eq:218}
		\begin{split}
			\int_\Om\varphi_l^E(z(T,\cdot))\,\dd x-\int_\Om&\varphi_l^E(\zin)\,\dd x
			\\&\qquad=
			-\int_0^T\!\!\int_\Om D_{ij}\varphi_l^E(z)A_{ik}(z)\nabla z_k\cdot\nabla z_j\,\dd x\dd t.
		\end{split}
	\end{equation}
	By the dominated convergence theorem, the LHS converges,  as $E\to\infty$,  to 
	$$\int_\Om z_l(T,x)\,\dd x-\int_\Om \zin_l\,\dd x.$$ 
	Since $D^2\varphi_l(z)=0$ for $|z|_1<1$, we have $\lim_{E\to\infty}D_{ij}\varphi_l^E(z)= 0$ for every $z\in\mathbb{R}_{\ge0}^{1+n}$. At the same time, using the bounds in Lemma~\ref{l:modelBounds} it is easy to see that for all $i,j\in\{0,\dots,n\}$
	\begin{align}
		|D_{ij}\varphi_l^E(z)\sum_{k=0}^nA_{ik}(z)\nabla z_k\cdot\nabla z_j|\lesssim \scrp(z)+|\nabla u|^2
	\end{align}
	uniformly in $E\ge 1$. Hence, the integral on the RHS of~\eqref{eq:218} converges to zero as $E\to\infty$ thanks to dominated convergence. 
	
	\medskip
	Let us now sketch the proof showing the exponential convergence to equilibrium.
	Below, $\epsilon_k$, $k=0,1,\dots,$ denote fixed positive constants.
	As we only consider Model~\Mzero{}, it suffices to take $\alpha\in(0,1]$.
	Conservation of the total energy and the mass of each species, combined with $\bar z$ being spatially constant, yields for a.e.\ time
	\begin{align}
		H_\rel(z,\bar z) &= H(z)-D_ih(\bar z)\int_\Om(z_i-\bar z_i)\,\dd x-H(\bar z) 
		\\&=  H(z)-H(\bar z).
	\end{align}
	Hence,  inequality~\eqref{eq:edin.s} gives for a.e.~$t\ge s\ge0$
	\begin{align}
		\eval{H_\rel(z(\tau),\bar z)}_{\tau=s}^{\tau=t} \le-\int_s^t\!\int_\Om\scrp(z)\,\dd x\dd\tau.
	\end{align}
	Writing $B(c,\bar c):=\sum_{i=1}^n b(c_i,\bar c_i)$, where $b(s,\bar s)=\bar s\lambda(s/\bar s)$, we have 
	by estimate~\eqref{eq:700a} and the logarithmic Sobolev inequality (cf.~\cite{MM_2018})
	\begin{align}\label{eq:219}
		\int_\Om\scrp(z)\,\dd x\ge \epsilon_0 \sum_{i=1}^n\int_\Om|\nabla \sqrt{c_i}|^2\,\dd x \ge \epsilon_1 \int_\Om B(c,\bar c)\,\dd x,
	\end{align}
	where $\epsilon_1 =\epsilon_1(\Om)>0$.
	
	Recalling the definition of~$g_\rel$ in~\eqref{eq:defgrel}, and using $\int u\,\dd x=\int \bar u\,\dd x$,~\eqref{eq:ene.s}, and the Poincar\'e--Wirtinger inequality, we can further estimate for some $\epsilon_2=\epsilon_2(\Om)>0$
	\begin{align}\label{eq:220}
		\eval{G_\rel(u(\tau),\bar u)}_{\tau=s}^{\tau=t}  
		\le -\epsilon_2\int_s^t\!\int_\Om|u-\bar u|^2\,\dd x\dd\tau +\mOns C\int_s^t\!\int_\Om\scrp(z)\,\dd x\dd\tau.
	\end{align}
	We next let $E\ge |\bar z|$, to be fixed later. Then, by uniform convexity, for $|z|\le E$,
	\begin{align}
		B(c,\bar c)+\alpha|u-\bar u|^2\gtrsim_{E,\alpha} |z-\bar z|^2\gtrsim_{E} h_\rel(z,\bar z).
	\end{align}
	(The argument leading to the second inequality is as in the proof of~\eqref{eq:coercA}.)
	
	At the same time,
	\begin{align}
		B(c,\bar c)+\alpha|u-\bar u|^2\ge \tfrac{1}{2}\left(\sum_{i=1}^nc_i\log c_i +\alpha u^2\right)- C(\bar z).
	\end{align}
	Hence, for $E=E(\bar z,\alpha)$ large enough, we obtain 
	\begin{align}
		B(c,\bar c)+\alpha|u-\bar u|^2\gtrsim_{E} h_\rel(z,\bar z).
	\end{align}
	Combining the above estimates and choosing $\alpha\in(0,1]$ such that
	$\alpha C\mOns\le \tfrac{1}{2}$ (with $C$ as in~\eqref{eq:220}), we infer
	\begin{align}
		\eval{\bigg[H_\rel(z(\tau),\bar z)+\alpha G_\rel(u(\tau),\bar u)\bigg]}_{\tau=s}^{\tau=t} 
		&\le-\frac{1}{2}\int_s^t\!\int_\Om\scrp(z)\,\dd x\dd\tau-\epsilon_2\alpha\int_s^t\|u-\bar u\|^2_{L^2(\Om)}\dd\tau
		\\&\le -\epsilon_3\int_s^t\bigg[H_\rel(z(\tau),\bar z)+\alpha G_\rel(u(\tau),\bar u)\bigg]\dd\tau,
	\end{align}
	where $\epsilon_3=\epsilon_3(\bar z,\alpha,\Omega)>0$.
	A version of Gronwall's inequality (see e.g.~\cite[p.~702]{FL_2016}) yields the asserted bound~\eqref{eq:expconvEntropy} for $\lambda=\epsilon_3>0$.
\end{proof}

Let us note that in the above proof we have not used the fact that $\pi_1\gamma^2|\nabla u|^2$ is dominated by $\scrp(z)$ (see~\eqref{eq:219}). If $\sigma(u)$ is sufficiently close to a linear function for $u\gg1$, e.g.~$\sigma(u)=u^{1-\epsilon(d)}$ for $\epsilon(d)>0$ small enough, this term may be exploited as in~\cite[Section~3]{MM_2018} to quantitatively improve the decay rate.

\section{Appendix}

\subsection{Auxiliary estimates}\label{sec:app}

\begin{lemma}[Estimates for Model~\Mzero]\label{l:modelBounds}
	Let the hypotheses of Model~\Mzero\ be satisfied.
	Then, formally, 
	\begin{subequations}
		\begin{align}\label{eq:u.M0}
			&\sum_{k=0}^nA_{0k}(z)\nabla z_k = a(z)\nabla z+m(z)\nabla D_0h(z),
			\\&\label{eq:70}\sum_{k=0}^nA_{ik}(z)\nabla z_k =
			m_i(z) \nabla D_ih(z)
			+   a(z)  c_i \tfrac{w'_i(u)}{w_i(u)} \nabla u \quad\text{ for }i\ge1,
		\end{align}
	\end{subequations}
	where $a(z)=\pi_1(z)\gamma(z)$ and $\gamma$ is given by~\eqref{eq:def.gamma}.\\
	Moreover, for any sufficiently regular function $z=(u,c_1,\dots, c_n)$ with positive components, we have the following estimates:
	
	Abbreviating
	$$	\scrp(z):=\nabla z: (D^2h(z)A(z)\nabla z)
	=\nabla D_ih(z)\cdot (\mathbb{M}_{il}(z)\nabla D_lh(z))$$
	and $\gamma(u,c):=-\hat\sigma''(u)-\sum_{l=1}^n\tfrac{w_l''(u)}{w_l(u)}c_l$, 
	one has
	\begin{align}\label{eq:700a}
		\scrp(z)\gtrsim \sum_{i=1}^n|\nabla \sqrt{c_i}|^2+|\sqrt{\gamma}\,\nabla u|^2+|\sqrt{m}\nabla D_0h(z)|^2,
	\end{align}
	and 
	\begin{align}\label{eq:701a}
		a(z)|\nabla u|^2\sim |\nabla u|^2.
	\end{align}
	Furthermore,
	\begin{align}\label{eq:704a}
		&|A(z)\nabla z|\lesssim 
		\Big(\max_{i=1,\dots,n}\sqrt{c_i}+\sqrt{\pi_1(z)}+\sqrt{m(z)}\Big)\sqrt{\scrp(z)},
		\\\label{eq:702a}
		&|\sum_{k=0}^nA_{ik}(z)\nabla z_k|\lesssim \sqrt{c_i}\sqrt{\scrp(z)}\qquad\text{for }i\ge1,
		\\&|\sum_{k=0}^nA_{0k}(z)\nabla z_k|\lesssim |\nabla u|+\sqrt{m}\sqrt{\scrp(z)}.\label{eq:703a}
	\end{align}
\end{lemma}

\begin{proof}
	Identities~\eqref{eq:u.M0}--\eqref{eq:70} follow from a straightforward computation using the definition of $\mu_i$ (see~\cite{FHKM_2020}, if necessary).
	Except for the last term in~\eqref{eq:700a}, estimate~\eqref{eq:700a} is a consequence of~\cite[\EPestimate{}]{FHKM_2020}  (in~\cite{FHKM_2020}: $m\equiv0$). The additional control of $|\sqrt{m}\nabla D_0h(z)|^2$ for $m=m(z)\ge0$ easily follows from the definition of $\scrp(z)$ and $\mathbb{M}$. 
	Eq.~\eqref{eq:701a} is immediate since $a=\pi\gamma\sim 1$ by hypothesis.
	Estimate~\eqref{eq:704a} has been established in~\cite[\fluxbound{}, eq.~(2.10)]{FHKM_2020} for $m\equiv0$, and the current version thus follows estimate~\eqref{eq:700a}, which implies   the bound 
	$|m(z)\nabla D_0h(z)|\le \sqrt{m}\sqrt{\scrp(z)}.$
	Estimate~\eqref{eq:702a} is a consequence of the proof of~\cite[\fluxbound{}]{FHKM_2020}, 
	while estimate~\eqref{eq:703a} follows from the fact that~$a\sim 1$ and the bound on $m(z)\nabla D_0h(z)$ observed before.
\end{proof}

\begin{lemma}\label{l:applic.M0}
	Model~\Mzero{} (see page~\pageref{eq:wi.M0}) fulfils conditions~\ref{it:hC3}--\ref{eq:hp.A00} and~\ref{hp:grad.flux.control} of Theorem~\ref{thm:wkuniq} when assuming additionally the regularity hypotheses $\hat\sigma, w_i\in C^4((0,\infty))$, \ref{it:R.locLip} and $m,m_i,\pi_1\in C^{0,1}_\loc((0,\infty)^{1+n})$. 
\end{lemma}
\begin{proof}
	The asserted estimates can be verified using Lemma~\ref{l:modelBounds}:
	the bounds in~\ref{hp:trunc} are immediate consequences of estimates~\eqref{eq:700a},~\eqref{eq:704a} combined with the bound $0\le\sqrt{\pi_1(z)}\lesssim (1+u)$. Condition~\ref{eq:hp.A00} follows from estimating $|a(z)\nabla u|\lesssim\sqrt{\pi_1} |\sqrt{\gamma}\nabla u|$ and using~\eqref{eq:700a}. Condition~\ref{hp:grad.flux.control} easily follows from~\eqref{eq:702a} and~\eqref{eq:703a}. (We have not aimed at optimising the conditions on $m(z)$, which are far from being sharp.) 
	The conditions in~\ref{eq:hp.Mnondeg} follow from the definition of $\mathbb{M}$ in~\eqref{eq:mobility1}.
\end{proof}

\begin{lemma}[Lower and upper entropy bounds]\label{l:S} 
	Let $h=h(u,c)$ be given by~\eqref{eq:S0} with~\eqref{eq:h2} being satisfied.
	There exist positive constants $\epsilon_\beta>0$, $\kappa_\beta\in(0,1)$ and $C_\beta,C\in(0,\infty)$ such that for all $(u,c)\in[0,\infty)^{1+n}$
	\begin{align}\label{eq:112ptw}
		h(u,c)&\ge -\hat\sigma(u)+\epsilon_\beta\sum_{i=1}^n c_i\log(c_i)-Cu^{\kappa_\beta}-C_\beta,
		\\\label{eq:112ptwii} h(u,c)&\le -\hat\sigma(u)+C\sum_{i=1}^nc_i\log(c_i)+C.
	\end{align}
\end{lemma}
\begin{proof}
	Letting $\beta_*=\tfrac{\beta+1}{2}\in(\beta,1)$, we estimate using~\eqref{eq:h2}
	\begin{equation*}
		\begin{split}
			c_i\log(w_i(u)) &=
			c_i\log(w_i(u))\chi_{\{w_i(u)\le c_i^{\beta_*}\}}  + c_i\log(w_i(u))\chi_{\{w_i(u)> c_i^{\beta_*}\}} 
			\\&\le \beta_* c_i\log(c_i)+ w_i(u)^{\frac{1}{\beta_*}}\log(w_i(u))+C
			\\&\le \beta_* c_i\log(c_i)+C u^{\frac{1}{2}(1+\beta/\beta_*)}+C.
		\end{split}
	\end{equation*}
	Thus,
	\begin{equation*}
		\begin{split}
			h(u,c)&=-\hat\sigma(u)+\sum_{i=1}^n\big(\lambda(c_i)-c_i\log(w_i(u))\big)
			\\&\ge -\hat\sigma(u)+\sum_{i=1}^n\big((1-\beta_*)c_i\log(c_i)-c_i+1\big)-C u^{\frac{1}{2}(1+\beta/\beta_*)}-C.
		\end{split}
	\end{equation*}
	This yields~\eqref{eq:112ptw} with $\kappa_\beta:=\frac{1}{2}(1+\beta/\beta_*)<1$, $\epsilon_\beta=\tfrac{1}{2}(1-\beta_*)>0$ and a suitable constant $C_\beta<\infty$.
	
	Estimate~\eqref{eq:112ptwii} easily follows from the hypothesis that $w_i(0)>0$ for all $i$ (see also the proof of~\lowerBdS{} in~\cite{FHKM_2020}).
\end{proof}

\begin{lemma}[Minimum principle]\label{l:u.positivity}
	In addition to the hypotheses of Theorem~\ref{thm:wkuniq} assume that $\mOns=0$. 
	Let $T\in(0,T^*)$ and $\underline{u}:=\inf_{\Om_T} \uin>0$. 
	Then the energy component $u$ of the dissipative renormalised solution $z=(u,c)$ satisfies $u\ge \underline{u}$ almost everywhere in $\Om_T$.
\end{lemma}
\begin{proof}[Sketch proof]
	The hypotheses imply that $u\in L^\infty_\loc(I;L^2(\Om))$ and that there exists $r>1$ such that $a(z)\nabla u\in L^r_\loc(I;L^r(\Om))$, $\partial_tu\in L^r_\loc(I;(W^{1,r'}(\Om))^*)$, $\tfrac{1}{r}+\tfrac{1}{r'}=1$. The weak formulation of the energy component~\eqref{eq:u.weak} can therefore be integrated by parts with respect to time to give 
	\begin{align}\label{eq:u.weakParts}
		\int_0^{T'}\langle\partial_tu,\varphi\rangle\,\dd t
		= -\int_0^{T'}\!\!\int_\Om a(z)\nabla u\cdot\nabla\varphi\,\dd x\dd t.
	\end{align}
	for a.a.\ $T'\in(0,T]$.
	Ignoring regularity issues for the moment and testing the equation with $\varphi=(u-\underline{u})_-$ leads to
	\begin{align}\label{eq:minprFormal}
		\eval{\frac{1}{2}\int_\Om |(u-\underline{u})_-|^2\,\dd x}_{t=0}^{t={T'}}
		+\int_0^{T'}\!\!\int_\Om a(z)|\nabla(u-\underline{u})_-|^2\,\dd x\dd t&=0.
	\end{align}
	This implies that $(u-\underline{u})_-=0$ and hence $u\ge\underline u$ a.e.\ in $\Om_T$.
	
	To make the argument rigorous, one considers a smooth partition of unity $(\chi_k)_{k=1}^N$ on $\bar\Om$ as in the proof of the $L^2$ identities in~\cite[\ustrong]{FHKM_2020},  see also~\cite[Lemma~12]{CJ_2019_existence} and~\cite[Lemma~4]{Fischer_2015}. 
	For  simplicity, we only outline the reasoning in the case of $\psi:=\chi_k$ being compactly supported in $\Om$, and refer, for the general case, to the first and the third of the references provided before. 
	
	Denote by $\tilde\rho$ the standard mollifying kernel, let $\rho:=\tilde\rho\ast\tilde\rho$ and $\rho_\ve(x):=\tfrac{1}{\ve^d}\rho(\tfrac{x}{\ve})$.
	Then, for $\ve>0$ small enough (only depending on $\mathrm{dist}(\supp\psi,\partial\Om)$), 
	choose in~\eqref{eq:u.weakParts} the test function 
	$$\varphi=\rho_\ve\ast((\rho_\ve\ast u-\underline{u})_-\psi),$$
	which lies in $W^{1,r}_\loc(I;H^s(\Om))$ for any $s\in\mathbb{N}$. 
	Abbreviate $u_\ve=\rho_\ve\ast u$ and compute
	\begin{align}
		\partial_tu_\ve (u_\ve-\underline{u})_-\psi
		=\tfrac{1}{2}\tfrac{\dd}{\dd t}|(u_\ve-\underline{u})_-|^2\psi.
	\end{align}
	One the other hand, the term
	\begin{align}
		\int_0^{T'}\!\!\int_\Om \rho_\ve\ast(a(z)\nabla u)\cdot\nabla((u_\ve-\underline{u})_-\psi)\,\dd x\dd t
	\end{align}
	can be shown to converge to 
	\begin{align}
		\int_0^{T'}\!\!\int_\Om a(z)\nabla u\cdot\nabla((u-\underline{u})_-\psi)\,\dd x\dd t
	\end{align}
	by arguing similarly as in the proof of the $L^2$-energy identity in~\cite[\ustrong]{FHKM_2020}, see also~\cite[Lemma~12]{CJ_2019_existence}, where one should use the fact that $\|(u_\ve-\underline u)_-\|_{L^\infty}\le \underline u$, which follows from the non-negativity of $u$. 
	Thus, 
	\begin{align}
		\eval{\frac{1}{2}\int_\Om |(u-\underline{u})_-|^2\chi_k\,\dd x}_{t=0}^{t={T'}}+	\int_0^{T'}\!\!\int_\Om a(z)\nabla u\cdot\nabla((u-\underline{u})_-\chi_k)\,\dd x\dd t=0,
	\end{align}
	and upon summation over $k$ one arrives at~\eqref{eq:minprFormal}.
\end{proof}

\subsection{Entropy dissipation inequality for (M1.i) along the construction}\label{ssec:ED.ren}

The purpose of this paragraph is to show that the renormalised solutions constructed in~\cite[Theorem 1.8]{FHKM_2020} obey, for almost all $T>0$, the entropy dissipation inequality \eqref{eq:edin}.
We only provide the key step, which consists in taking the limit $\ve\to0$ in an entropy dissipation inequality analogous to \eqref{eq:edin} at the level of the approximate solutions $(u^\ve,c^\ve)$.  
Moreover, we focus on a $\liminf$-estimate for the diffusive entropy dissipation $\int_0^T\!\!\int_\Om \scrp(z)\,\dd x \dd t$, since it is primarily this quantity which takes a more involved form as compared to previous literature. The crucial point in handling this term is contained in the following lemma, which we state in a general form.

\begin{lemma}\label{l:ED.liminf}
	Let $\Om\subset\mathbb{R}^d$ be a bounded Lipschitz domain. 
	Assume that $N\in\mathbb{N}_+$ and let $T\in(0,\infty)$. Let further $D\subseteq [0,\infty)^N$ be a convex domain (possibly unbounded).
	Let $\mathbb{B}^\ve\in C(\ol D)^{N\times N}$, $\ve\in(0,1]$, be a family of  positive semi-definite symmetric matrices satisfying 
	\begin{align}
		\mathbb{B}^\ve\to \mathbb{B}\text{ locally uniformly in }\ol D
	\end{align}
	for some positive semi-definite symmetric matrix $\mathbb{B}\in C(\ol D)^{N\times N}$.
	Let $v^\ve=(v^\ve_1,\dots,v_N^\ve)$, $\ve\in(0,1]$, be a family of vector-valued functions
	such that  
	$v^\ve\in L^2(0,T;H^1(\Om)^N)$ and $v^\ve\in \ol D$ a.e.~in $(0,T)\times\Om$.
	Suppose that there exists $v\in L^2(0,T;H^1(\Om)^N)$ with $v\in \ol D$ a.e.~in $(0,T)\times\Om$
	such that
	\begin{align}
		&v^\ve\rightharpoonup v\quad\text{ in }L^2(0,T;H^1(\Om)^N),\label{eq:pp1}
		\\&v^\ve\to v\quad\text{ pointwise a.e.~in }(0,T)\times\Om,
	\end{align}
	and
	\begin{align}\label{eq:pp3}
		\sup_{\ve\in(0,1]}\|\sum_j\mathbb{S}^\ve_{ij}(v^\ve)\nabla v^\ve_j\|_{L^2((0,T)\times\Om)}^2\le C<\infty,
	\end{align}
	where~$\mathbb{S}^\ve(v)$ denotes the principal square root of~$\mathbb{B}^\ve(v)$.
	
	Then, as $\ve\to0$,
	\begin{align}
		\sum_j\mathbb{S}_{ij}^\ve(v^\ve)\nabla v^\ve_j
		\rightharpoonup \sum_j\mathbb{S}_{ij}(v)\nabla v_j
		\quad\text{ in }L^2((0,T)\times\Om)^d \quad\text{ for all }i\in\{1,\dots,N\},
	\end{align}
	where we have abbreviated $\mathbb{S}(v):=\mathbb{B}^\frac{1}{2}(v)$, the principal square root of $\mathbb{B}$.
	
	As a consequence,
	\begin{align}\label{eq:pp6}
		\|\sum_j\mathbb{S}_{ij}(v)\nabla v_j\|_{L^2((0,T)\times\Om)}^2\le \liminf_{\ve\to0}\|\sum_j\mathbb{S}_{ij}^\ve(v^\ve)\nabla v^\ve_j\|_{L^2((0,T)\times\Om)}^2
	\end{align}
	for all $i\in\{1,\dots,N\}$.
\end{lemma}

Notice that, using the notation of Lemma~\ref{l:ED.liminf},
\begin{align}
	\|\mathbb{S}(v)\nabla v\|_{L^2((0,T)\times\Om)}^2	=	\|\mathbb{B}^\frac{1}{2}(v)\nabla v\|_{L^2((0,T)\times\Om)}^2
	=\int_0^T\!\int_\Om \sum_{i,j}\nabla v_{i}\cdot\mathbb{B}_{ij}(v)\nabla v_j\,\dd x\dd t.
\end{align}

\begin{proof}Fix $i\in \{1,\dots, N\}$.
	Thanks to~\eqref{eq:pp3} there exists $X\in L^2((0,T)\times\Om)^d$ and a subsequence $\ve\to0$ such that 
	\begin{align}\label{eq:pp4}
		\sum_{j=1}^N\mathbb{S}_{ij}^\ve(v^\ve)\nabla v^\ve_j \rightharpoonup X\quad\text{ in }L^2((0,T)\times\Om)^d.
	\end{align}
	It suffices to show that $X=\sum_{j=1}^N\mathbb{S}_{ij}(v)\nabla v_j$
	(which implies that~\eqref{eq:pp4} is independent of the subsequence $\ve\to0$).
	To this end, first note that $\lim_{\ve\to0}\mathbb{B}^\ve(v^\ve)=\mathbb{B}(v)$ pointwise a.e.~in $(0,T)\times\Om$ and hence, using the continuity of the square root operator on positiv semi-definite symmetric matrices,
	\begin{align}\label{eq:pp2}
		\lim_{\ve\to0}\mathbb{S}^\ve(v^\ve)=\mathbb{S}(v)\text{ pointwise a.e.~in }(0,T)\times\Om.
	\end{align}
	Let now $\xi_K\in C^\infty_c([0,\infty),[0,1])$, $K\in\mathbb{N}$, be a sequence of non-increasing functions satisfying $\xi_K(s)=1$ for $s\le K$ and $\xi_K(s)=0$ for $s\ge2K$.
	Combining the weak convergence~\eqref{eq:pp1} with~\eqref{eq:pp2}, we infer for every $K\in\mathbb{N}$
	\begin{align}
		\xi_K\big(\sum_{\ell=1}^N|\mathbb{S}_{i\ell}^\ve(v^\ve)|\big)\sum_{j=1}^N\mathbb{S}_{ij}^\ve(v^\ve)\nabla v^\ve_j
		\rightharpoonup
		\xi_K\big(\sum_{\ell=1}^N|\mathbb{S}_{i\ell}(v)|\big)\sum_{j=1}^N\mathbb{S}_{ij}(v)\nabla v_j
		\quad\text{ in }L^2((0,T)\times\Om)^d.
	\end{align}
	On the other hand, the convergence~\eqref{eq:pp4} combined with~\eqref{eq:pp2} implies that 
	\begin{align}
		\xi_K\big(\sum_{\ell=1}^N|\mathbb{S}_{i\ell}^\ve(v^\ve)|\big)\sum_{j=1}^N\mathbb{S}_{ij}^\ve(v^\ve)\nabla v^\ve_j
		\rightharpoonup
		\xi_K\big(\sum_{\ell=1}^N|\mathbb{S}_{i\ell}(v)|\big)X\quad\text{ in }L^2((0,T)\times\Om)^d.
	\end{align}
	Therefore,
	\begin{align}\label{eq:pp5}
		\xi_K\big(\sum_{\ell=1}^N|\mathbb{S}_{i\ell}(v)|\big)(\sum_{j=1}^N\mathbb{S}_{ij}(v)\nabla v_j-X)=0 \quad
		\text{ a.e.~in }(0,T)\times\Om.
	\end{align}
	Since $|v|<\infty$ a.e.~in $(0,T)\times\Om$, we also have 
	$\sum_{\ell=1}^N|\mathbb{S}_{i\ell}(v)|<\infty$ a.e.~in $(0,T)\times\Om$.
	Recalling that $\xi_K(s)=1$ for $s\le K$ and that $K\in\mathbb{N}$ was arbitrary, we infer from~\eqref{eq:pp5} that $X=\sum_{j=1}^N\mathbb{S}_{ij}(v)\nabla v_j$ a.e.~in $(0,T)\times\Om$, which proves the assertion. 
	
	The inequality~\eqref{eq:pp6} is a consequence of the weak lower semi-continuity of the norm in $L^2((0,T)\times\Om)$.

\end{proof}

Let us explain how Lemma~\ref{l:ED.liminf} can be applied to the model considered in~\cite[Theorem~1.8]{FHKM_2020}, which corresponds to submodel~\ref{it:M1.ren} of~\Mone{} in the present manuscript.
Note that this model admits the minimum principle in Lemma~\ref{l:u.positivity}. Thus, assuming as it is implicitly done in Theorem~\ref{thm:wkuniq}, that $\uin\ge\underline u>0$,
the renormalised solution $(u,c_1,\dots,c_n)$ in~\cite[Theorem~1.8]{FHKM_2020} satisfies 
$u\ge\underline u$. 
We may therefore take the following choices in Lemma~\ref{l:ED.liminf}:

We let $N:=1+n$ (and allow the slight inconsistency that we now start with index zero, that is $v=(v_0,\dots,v_n)$). We then let $D:=(\underline{u},\infty)\times(0,\infty)^n,$ so that 
\begin{align}
	\ol D=[\underline{u},\infty)\times[0,\infty)^n.
\end{align}
Denoting $v=(u,\sqrt{c_1},\dots,\sqrt{c_n})$, the matrix $\mathbb{B}$ is defined as follows:
\begin{align}
	&\mathbb{B}_{ij}(v)=4\sqrt{c_i}M_{ij}(u,c)\sqrt{c_j}\quad\text{ if }i,j\in\{1,\dots,n\},
	\\&\mathbb{B}_{i0}(v)=\mathbb{B}_{0i}(u,c)=2\sqrt{c_i}M_{i0}(u,c) \;\;(=2\sqrt{c_i}M_{0i}(u,c) )\quad\text{ if }i\in\{1,\dots,n\},
	\\&\mathbb{B}_{00}(v)=M_{00}(u,c),
\end{align}
that is, $\mathbb{B}(v)=\diag(1,2\sqrt{c_1},\dots,2\sqrt{c_n})M(u,c)\diag(1,2\sqrt{c_1},\dots,2\sqrt{c_n})$.
Here, $M$ denotes the positive semi-definite symmetric matrix 
\begin{align*}
	M(u,c):=D^2h(u,c)\,\mathbb{M}(u,c)\,D^2h(u,c).
\end{align*} 
The matrices $\mathbb{B}^\ve$ are defined analogously, but with $\mathbb{M}$ 
replaced by the matrix $\mathbb{M}^\ve$ defined at the beginning of Section~5 in~\cite{FHKM_2020}, where one can also find the definition of the approximate solutions  (in~\cite{FHKM_2020} they are introduced as $Z=Z^{\ve,\varrho}$ and  later on denoted by $Z^\ve=(u^\ve,c^\ve)$ after choosing $\varrho=\ve$). In Lemma~\ref{l:ED.liminf}, we choose $v^\ve_0=u^\ve$ and $v^\ve_i=\sqrt{c^\ve_i}$ for $i=1,\dots,n$, and letting $Z=(u,c)$ be the limit along a subsequence as obtained in~\cite[Lemma~6.1]{FHKM_2020},
we take $v=(u,\sqrt{c_1},\dots,\sqrt{c_n})$ in Lemma~\ref{l:ED.liminf}. We note that \cite[Lemma~6.1]{FHKM_2020} guarantees the convergence properties of $v^\ve$ to $v$ required to apply Lemma~\ref{l:ED.liminf}.

With the help of formula (1.10) in~\cite{FHKM_2020} one may compute $\mathbb{B}$ and $\mathbb{B}^\ve$ explicitly to verify that they are indeed continuous on $\ol D$. The locally uniform convergence of $\mathbb{B}^\ve$ to $\mathbb{B}$ in $\ol D$ follows easily from the construction.

Finally, notice that with this choice of $\mathbb{B}$ and with $v=(u,\sqrt{c_1},\dots,\sqrt{c_n})$, the quantity 
$ \sum_{i,j}\nabla v_{i}\cdot\mathbb{B}_{ij}(v)\nabla v_j$ corresponds to 
$\scrp(u,c)= \begin{pmatrix} \nabla u \\\nabla c\end{pmatrix}^T\!\!\cdot M(u,c)\begin{pmatrix}\nabla u \\ \nabla c\end{pmatrix}$, so that~\eqref{eq:pp6} allows to infer the $\liminf$-estimate for the diffusive entropy dissipation in the context of~\cite[Theorem~1.8]{FHKM_2020}.

\section*{Acknowledgements}
The author would like to thank an anonymous referee for several interesting comments and for suggesting to introduce the notion of dissipative renormalised solutions in Definition~\ref{def:diss.renorm}.


\end{document}